\documentclass{amsart} 
\usepackage{amsxtra}
\usepackage{amssymb}
\usepackage{tikz-cd}
\usetikzlibrary{cd,shapes}
\usetikzlibrary{arrows}
\usepackage[mathscr]{eucal}

\input xy \xyoption{matrix} \xyoption{arrow}\xyoption{frame}

\theoremstyle{plain}
\tikzstyle{mutable}=[align=center, draw,inner sep=0pt,shape=ellipse,minimum height=0.4cm, minimum width=0.4cm]
\tikzstyle{frozen}=[inner sep=0.5mm,rectangle,draw]

\newtheorem{theorem}{Theorem}[section]
\newtheorem{lemma}[theorem]{Lemma}
\newtheorem{definition-theorem}[theorem]{Definition-Theorem}
\newtheorem{proposition}[theorem]{Proposition}
\newtheorem{corollary}[theorem]{Corollary}
\newtheorem{definition}[theorem]{Definition}
\newtheorem{remark}[theorem]{Remark}
\newtheorem{conjecture}[theorem]{Conjecture}
\newtheorem{notation}[theorem]{Notation}

\newtheorem*{maintheorem*}{Main Theorem}
\newtheorem*{theorem*}{Theorem}
\newtheorem*{theoremA*}{Theorem A}
\newtheorem*{theoremB*}{Theorem B}
\newtheorem{thmintro}{Theorem}

\theoremstyle{definition}

\newtheorem*{definition*}{Definition}
\newcommand \bth[1] { \begin{theorem}\label{t#1} }
\newcommand \ble[1] { \begin{lemma}\label{l#1} }

\newcommand \bpr[1] { \begin{proposition}\label{p#1} }
\newcommand \bco[1] { \begin{corollary}\label{c#1} }
\newcommand \bde[1] { \begin{definition}\label{d#1}\rm }
\newcommand \bex[1] { \begin{example}\label{e#1}\rm }
\newcommand \bre[1] { \begin{remark}\label{r#1}\rm }
\newcommand \bcj[1] { \begin{conjecture}\label{j#1}\rm }

\newcommand \bnota[1] { \begin{notation}\label{n#1}\rm }
\renewcommand {\eth} { \end{theorem} }
\newcommand {\ele} { \end{lemma} }

\newcommand {\epr} { \end{proposition} }
\newcommand {\eco} { \end{corollary} }
\newcommand {\ede} { \end{definition} }
\newcommand {\eex} { \end{example} }
\newcommand {\ere} { \end{remark} }
\newcommand {\ecj} { \end{conjecture} }

\newcommand {\enota} { \end{notation} }

\def \KK {{\mathbb K}}

\def \ring {{R}}

\def \AAA {\mathbb{ A}} 
\def \NewRing {{\mathbb A}'} 

\def \II {{\mathcal{I}}}

\def \al {\alpha}

\def \la {\lambda}

\def \De {\Delta}

\def \rcor {\rangle}
\def \lcor {\langle}

\def \g  {\mathfrak{g}}   

\def \h  {\mathfrak{h}}

\def \frb  {\mathfrak{b}}

\DeclareMathOperator \sign {{\mathrm{sgn}}}

\DeclareMathOperator \ad { {\mathrm{ad}} }

\usepackage[normalem]{ulem}
\usepackage{amsmath,amsthm,amsfonts,amssymb, tikz-cd} 
\usepackage{stackrel}

\usepackage{verbatim}      	
\usepackage{listings} 
\usepackage{mathtools}
\usepackage{epsfig}         	
\usepackage{url}		
\usepackage{epic,eepic,latexsym}
\usepackage{graphicx}
\usepackage{xcolor}
\usepackage{amsxtra}

\usepackage{lcd}
\usepackage{mathrsfs}
\usepackage[linktocpage]{hyperref}
\usepackage{lscape}
\usepackage{tabularx}
\usepackage[all]{xy}
\usepackage{enumitem}
\usepackage{stmaryrd}
\usepackage{tipa}
\usepackage{upgreek}

\usetikzlibrary {arrows.meta,positioning}
\usepackage{caption}
\usepackage{subcaption}

\DeclareMathOperator{\tw}{tw}

\newcommand{\BBB}{\mathbb{B}}

\DeclareMathOperator{\Mat}{Mat}

\DeclareMathOperator{\trop}{trop}

\DeclareMathOperator{\can}{can}
\DeclareMathOperator{\N}{N}

\DeclareMathOperator{\gr}{gr}

\def\!{\mskip-\thinmuskip}
\let\?=\overline
\let\:=\colon
\let\llb=\llbracket
\let\rrb=\rrbracket

\newcommand {\kk} {\Bbbk}

\newcommand {\sd} {{\bf s}}
\newcommand {\sdt} {{\bf t}}

\theoremstyle{plain}
 \ifdefined\thm
 \else
   \theoremstyle{plain}
    \newtheorem{thm}[theorem]{Theorem}
 \fi

 \ifdefined\prop
 \else
   \theoremstyle{plain}
    \newtheorem{prop}[theorem]{Proposition}
 \fi\ifdefined\cor
 \else
   \theoremstyle{plain}
    \newtheorem{cor}[theorem]{Corollary}
 \fi

 \ifdefined\lem
 \else
   \theoremstyle{plain}
    \newtheorem{lem}[theorem]{Lemma}
 \fi

 \ifdefined\def
 \else
   \theoremstyle{definition}
   \newtheorem{def}[theorem]{Definition}
 \fi
 
 \ifdefined\defn
 \else
   \theoremstyle{definition}
   \newtheorem{defn}[theorem]{Definition}
 \fi
 
 \ifdefined\example
 \else
   \theoremstyle{definition}
     \newtheorem{example}[theorem]{Example}
 \fi
 
 \ifdefined\eg
 \else
     \theoremstyle{definition}
         \newtheorem{eg}[theorem]{Example}
 \fi
 
 \ifdefined\rem
 \else
     \theoremstyle{remark}
         \newtheorem{rem}[theorem]{Remark}
 \fi
 
 \ifdefined\quest
 \else
     \theoremstyle{plain}
     \newtheorem{quest}[theorem]{Question}
 \fi
 
 \ifdefined\assumption*
 \else
     \theoremstyle{plain}
     \newtheorem*{assumption*}{Assumption}
 \fi

    \newtheorem{asm}[theorem]{Assumption}

\newcommand{\Q}{{\mathbb Q}}
\newcommand{\R}{{\mathbb R}}
\newcommand{\Z}{{\mathbb Z}}

\newcommand{\supp}{\operatorname{supp}}

\newcommand{\pr}{\operatorname{pr}} 

\usepackage{tikz-cd}
\usepackage{bm}

\newcounter{listequation}

\numberwithin{equation}{section}

\newcommand{\BZ}{\mathrm{BZ}}
\newcommand{\BFZ}{\mathrm{BFZ}}

\newcommand{\wt}{\mathrm{wt}}

\newcommand{\cF}{\mathcal{F}}
\newcommand{\Br}{\mathsf{Br}}

\newcommand{\tB}{\widetilde{B}}

\newcommand{\LP}{\mathcal{LP}}
\newcommand{\bLP}{\overline{\LP}}
\newcommand{\hLP}{\widehat{\LP}}

\renewcommand{\N}{\mathbb{N}}

\renewcommand{\can}{\mathbf{L}}

\newcommand{\ddiag}{\symm'}
\newcommand{\symm}{\mathsf{d}}
\newcommand{\bi}{{\bf i}}
\newcommand{\ubi}{\underline{\bf i}}

\newcommand{\seq}{\boldsymbol{\mu}}

\newcommand{\cM}{M^{\circ}}

\newcommand{\Hf}{\frac{1}{2}}
\newcommand{\hf}{\frac{1}{2}}

\newcommand{\ufv}{{\operatorname{uf}}}
\newcommand{\fv}{{\mathrm f}}

\newcommand{\clAlg}{\mathcal{A}}
\newcommand{\bClAlg}{\overline{\clAlg}}
\newcommand{\upClAlg}{\mathcal{U}}
\newcommand{\bUpClAlg}{\overline{\upClAlg}}
\newcommand{\alg}{\mathscr{A}}

\newcommand{\uk}{\underline{k}}

\renewcommand{\tw}{\operatorname{tw}}

\newcommand{\base}{{\AAA'}}

\ifdefined\supp
\else
\newcommand{\supp}{\operatorname{supp}}
\fi

\ifdefined\tw
\else
	\newcommand{\tw}{{\opname tw}}
\fi

\ifdefined\G
    \renewcommand{\G}{{\mathbb G}}
\else
   \newcommand{\G}{{\mathbb G}}
\fi

\ifdefined\C
    \renewcommand{\C}{{\mathbb C}}
\else
   \newcommand{\C}{{\mathbb C}}
\fi

\usepackage{comment}

\allowdisplaybreaks

\tikzset{Dynkin/.style={circle,draw,scale=.40}}

\author{Hironori Oya}
\address[Hironori Oya]{Department of Mathematics, Institute of Science Tokyo, 2-12-1 Ookayama, Meguro-ku, Tokyo, 152-8551, Japan}
\email{hoya@math.titech.ac.jp}

\author{Fan Qin}
\address[Fan Qin]{School of Mathematical Sciences \\ Beijing Normal University \\ China}
\email{qin.fan.math@gmail.com}

\author{Milen Yakimov}
\address[Milen Yakimov]{Department of Mathematics \\ Northeastern University \\ 360 Huntington Ave, Boston, MA 02115 \\ USA 
and International Center for Mathematical Sciences, Institute of Mathematics and Informatics \\
Bulgarian Academy of Sciences \\ 
Acad. G. Bonchev Str., Bl. 8 \\
Sofia 1113, Bulgaria
}
\email{m.yakimov@northeastern.edu}


\title[Integral quantum cluster structures]{Integral cluster structures on \\ quantized coordinate rings}

\begin{document}

\begin{abstract}
    We develop (quantum) cluster algebra structures over arbitrary commutative unital rings 
    $\kk$ and prove that the 
    (quantized) coordinate rings of connected simply-connected complex simple algebraic groups $G$ over $\kk$ admit such structures. We first show that the integral form of the quantized coordinate ring of $G$ admits an upper quantum cluster algebra structure over $\AAA=\Z[q^{\pm\frac{1}{2}}]$ by using a combination of tools from quantum groups, canonical bases and cluster algebras and a previous result of the second and third authors over $\Q(q^\Hf)$. We then obtain (integral) quantum versions of recent results of the first author: when $G$ is not of type $F_4$, the quantized coordinate ring of $G$ admits a quantum cluster algebra structure over $\AAA'$, where $\AAA'=\AAA$ when $G$ is not of types $G_2$, $E_8$, and $F_4$; $\AAA'=\AAA[(q^2+1)^{-1}]$ when $G$ is of type $G_2$, and $\AAA'=\Q(q^\Hf)$ when $G$ is of type $E_8$. We furthermore prove that the classical versions of these results hold over $\AAA'$ (where $\AAA'=\Z$ if $G$ is not of type $F_4$ or $G_2$ and $\AAA'=\Z[\frac{1}{2}]$ if $G$ is of type $G_2$) and that the integral form of the coordinate ring of $G$ of type $F_4$ is an upper cluster algebra. 
    Finally, by using common triangular bases of (quantum) cluster algebras, we prove that the above results also hold under specializations of $\AAA$ and $\AAA'$ to commutative unital rings $\kk$.
\end{abstract}

\maketitle
\section{Introduction}
\subsection{Background}
Cluster algebras were introduced by Fomin and Zelevinsky \cite{fomin2002cluster} in 2001, and within a very short period it was discovered that they provide powerful tools for many problems in mathematics and mathematical physics. Quantum cluster algebras were defined three years later by Berenstein and Zelevinsky \cite{BerensteinZelevinsky05}. One of the major motivations of Fomin and Zelevinsky for the introduction of cluster algebras was to provide a method for the systematic and detailed study of canonical bases of the integral forms of (quantized) coordinate rings. In order to realize this goal, one should first prove that the integral forms of the (quantized) coordinate rings in question admit (quantum) cluster algebra structures.

A major role in representation theory is played by the classical coordinate ring $\C[G]$ and the quantized coordinate ring $R_q[G]$ of a connected, simply-connected complex simple algebraic group $G$.
The aim of this paper is to substantially strengthen results on the construction of classical and quantum cluster algebra structures on them. The second and third named authors \cite{QY2025partially} proved that the quantized coordinate ring $R_q[G]$ over $\Q(q^{\Hf})$ is a partially compactified upper quantum cluster algebra, while the first author \cite{Oya} proved that, if $G$ is not of type $F_4$, the coordinate ring $\C[G]$ of $G$ over $\C$ is a partially compactified cluster algebra that coincides with the corresponding upper cluster algebra.
Following the terminology of \cite{gross2018canonical}, the term {\em{partially compactified cluster algebra}} refers to the algebra where the frozen variables are not inverted. Building on the results of \cite{QY2025partially} and incorporating many methods from quantum groups, canonical bases and cluster algebras, we prove the existence of classical and quantum partially compactified cluster algebra structures on the integral forms of $\C[G]$ and $R_q[G]$. We furthermore extend these results to arbitrary commutative unital rings $\kk$ by using triangular bases of (quantum) cluster algebras.

\subsection{Quantum and classical integral results}
Set 
\[
\AAA := \Z[q^{\pm \Hf}], \quad  
\AAA_{1/2}\coloneqq \AAA[[2]_q^{-1}]=\AAA[(q^2+1)^{-1}],
\quad \mbox{and} \quad \KK:= \Q(q^{\Hf}). 
\]
For a connected, simply-connected complex simple algebraic group $G$, denote by $R_q[G]_\AAA$ the integral form of $R_q[G]$ defined in \cite{Lus:intro,Lus-Zform}, see Section \ref{2.3} and Appendix \ref{app:QCA} for details. For an intermediate algebra $\AAA \subseteq \AAA' \subseteq \KK$, write $R_q[G]_{\AAA'} := R_q[G]_\AAA \otimes_\AAA \AAA'$ and, for a quantum seed $\sd$, denote by $\bClAlg(\sd)_{\AAA'} \subseteq \bUpClAlg(\sd)_{\AAA'}$ the corresponding partially compactified quantum cluster algebra and its upper counterpart, defined over $\AAA'$, see Section \ref{sec:cluster-alg} for details.

Berenstein and Zelevinsky \cite{BerensteinZelevinsky05} constructed a candidate of the quantum seed $\sd^\BZ$ of $R_q[G]$ associated
to every reduced expression of $(w_0, w_0) \in W \times W$, where $w_0$ is the longest element of the Weyl group $W$ of $G$, see Section \ref{sec:BZ-seed}; the necessary property that was not proved in \cite{BerensteinZelevinsky05} is that the skew field of fractions of the corresponding quantum torus is precisely the skew field of fractions of $R_q[G]$. For unshuffled words, this was proved in \cite{goodearl2016berenstein}. In \cite{QY2025partially} it was proved that those seeds are related to each other by mutations, which also showed the necessary property for all of them. 
Our first main result is the following:
\begin{thmintro}[{= Theorems \ref{thm:upcluster-G-U}, \ref{thm:cluster-G-A-U}, and \ref{thm:integral-cluster-G}}]\label{ti:cluster-G-A-U}
Let $G$ be a connected, simply-connected complex simple algebraic group and $\sd^\BZ$ be any Berenstein--Zelevinsky quantum seed of $R_q[G]$ for a reduced word of $(w_0,w_0) \in W \times W$. Then,
\[
R_q[G]_{\AAA}=\bUpClAlg(\sd^\BZ)_{\AAA}.
\]
Moreover, if $G$ is not of type $F_4$, then 
\[
R_q[G]=\bClAlg(\sd^\BZ),
\]
and this equality can be restricted to the equality on the integral forms as follows: 
\begin{enumerate} 
\item If $G$ is of type $A_r$, $B_r$, $C_r$, $D_r$, $E_6$, or $E_7$, then 
\[
R_q[G]_\AAA=\bClAlg(\sd^\BZ)_\AAA.
\]
\item If $G$ is of type $G_2$, then 
\[
R_q[G]_{\AAA_{1/2}}=\bClAlg(\sd^\BZ)_{\AAA_{1/2}}.
\]
\end{enumerate}
\end{thmintro}
We note that the theorem does not address the integral form of $R_q[G]$ in the $E_8$ case.

The algebras $\bClAlg(\sd^\BZ)_{\AAA}$ and $\bUpClAlg(\sd^\BZ)_{\AAA}$ can be straightforwardly defined as the quantum cluster algebras over $\AAA$, while the integral form $R_q[G]_\AAA$ (or $R_q[G]_{\AAA_{1/2}}$) of $R_q[G]$ arises from representation theory and is often challenging to study. Theorem \ref{ti:cluster-G-A-U} shows that, despite their very different constructions, these two integral forms actually coincide when $G$ is not of type $E_8$ or $F_4$. This provides supporting evidence for the idea that integral forms in cluster theory can serve as integral forms for the corresponding quantized coordinate rings in Lie theory, see \cite[Remark 1.4]{qin2023analogs} and \cite[Theorem A]{goodearl2020integral}.

The first ingredient of the proof of Theorem \ref{ti:cluster-G-A-U} is a bootstrap method to improve the result $R_q[G]_\KK=\bUpClAlg(\sd^\BZ)_\KK$ 
from \cite{QY2025partially}
to $\AAA$. This is achieved by a localization technique with respect to prime elements of noncommutative algebras. 

The second ingredient of the proof of Theorem \ref{ti:cluster-G-A-U} is the following result which might be important in its own right. This is a quantum analog and integral refinement of \cite[Theorem 3.2]{Oya}.
\begin{thmintro}[{= Theorem \ref{thm:integral-coord-ring}}]\label{ti:integral-coord-ring}
Assume $G$ is a connected, simply-connected complex simple algebraic group which is not of type $F_4$. We choose 
\[
\base=
\begin{cases}
    \KK&\text{when $G$ is of type $E_8$},\\
    \AAA_{1/2}=\AAA[(q^2+1)^{-1}]&\text{when $G$ is of type $G_2$},\\
    \AAA&\text{otherwise.}
\end{cases}
\]
Then the generalized quantum minors generate $R_q[G]_{\base}$ as an $\base$-algebra. 
\end{thmintro}

We also prove classical versions of Theorems \ref{ti:cluster-G-A-U} and \ref{ti:integral-coord-ring}. In the introduction we state the analog of Theorem \ref{ti:cluster-G-A-U}, which in the $E_8$ case is stronger than the quantum result. For the full statement of the analog of Theorem \ref{ti:integral-coord-ring}, we refer to Theorem \ref{thm:integral-coord-ring-classical}, where in the $E_8$ case we show that the generalized minors generate the integral form over $\Z$ of the coordinate ring of $G$. Denote by $R[G]_{\AAA'}$ the integral form of the coordinate ring of $G$ over an intermediate ring $\Z \subseteq \AAA' \subseteq \C$. Recall that Berenstein, Fomin and Zelevinsky \cite{BerensteinFominZelevinsky05} constructed a classical seed of the coordinate ring of $G$ over $\C$ associated to every reduced expression of $(w_0, w_0) \in W \times W$. Any such seed will be denoted by $\sd^\BFZ$ suppressing the dependence on reduced expression of $(w_0,w_0)$. 
\begin{thmintro}[{= Corollary \ref{cor:upcluster-G-U} and Theorem \ref{thm:integral-cluster-G-classical}}]\label{ti:integral-cluster-G-classical}
Let $G$ be a connected, simply-connected complex simple algebraic group. 
For every Berenstein--Fomin--Zelevinsky classical seed $\sd^\BFZ$ for $G$ associated with a reduced word of $(w_0,w_0) \in W \times W$, 
\[
R[G]_{\Z}=\bUpClAlg(\sd^\BFZ)_{\Z}.
\]
Moreover, if $G$ is not of type $F_4$, we have the following cluster algebra structures:  
\begin{enumerate} 
\item If $G$ is of type $A_r$, $B_r$, $C_r$, $D_r$, $E_6$, $E_7$, or $E_8$, then 
\[
R[G]_\Z=\bClAlg(\sd^\BFZ)_\Z.
\]
\item If $G$ is of type $G_2$, then 
\[
R_q[G]_{\Z\left[\frac{1}{2}\right]}=\bClAlg(\sd^\BFZ)_{\Z\left[\frac{1}{2}\right]}.
\]
\end{enumerate}
\end{thmintro}
\subsection{Results over arbitrary commutative unital rings}
In Section \ref{sec:cluster-alg}, we define versions of (quantum) cluster algebras and upper (quantum) cluster algebras over arbitrary commutative unital rings $\kk$ with any chosen quantum parameter $q^\Hf\in\kk^\times$ under a very mild condition (Assumption \ref{asm:bad_case}). The definitions are natural, but one has to be very careful about specializations.  To the best of our knowledge, this level of generality has not been established in the literature. Using common triangular bases of (quantum) cluster algebras associated with $\sd^\BZ$, we also prove that the corresponding partially 
compactified upper (quantum) cluster algebras over arbitrary commutative unital rings $\kk$ behave well under specialization under natural assumptions that are broadly satisfied, see Theorem \ref{thm:specialize_bUpClAlg}.

\begin{thmintro}[{= Corollary \ref{cor:upcluster-G-U}, Theorems \ref{thm:integral-cluster-G-k} and \ref{thm:integral-cluster-G-classical}}]\label{thm:D}
Let $\kk$ be a commutative unital ring and $q^\hf \in \kk^\times$ be a unit ($q^\hf =1$ in the classical case). If the assumptions of Theorem \ref{ti:cluster-G-A-U} or \ref{ti:integral-cluster-G-classical} are satisfied so that the (quantized) coordinate ring of $G$ over $\base$ is an (upper) cluster algebra over $\base$ and there is a specialization map from $\base$ to $\kk$ (see Definition \ref{def:specialization}), then the (quantized) coordinate ring of $G$ over $\kk$ is isomorphic to the corresponding partially compactified (quantum, upper) cluster algebra over $\kk$. 
\end{thmintro}

\subsection{Convention} For an integer $a$, denote, 
$[a]_+ = \max(a,0)$. All vectors are column vectors unless otherwise specified.

For two finite subsets $I$ and $J$ of $\Z$, and an $I\times J$ matrix $X$, denote by $X^T$ its transpose. For $I'\subset I$ and $J'\subset J$, $X_{I',J'}$ will denote the submatrix of $X$ with rows indexed by $I'$ and columns indexed by $J'$, respectively. We will view $\Z^{I'}$ as a subset of $\Z^I$, and use $\pr_{I'}$ to denote the natural projection from $\Z^I$ to $\Z^{I'}$.

\noindent
{\bf{Acknowledgments.}} The authors would like to thank Tsukasa Ishibashi and Wataru Yuasa for fruitful discussions on the calculation of the $G_2$ case. The authors are grateful to Yoshiyuki Kimura for valuable comments. 
The research of Hironori Oya was supported by JSPS KAKENHI Grant-in-Aid for Early-Career Scientists, Grant Number 23K12950. 
The research of Fan Qin was supported by the National Natural Science Foundation of China (Grant No.~12422102). The research of Milen Yakimov was supported by the Bulgarian Science Fund grant KP-06-N92/5 and the Ministry of Education and Science grant DO1-239/10.12.2024, the Simons Foundation grant SFI-MPS-T-Institutes-00007697 and the USA National Science Foundation grant DMS–2200762.
\section{Background on quantized coordinate rings}
\subsection{Complex simple Lie algebras and quantum groups}
\label{2.1}
Let $\g$ be a complex simple Lie algebra with Cartan matrix $(c_{ij})_{i,j=1}^r$ with symmetrizing integers $(d_i)_{i=1}^r$. We fix the labeling of the corresponding Dynkin diagram as displayed in Figure \ref{fig:placeholder}. This labeling is slightly different from \cite{Humphreys72,Kac90}, but it simplifies the expression in Eq. \eqref{eq:Mq}.

Let $\Pi= \{ \alpha_1, \ldots, \alpha_r \}$, $\Phi$, $\Phi^+$, $\{\alpha_i\spcheck 
\}$, $\{\varpi_i\}$ and $\{\varpi_i\spcheck\}$ be the sets of its simple roots, roots, positive roots,  simple coroots, fundamental weights, and fundamental coweights. Denote by $P$ and $Q$ the weight and root lattices of $\g$. Let $P^+$ be the set of dominant integral weights of $\g$ and
$Q^+:= \sum_i \Z_{\geq 0} \alpha_i$. Define a partial ordering on $P$ by $\mu\leq \lambda$ if and only if $\lambda-\mu\in Q^+$. We write $\mu <\lambda$ if $\mu\leq \lambda$ and $\mu\neq \lambda$. 

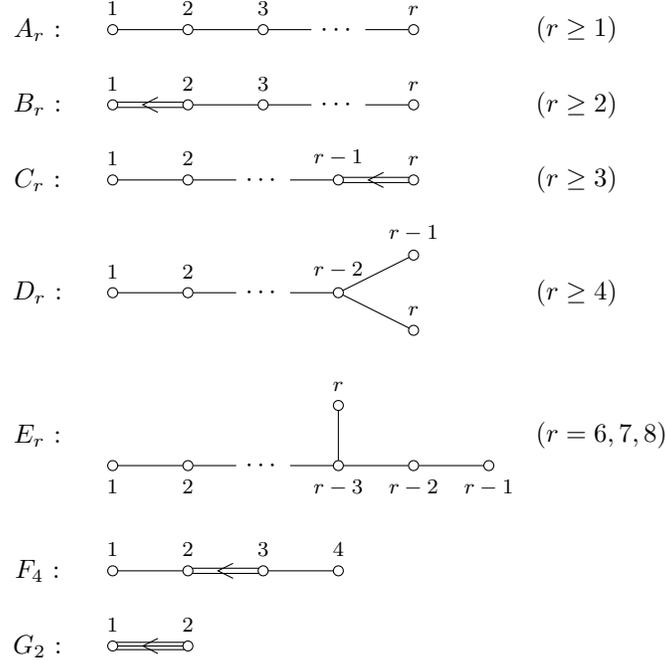
\begin{figure}
    \centering
\begin{tikzpicture}
\node at (-1,0) {$A_{r}$ : };
\node[Dynkin, label={above:\footnotesize$1$}] (A1) at (0,0) {};
\node[Dynkin, label={above:\footnotesize$2$}] (A2) at (1,0) {};
\node[Dynkin, label={above:\footnotesize$3$}] (A3) at (2,0) {};
\node (Ad) at (3,0) {$\cdots$};
\node[Dynkin, label={above:\footnotesize$r$}] (An) at (4,0) {};
\path[-]
 (A1) edge (A2)
 (A2) edge (A3)
 (A3) edge (Ad)
 (Ad) edge (An);
\node[right] at (5.5,0) {$(r\geq 1)$}; 
\node at (-1,-1) {$B_{r}$ : };
\node[Dynkin, label={above:\footnotesize$1$}] (B1) at (0,-1) {};
\node[Dynkin, label={above:\footnotesize$2$}] (B2) at (1,-1) {};
\node[Dynkin, label={above:\footnotesize$3$}] (B3) at (2,-1) {};
\node (Bd) at (3,-1) {$\cdots$};
\node[Dynkin, label={above:\footnotesize$r$}] (Bn) at (4,-1) {};
\draw[-] (B1.30) -- (B2.150);
\draw[-] (B1.330) -- (B2.210);
\draw[-] (0.6,-0.9) -- (0.4,-1) -- (0.6,-1.1);
\path[-]
 (B2) edge (B3)
 (B3) edge (Bd)
 (Bd) edge (Bn);
\node[right] at (5.5,-1) {$(r\geq 2)$}; 
\node at (-1,-2) {$C_{r}$ : };
\node[Dynkin, label={above:\footnotesize$1$}] (C1) at (0,-2) {};
\node[Dynkin, label={above:\footnotesize$2$}] (C2) at (1,-2) {};
\node (Cd) at (2,-2) {$\cdots$};
\node[Dynkin, label={above:\footnotesize$r-1$}] (Cm) at (3,-2) {};
\node[Dynkin, label={above:\footnotesize$r$}] (Cn) at (4,-2) {};
\draw[-] (Cm.30) -- (Cn.150);
\draw[-] (Cm.330) -- (Cn.210);
\draw[-] (3.6,-1.9) -- (3.4,-2) -- (3.6,-2.1);
\path[-]
 (C1) edge (C2)
 (C2) edge (Cd)
 (Cd) edge (Cm);
\node[right] at (5.5,-2) {$(r\geq 3)$}; 
 \node at (-1,-3.5) {$D_{r}$ : }; 
 \node[Dynkin, label={above:\footnotesize$1$}] (D1) at (0,-3.5) {};
\node[Dynkin, label={above:\footnotesize$2$}] (D2) at (1,-3.5) {};
\node (Dd) at (2,-3.5) {$\cdots$};
\node[Dynkin, label={above:\footnotesize$r-2$}] (Dn-2) at (3,-3.5) {};
\node[Dynkin, label={above:\footnotesize$r-1$}] (Dn-1) at (4,-3) {};
\node[Dynkin, label={above:\footnotesize$r$}] (Dn) at (4,-4) {};
\path[-]
 (D1) edge (D2)
 (D2) edge (Dd)
 (Dd) edge (Dn-2)
 (Dn-2) edge (Dn-1)
 (Dn-2) edge (Dn);
\node[right] at (5.5,-3.5) {$(r\geq 4)$}; 
\node at (-1,-5.4) {$E_{r}$ : }; 
 \node[Dynkin, label={below:\footnotesize$1$}] (E1) at (0,-5.8) {};
\node[Dynkin, label={below:\footnotesize$2$}] (E2) at (1,-5.8) {};
\node (Ed) at (2,-5.8) {$\cdots$};
\node[Dynkin, label={below:\footnotesize$r-3$}] (En-3) at (3,-5.8) {};
\node[Dynkin, label={below:\footnotesize$r-2$}] (En-2) at (4,-5.8) {};
\node[Dynkin, label={below:\footnotesize$r-1$}] (En-1) at (5,-5.8) {};
\node[Dynkin, label={above:\footnotesize$r$}] (En) at (3,-5) {};
\path[-]
 (E1) edge (E2)
 (E2) edge (Ed)
 (Ed) edge (En-3)
 (En-3) edge (En-2)
 (En-2) edge (En-1)
 (En-3) edge (En);
\node[right] at (5.5,-5.4) {$(r=6,7,8)$};
\node at (-1,-7.2) {$F_{4}$ : };
\node[Dynkin, label={above:\footnotesize$1$}] (F1) at (0,-7.2) {};
\node[Dynkin, label={above:\footnotesize$2$}] (F2) at (1,-7.2) {};
\node[Dynkin, label={above:\footnotesize$3$}] (F3) at (2,-7.2) {};
\node[Dynkin, label={above:\footnotesize$4$}] (F4) at (3,-7.2) {};
\draw[-] (F2.30) -- (F3.150);
\draw[-] (F2.330) -- (F3.210);
\draw[-] (1.6,-7.1) -- (1.4,-7.2) -- (1.6,-7.3);
\path[-]
 (F1) edge (F2)
 (F3) edge (F4);
 \node at (-1,-8.2) {$G_{2}$ : };
\node[Dynkin, label={above:\footnotesize$1$}] (G1) at (0,-8.2) {};
\node[Dynkin, label={above:\footnotesize$2$}] (G2) at (1,-8.2) {};
\draw[-] (G1.45) -- (G2.135);
\draw[-] (G1.315) -- (G2.225);
\draw[-] (G1.0) -- (G2.180);
\draw[-] (0.6,-8.1) -- (0.4,-8.2) -- (0.6,-8.3);
\end{tikzpicture}
    \caption{Dynkin diagrams}
    \label{fig:placeholder}
\end{figure}

Let $W$ be the Weyl group of $\g$, generated by the simple reflections $\{s_i \mid i\in [1, r]\}$. The identity element of $W$ is denoted by $e$. 
The Weyl group $W$ acts on $P$ by 
\[
s_i(\mu)=\mu-\langle \alpha_i\spcheck, \mu\rangle\alpha_i
\]
for $\mu\in P$ and $i\in [1, r]$. For $w\in W$, let
\[
\ell(w)\coloneqq\min \{\ell\in \Z_{\geq 0} \mid w=s_{i_1}\cdots s_{i_\ell}\},
\]
be the length of $w$. The unique longest element of $W$ is denoted by $w_0$. Denote by $(.,. )$ the $W$-invariant bilinear form on $\R \Pi$
normalized by $(\alpha_i, \alpha_i )= 2$ for short roots $\alpha_i$. 

The connected simply-connected algebraic group with Lie algebra $\g$ will be denoted by $G$. Denote by 
$\frb_\pm$ a pair of opposite Borel subalgebras and by $\h := \frb_+ \cap \frb_-$ the corresponding Cartan subalgebra. 
Let $B^\pm$ be the opposite Borel subgroups of $G$ with Lie algebras $\frb_\pm$ and $U^\pm$ be their unipotent radicals. 

Throughout the paper we will denote 
\[
\KK := \Q(q^{\Hf}), 
\quad \mbox{and respectively},   
\quad \AAA := \Z[q^{\pm \Hf}].
\]
In the following, we simply write $\otimes_{\KK}$ as $\otimes$. Set 
	\begin{align*}
	    &q_{i}:=q^{d_i}\text{ for }i\in [1,r],
        \qquad[n]_q:=\frac{q^{n}-q^{-n}}{q-q^{-1}}\ \text{ for $n\in\Z$},\\
		&[n]_q!:=[n]_q[n-1]_q\cdots[1]_q\text{\ for $n\in\Z_{>0}$},\ [0]_q!:=1,\\
		&{\displaystyle \left[\begin{array}{c}
				n\\
				k
			\end{array}\right]_q:=\frac{[n]_q!}{[k]_q![n-k]_q!}\ \text{ for $n, k\in\Z$ with $n\geq k\geq 0$}}.
	\end{align*}
	For a rational function $F\in\KK$, we define $F_{q_i}$ as
	the rational function obtained from $F$ by substituting $q$ by $q_{i}$	($i\in [1,r]$).

The quantized enveloping algebra $U_q(\g)$ associated with $\g$ is the unital associative $\KK$-algebra defined by the
	generators 
	\[
	\{X_{i}^+,X_{i}^-, K_i, K_i^{-1}\mid i\in [1, r]\}
	\]
	and the relations (i)--(iv) below: 
	\begin{enumerate}
		\item[(i)] $K_iK_i^{-1}=1=K_i^{-1}K_i,\;K_iK_j=K_jK_i$ for $i, j\in [1,r]$,
		\item[(ii)] $K_iX_{j}^+=q_i^{c_{ij}}X_{j}^+K_i,\;K_iX_{j}^-=q_i^{-c_{ij}}X_{j}^-K_i$ for $i, j\in [1,r]$,
		\item[(iii)] ${\displaystyle \left[X_{i}^+,X_{j}^-\right]=\delta_{ij}\frac{K_{i}-K_{i}^{-1}}{q_{i}-q_{i}^{-1}}}$ for $i, j\in [1,r]$,
		\item[(iv)] ${\displaystyle \sum_{k=0}^{1-c_{ij}}(-1)^{k}\left[\begin{array}{c}
				1-c_{ij}\\
				k
			\end{array}\right]_{q_i}(X_{i}^{\epsilon})^{k}X_{j}^{\epsilon}(X_{i}^{\epsilon})^{1-c_{ij}-k}=0}$ for $i, j\in [1,r]$ with $i\neq j$, and $\epsilon=+,-$. 
	\end{enumerate}
    
The algebra $U_q(\g)$ is a Hopf algebra with coproduct 
\[
\Delta(K_i) = K_i \otimes K_i, \enspace 
\Delta(X_i^+) = X_i^+ \otimes 1 + K_i \otimes X_i^+, \enspace 
\Delta(X_i^-) = X_i^- \otimes K_i^{-1} + 1 \otimes X_i^-.
\]
and with counit 
\[
\varepsilon(K_i)=1,\enspace \varepsilon(X_i^+)=\varepsilon(X_i^-)=0.
\]
Its antipode $S$ is defined by 
\[
S(K_i) = K_i^{-1},\quad S(X_i^+) = - K_i^{-1} X_i^+,\quad 
S(X_i^-) = - X_i^- K_i.
\]
Denote by $U_q^\pm(\g)$ and $U_q^0(\g)$ the 
subalgebras of $U_q(\g)$ generated by $\{X^\pm_i \mid 
i \in [1,r] \}$ and $\{K_i^{\pm1} \mid i \in [1,r] \}$, respectively. Set
\[
U^{\geq}_q(\g) := U_q^0(\g) U_q^+(\g) \quad \mbox{and} \quad U^{\leq}_q(\g) := U_q^-(\g)U_q^0(\g) .
\]
The algebra $U_q(\g)$ is $Q$-graded 
by setting $\wt X_i^\pm := \pm \al_i$ and $\wt K_i^{\pm 1} := 0$. The component of weight $\gamma \in Q$ of a graded subalgebra $A$ of $U_q(\g)$ will be denoted by $A_\gamma$. 

Denote by $U_q^-(\g)_\AAA$ the divided power integral form of $U_q^-(\g)$, i.e., the $\AAA$-subalgebra generated by 
\[
X_i^{-(k)} := \frac{(X_i^-)^k}{[k]_{q_i}!} \quad 
\mbox{for} \quad i \in [1,r], k \in \Z_{\geq 0}.
\]
The integral form $U_q(\g)_\AAA$ is the $\AAA$-Hopf 
subalgebra of $U_q(\g)$ generated by the elements $(X_i^\pm)^{(k)}$ and $K_i^{\pm1}$ for $i \in [1,r], k \in \Z_{\geq 0}$.

\subsection{Modules over quantum groups}
\label{2.2}
In this paper, $U_q(\g)$-modules (resp.~$\g$-modules) mean left $U_q(\g)$-modules (resp.~$\g$-modules). For a $U_q(\g)$-module (resp.~$\g$-module) $V$, the dual space $V^{\ast}$ is regarded as a $U_q(\g)$-module (resp.~$\g$-module) by 
\[
\langle x.\xi, v\rangle:= \langle \xi, S(x).v\rangle\quad (\text{resp.~}\langle X.\xi, v\rangle:= -\langle \xi, X.v\rangle)
\]
for $\xi\in V^{\ast}, x\in U_q(\g)$ (resp.~$X\in \g$) and $v\in V$. 
\begin{rem}\label{r:dual}
Let $V_1, V_2$ be finite dimensional $U_q(\g)$-module. Then we have an isomorphism of $U_q(\g)$-modules
\[
V_2^{\ast}\otimes V_1^{\ast}\xrightarrow[]{\sim} (V_1\otimes V_2)^{\ast},\ \xi_2\otimes \xi_1\mapsto (v_1\otimes v_2\mapsto \xi_1(v_1)\xi_2(v_2)). 
\]
We always identify these two spaces by this isomorphism. In particular, if $\{v^{(i)}_j\}_j$ is a basis of $V_i$ and $\{\xi^{(i)}_j\}_j$ is its dual basis of $V_i^{\ast}$ $(i=1,2)$, then we have 
\[
\langle \xi^{(2)}_{k_2}\otimes \xi^{(1)}_{k_1}, v^{(1)}_{l_1}\otimes v^{(2)}_{l_2}\rangle=\delta_{k_1, l_1} \delta_{k_2, l_2}.
\] 
\end{rem}
The {\em{weight spaces}} of a $U_q(\g)$-module $V$ are given by
\[
V_\nu := \{ v \in V \mid K_i. v = q^{ ( \nu, \al_i )} v, \enspace 
\forall\,  i \in [1,r] \}, \quad \nu \in P.
\]

The {\em{classical weight spaces}} of a finite dimensional $\g$-module $V$ are denoted by 
\[
V_\nu := \{ v \in V \mid h. v = \lcor \nu, h \rcor v, \enspace 
\forall\,  h \in \h \}, \quad \nu \in P.
\]

For $v\in V_{\mu}$, we write $\wt v=\nu$. A $U_q(\g)$-module $V$ is called a \emph{type one} module if it is a direct sum of its weight spaces. The finite dimensional type one $U_q(\g)$-modules form a semisimple braided tensor category. 

For a dominant integral weight $\la \in P^+$, 
denote by $V(\la)$ and $V_q(\la)$ the finite dimensional irreducible highest 
weight $\g$-module and type one $U_q(\g)$-module
with highest weights $\lambda$, respectively. 

Fix a highest weight vector $v_\la \in V_q(\la)$. The {\em{integral form}} of the module $V_q(\lambda)$ is the $U_q(\g)_\AAA$-module
\[
V_q(\lambda)_\AAA := U_q^-(\g)_\AAA v_\lambda. 
\]
Its weight spaces are 
\[
\big( V_q(\lambda)_\AAA \big)_\nu := V_q(\lambda)_\nu 
\cap V_q(\lambda)_\AAA. 
\]
Its {\em{dual integral form}} is defined by 
\[
V_q(\lambda)\spcheck_\AAA := \{ \xi \in V(\lambda)^* \mid 
\langle \xi, V_q(\lambda)_\AAA \rangle \subseteq \AAA\}
\]
having the weight spaces
\[
\big( V_q(\lambda)\spcheck_\AAA \big)_\nu := (V_q(\lambda)^*)_\nu 
\cap V_q(\lambda)\spcheck_\AAA.
\] 

Denote by $\theta$ the highest root of $\g$. The {\em{classical adjoint representation}} $\ad$ of $\g$ is isomorphic 
to $V(\theta)$. The representation $V_q(\theta)$ is called the {\em{quantum adjoint representation}} \cite{Lusztig-qadjoint}. 

Let $\{T_i \mid i \in [1,r] \}$ be the generators of the braid group $\Br$ associated with $W$. For $w \in W$, set $T_w := T_{i_1} \cdots T_{i_N}$ for any reduced expression $s_{i_1} \cdots s_{i_N}$ of $w$ (independent on the choice of reduced expression). The Lusztig braid group actions \cite{Lus:intro} on $U_q(\g)$ and on finite dimensional type one $U_q(\g)$-modules will be denoted by the same symbols; those actions are denoted by $T''_{i, 1}$ in \cite{Lus:intro}. 

The {\em{extremal weight vectors}}, 
\[
v_{w \lambda} := (T_{w^{-1}})^{-1} v_\lambda \in \big(V_q(\lambda)_\AAA\big)_{w\lambda}, \quad 
\forall \lambda \in P^+, w \in W,
\]
only depend on $w \la$, \cite[Lemma 39.1.2, Proposition 39.3.7]{Lus:intro} and not on the choice of $w$. Taking into account that 
the extremal weight spaces $V(\lambda)_{w\lambda}$ are one dimensional, let
\[
\xi_{w \lambda} \in \big(V_q(\lambda)\spcheck_\AAA\big)_{- w\lambda}
\]
be the unique vector such that
$\lcor \xi_{w \lambda}, v_{w\lambda} \rcor =1$.

It is well known that classical and quantum multiplicities are equal:
\begin{equation}
    \label{mult}
    [V(\lambda) \otimes V(\lambda') : V(\mu)]=
    [V_q(\lambda) \otimes V_q(\lambda') : V_q(\mu)], 
    \quad \forall \lambda, \lambda', \mu \in P^+.
\end{equation}
\subsection{Quantized coordinate rings} 
\label{2.3}
Let $V$ be a finite dimensional $U_q(\g)$-module. 
The {\em{matrix coefficient}} associated to a pair of vectors $v \in V$ and $\xi \in V^*$ is denoted by  
\begin{equation} 
\label{c-notation}
c^V(\xi, v) \in U_q(\g)^*, \quad 
\mbox{where} \quad c^V(\xi, v)(x)\coloneqq  \langle \xi, x. v \rangle,  \quad \forall x \in U_q(\g). 
\end{equation}
When the representation $V$ is clear from the context, we will write $c(\xi, v)$ instead of $c^V(\xi, v)$ for short. 

The {\em{quantized coordinate ring $R_q[G]$ of $G$}} is the subspace of $U_q(\g)^*$ spanned by  
\[
\{ c(\xi, v) \mid v \in V_q(\lambda), \xi \in V_q(\lambda)^* \}. 
\]
The quantized coordinate ring $R_q[G]$ has a Hopf algebra structure, where its product and coproduct are induced from the coproduct and the product of $U_q(\g)$, respectively. Moreover, there is an isomorphism of $U_q(\g)\otimes U_q(\g)$-modules 
\[
\bigoplus_{\lambda\in P^+} V(\lambda)^{\ast}\boxtimes V(\lambda)\xrightarrow{\sim} R_q[G],\ \xi\otimes v\mapsto c(\xi, v),  
\]
where $V(\lambda)^{\ast}\boxtimes V(\lambda)$ is the tensor product $V(\lambda)^{\ast}\otimes V(\lambda)$ with the $U_q(\g)\otimes U_q(\g)$-module structure
\[
(x_1\otimes x_2).(\xi\otimes v)=(x_1.\xi)\otimes (x_2.v),\quad \forall \xi\in V(\lambda)^{\ast},\ \forall v\in V(\lambda),\ \forall x_1, x_2\in U_q(\g),
\]
and the $U_q(\g)\otimes U_q(\g)$-module structure on $R_q[G]$ is given by 
\[
\langle (x_1\otimes x_2).\phi, y\rangle=\langle \phi, S(x_1)yx_2\rangle,\quad \forall \phi\in R_q[G],\ \forall x_1, x_2, y\in U_q(\g).
\]
See \cite[Definition 7.2.1, Proposition 7.2.2]{Kashiwara93}. Let $V_1$ and $V_2$ be $U_q(\g)$-modules. Then, for any $\xi_i\in V_i^{\ast}$ and $v_i\in V_i$ $(i=1,2)$, we have 
\begin{equation}
c^{V_1}(\xi_1, v_1)c^{V_2}(\xi_2, v_2)=c^{V_1\otimes V_2}(\xi_2\otimes \xi_1, v_1\otimes v_2).\label{eq:product}
\end{equation}
Here $\xi_2\otimes \xi_1\in V_2^{\ast}\otimes V_1^{\ast}$ is considered as an element of $(V_1\otimes V_2)^{\ast}$. Recall Remark \ref{r:dual}.

The integral form of $R_q[G]$ is defined as 
\begin{align*}
    R_q[G]_\AAA = \{ \phi\in R_q[G] \mid \langle \phi, U_q(\g)_\AAA\rangle \subset \AAA \}.
\end{align*}
It is a Hopf algebra over $\AAA$. Let $\kk$ be a commutative unital ring with a unital ring homomorphism $\alpha\colon \AAA\to \kk$. Then $\kk$ is regarded as a $\AAA$-module via $\alpha$, and a Hopf algebra $R_q[G]_{\kk}$ over $\kk$ is defined as 
\begin{equation}
\label{AAA'}
R_q[G]_{\kk} := R_q[G]_\AAA \otimes_{\AAA} \kk. 
\end{equation}
Since $R_q[G]_\AAA$ is free over $\AAA$, $R_q[G]_{\kk}$ is free over $\kk$. See Appendix \ref{app:QCA} and \cite[Section 3]{Lus-Zform} for more details. In this paper, we adopt the following conventions unless otherwise specified. 
\begin{itemize}
    \item When $\kk$ contains $\AAA$ (e.g.~$\kk=\KK$), we take $\alpha$ as an inclusion map. 
    \item When $\kk$ is a subring of $\C$ (e.g.~$\kk=\Z$), we take $\alpha$ as a ring homomorphism defined by $\alpha(q^{\frac{1}{2}})=1$. This case is referred to as the classical case, and we write $R_q[G]_{\kk}$ as $R[G]_{\kk}$ in this case. 
\end{itemize}
Note that $R_q[G]_{\KK}=R_q[G]$.
It is shown in \cite[Section 4]{Lus-Zform} that $R[G]_{\C}$ is isomorphic to the coordinate ring $\C[G]$ of $G$ over $\C$. 

For $\lambda\in P^+$, denote by $\BBB(\lambda)$ the canonical basis of $V_q(\lambda)$ in the sense of \cite{Lus:intro}. The dual basis of $\BBB(\lambda)$ in $V_q(\lambda)^{\ast}$ is denoted by $\BBB(\lambda)^*\coloneqq \{b^{\ast}\mid b\in \BBB(\lambda)\}$, where $\langle b^{\ast}, b'\rangle=\delta_{b, b'}$ for $b, b'\in \BBB(\lambda)$. We have the following properties:
\begin{itemize}
    \item[(i)] $\BBB(\lambda)$ is an $\AAA$-basis of $V_q(\lambda)_{\AAA}$, and $\BBB(\lambda)^{\ast}$ is an $\AAA$-basis of $V_q(\lambda)\spcheck_{\AAA}$.
    \item[(ii)] $v_{w\lambda}\in \BBB(\lambda)$ and $\xi_{w\lambda}\in \BBB(\lambda)^{\ast}$ for all $w\in W$ and $\lambda\in P^+$.
    \item[(iii)] $c^{V_q(\lambda)}(b^{\ast}, b')\in R_q[G]_\AAA$ for all $\lambda\in P^+$ and $b, b'\in \BBB(\lambda)$.
\end{itemize}

\begin{proposition}[{\cite[Proposition 3.3]{Lus-Zform}}]\label{p:Lusgen}
    The $\AAA$-algebra $R_q[G]_\AAA$ is generated by 
    \[
    \widetilde{\Gamma}\coloneqq \{c^{V_q(\varpi_i)}(b^{\ast}, v_{w_0\varpi_i}), c^{V_q(\varpi_i)}(b^{\ast}, v_{\varpi_i})\mid i\in [1,r], b\in \BBB(\varpi_i)\}.
    \]
\end{proposition}
 
The {\em{generalized quantum minors}} of $G$ are the elements 
\[
\De_{w \lambda, u \lambda} := c^{V_q(\lambda)}(\xi_{w\lambda}, v_{u\lambda}) \in R_q[G]_\AAA, 
\quad \forall \lambda \in P^+, w, u \in W.
\]
We have \cite[Eq. (9.13)]{BerensteinZelevinsky05}
\begin{equation}
    \label{DeDe}
    \Delta_{w \lambda, u \lambda} \Delta_{w \mu, u \mu} = \Delta_{w(\lambda + \mu), u(\lambda + \mu)}, 
    \quad \forall \lambda, \mu \in P^+.
\end{equation}

The following property is stated for $R_q[G]$ in \cite[Lemma 3.1]{joseph-paper}, \cite[Lemma 9.1.9]{joseph2012quantum}, and for $R_q[G]_{\kk}$ with $\alpha(q^{\frac{1}{2}})=1$ in \cite[Theorem 3.15]{Lus-Zform}. 
\begin{lem}\label{l:domain}
Let $\kk$ be an integral domain and $\alpha\colon \AAA\to \kk$ a unital ring homomorphism. Then $R_q[G]_{\kk}$ is a domain, i.e., $0$ is the only zero divisor of $R_q[G]_{\kk}$. 
\end{lem}
\begin{proof}
Let $\pi_{+}\colon U_q(\g)^{\ast}\to U^{\geq}_q(\g)^{\ast}$ and $\pi_{-}\colon U_q(\g)^{\ast}\to U^{\leq}_q(\g)^{\ast}$ be the restriction maps. Note that these are $\KK$-algebra homomorphisms with respect to the algebra structure induced from the coproduct $\Delta$. Set $R_q[B^{\pm}]_\AAA:=\pi_{\pm}(R_q[G]_\AAA)$ and $R_q[B^{\pm}]_{\kk}:=R_q[B^{\pm}]_\AAA \otimes_{\AAA} \kk$. 

For $\mu\in P$, define  $c(\mu)\in U^{\geq}_q(\g)^{\ast}$ as 
\[
\langle c(\mu), xK_{\alpha}\rangle =\varepsilon(x)q^{(\mu, \alpha)}
\]
for $x\in U^+_q(\g)$ and $\alpha\in Q$, where $K_{\alpha}:=\prod_{i\in [1, r]}K_i^{m_i}$ for $\alpha=\sum_{i\in [1, r]}m_i\alpha_i$. Let $\BBB(\infty)$ be the canonical basis of $U^-_q(\g)$ in the sense of \cite{Lus:intro}. For $\widetilde{b}\in \BBB(\infty)$, define $c(\widetilde{b})\in U^{\geq}_q(\g)^{\ast}$ by 
\[
\langle c(\widetilde{b}), \omega(\widetilde{b}')K_{\alpha}\rangle =\delta_{\widetilde{b}, \widetilde{b}'}
\]
for $\widetilde{b}'\in \BBB(\infty)$ and $\alpha\in Q$. Here $\omega$ is the algebra involution on $U_q(\g)$ defined by $\omega(X_i^-)=X_i^+$ and $\omega(K_i)=K_i^{-1}$ for $i\in [1, r]$. 
We have $c(\mu), c(\widetilde{b})\in R_q[B^+]_\AAA$ for $\mu\in P$ and $\widetilde{b}\in \BBB(\infty)$. We can directly show that $c(\mu)=\pi_{+}(c(b^{\ast}, b))$ for some $b\in \BBB(\lambda)$ with $\wt b=\mu$, and $c(\widetilde{b})=\pi_{+}(c(b^{\ast}, v_{w_0\lambda}))c(-w_0\lambda)$ for some $b\in \BBB(\lambda)$ with $\omega(\widetilde{b}).v_{w_0\lambda}=b$, because 
\[
\{\omega(\widetilde{b}).v_{w_0\lambda}\mid \widetilde{b}\in \BBB(\infty)\text{ such that  }\omega(\widetilde{b}).v_{w_0\lambda}\neq 0\}=\BBB(\lambda)
\]
for $\lambda\in P^+$ \cite[Theorem 14.4.11, Proposition 21.1.2]{Lus:intro}. Then, $\{c(\widetilde{b})c(\mu)\mid \mu\in P, \widetilde{b}\in \BBB(\infty)\}$ is an $\AAA$-basis of $R_q[B^+]_\AAA$. 

By \cite[Lemma 3.7]{Lus-Zform}, there is an injective $\kk$-algebra homomorphism $R_q[G]_{\kk}\to R_q[B^-]_{\kk}\otimes_{\kk}R_q[B^+]_{\kk}$. Hence it suffices to show that $R_q[B^-]_{\kk}\otimes_{\kk}R_q[B^+]_{\kk}$ is a domain. 
Note that $R_q[B^{\pm}]_{\kk}$ is isomorphic to $\mathbf{O}^{\pm}_\kk$ in \cite[Section 3.5]{Lus-Zform}, where $(\widetilde{b}^+1_{\mu})^{\ast}$ in $\mathbf{O}^+_\kk$ corresponds to $c(\widetilde{b})c(\mu)$ in $R_q[B^+]_{\kk}$ for $\mu\in P, \widetilde{b}\in \BBB(\infty)$, and similar for $\mathbf{O}^-_\kk$ and $R_q[B^-]_{\kk}$ (cf.~Appendix \ref{app:QCA}). 

In $R_q[B^+]_{\AAA}$, we have the following. 
\begin{itemize}
    \item For $\mu, \nu\in P$ and $\widetilde{b}\in \BBB(\infty)$, 
    \begin{equation}
    c(\mu)c(\nu)=c(\mu+\nu),\quad c(\widetilde{b})c(\mu)=q^{(\wt \widetilde{b}, \mu)}c(\mu)c(\widetilde{b}). \label{eq:borel}  
    \end{equation}
    \item The $\AAA$-submodule $R_q[U^+]_{\AAA}:=\sum_{\widetilde{b}\in \BBB(\infty)}\AAA c(\widetilde{b})$ of $R_q[B^+]_{\AAA}$ is a subalgebra. 
\end{itemize}
The algebra $R_q[U^+]_{\AAA}$ is called a quantized coordinate ring of $U^+$ or quantum unipotent subgroup. See~\cite[Section 6]{goodearl2020integral} and references therein. 
It is known that $R_q[U^+]_{\AAA}$ has an $\AAA$-basis, called the dual PBW-basis, and the algebra structure with respect to the dual PBW-basis is governed by the dual Levendorskii--Soibelman formula in \cite[Theorem 4.27]{Kimura10} (cf.~\cite[Theorem 5.2]{goodearl2020integral}). By the form of the dual Levendorskii--Soibelman formula together with \eqref{eq:borel} and the antiisomorphism $\omega^{\ast}\colon R_q[B^+]_{\kk}\xrightarrow[]{\sim} R_q[B^-]_{\kk}, \phi\mapsto \phi\circ\omega$ (\cite[Section 3.5]{Lus-Zform}), the standard argument (focusing on the ``leading term'') shows that  $R_q[B^-]_{\kk}\otimes_{\kk}R_q[B^+]_{\kk}$ is a domain when $\kk$ is an integral domain.  
\end{proof}
\section{Quantum cluster algebras over commutative rings}
\label{sec:cluster-alg}
Through out this section we fix a commutative unital ring $\kk$ and a unit $q^\Hf\in \kk^\times$. To study the cluster structure on the quantized coordinate ring $R_q[G]_{\kk}$ over $\kk$, we formulate the quantum cluster algebras over $\kk$. All notions and results in this section are already known over $\base$, where $\base$ is any intermediate ring between $\Z$ and $\C$ and $q^\Hf=1$ in the classical case and any intermediate ring between $\Z[q^{\pm \Hf}]$ and $\Q(q^\Hf)$ in the quantum case.

\subsection{Basics of cluster algebras}
\subsubsection*{Seeds}
Choose and fix a finite set $I$ and a partition $I=I_{\ufv}\sqcup I_{\fv}$. The elements in $I_{\ufv}$ and $I_{\fv}$ are called the unfrozen vertices and the frozen vertices respectively. Choose and fix $d_i\in \Z_{>0}$, $i\in I$. 

Let $\tB=(b_{ik})_{i\in I,k\in I_{\ufv}}$ denote a $\Z$-matrix such that $d_i b_{ij}=-d_j b_{ji}$ for all $i,j\in I_{\ufv}$. Let $x_i$, $i\in I$, denote indeterminates. The colletion $\sd:=(\tB,(x_i)_{i\in I})$ is called a \emph{(classical) seed} \cite{fomin2002cluster}, and $B:=\tB_{I_\ufv,I_\ufv}=(b_{ik})_{i, k\in I_{\ufv}}$ is called the \emph{principal $B$-matrix}. We say $\sd$ is \emph{of full rank} if $\tB$ is of full rank.

A vertex $k\in I_{\ufv}$ is said to be \emph{isolated in $\sd$} if $b_{ik}=0$ for all $i\in I$. We will make the following mild assumption in the rest of this paper.
\begin{asm}\label{asm:bad_case}
Either $2$ is not a zero divisor of $\kk$, or there is no isolated unfrozen vertex $k\in I_{\ufv}$.
\end{asm}

We associate to the seed $\sd$ the Laurent polynomial ring $\LP(\sd)_\kk:=\kk[x_i^{\pm 1}]_{i\in I}$ and its subring $\bLP(\sd)_\kk:=\kk[x_j]_{j\in I_{\fv}}[x_k^{\pm 1}]_{k\in I_{\ufv}}$. Denote the commutative product by $\cdot$.
When the base ring $\kk$ is clear from the discussion, the subscript $\kk$ will be omitted. 

Let $f_i$, $i\in I$, denote the standard basis vectors in $\Z^I$, and $e_k$, $k\in I_{\ufv}$, denote those in $\Z^{I_\ufv}$. We have the Laurent monomials $x^m:=\prod_{i\in I} x_i^{m_i}$ for $m=\sum_i m_i f_i=(m_i)_{i\in I}\in \Z^I$. We also introduce Laurent monomials $y_k:=x ^{\sum_i b_{ik}f_i}$ for $k\in I_{\ufv}$ and $y^n:=\prod y_k^{n_k}$ for $n=\sum_k n_k e_k =(n_k)\in \Z^{I_{\ufv}}$. 

Denote $\kk[N^\oplus]=\kk[y_k]_{k\in I_{\ufv}}$ and let $\widehat{\kk[N^\oplus]}=\kk\llb y_k\rrb_{k\in I_{\ufv}}$ be the ring of formal power series in $y_k$. When $\tB$ is of full rank, we further define the formal completion $\hLP(\sd):=\LP(\sd)\otimes_{\kk[N^\oplus]} \widehat{\kk[N^\oplus]}$, whose elements are called formal Laurent series.

The indeterminates $x_i$, $i\in I$, are called the \emph{cluster variables of $\sd$}. The \emph{cluster monomials of $\sd$} are $x^m$, $m\in \N^I$, and the \emph{localized cluster monomials} are $x^m$, $m\in \N^{I_\ufv}\oplus\Z^{I_\fv}$. The indeterminates $x_j$, $j\in I_{\fv}$ are called \emph{frozen variables}, and their Laurent monomials are called the \emph{frozen factors}.

\subsubsection*{Quantization}
Let $\Lambda=(\Lambda_{ij})_{i,j\in I}$ denote a skew-symmetric $\Z$-matrix. We make $\Lambda$ into a skew-symmetric bilinear form on $\Z^I$ such that $\Lambda(m,m'):=m^T \Lambda m'$. We endow $\LP(\sd)$ and $\hLP(\sd)$ with a $q$-twisted product $*$, such that
\begin{align*}
    x^m*x^{m'}=q^{\Hf \Lambda(m,m')}x^{m+m'},\ \forall m,m'\in \Z^I.
\end{align*} 
The algebra $(\LP(\sd), *)$ is the \emph{quantum torus algebra} associated with $\sd$. 
We will use $*$ as the multiplication of $\LP(\sd)$ and $\hLP(\sd)$ unless otherwise specified. We denote $\LP(\sd)_\kk$ when we want to emphasize our choice of $\kk$.  
We will take $q^\Hf=1$ and arbitrary $\Lambda$ when we work in the classical case. 

All Laurent monomials are units of $\LP(\sd)$ and $\hLP(\sd)$. This, in particular, implies that the elements  $x^m$, $m\in \Z^{I_\ufv}\oplus\N^{I_\fv}$ are regular elements of $\bLP(\sd)$, i.e., they are not zero divisors.

We say that $\Lambda$ is \emph{compatible} with $\tB$ if there exist positive integers $\ddiag_k\in \Z_{>0}$, $k\in I_{\ufv}$, such that $\sum _{j\in I}\Lambda_{ij}b_{jk}=-\delta_{ik}\ddiag_k$, $\forall i\in I,k\in I_{\ufv}$. A \emph{quantum seed} is a collection $\sd:=(\tB,\Lambda,(x_i)_{i\in I})$, where $\tB$ and $\Lambda$ are compatible, see \cite{BerensteinZelevinsky05}. Recall that a compatible $\Lambda$ for $\tB$ exists if and only if $\tB$ is of full rank \cite{gekhtman2003cluster,GekhtmanShapiroVainshtein05}.

For any permutation $\sigma$ of $I$, we can construct a seed $\sigma \sd$ such that $b_{\sigma i,\sigma k}(\sigma \sd)=b_{ik}$, $\Lambda_{\sigma i,\sigma j}(\sigma \sd)=\Lambda_{ij}$, and $x_{\sigma i}(\sigma \sd)=x_i$. We will apply permutations on seeds if necessary and view $\sigma \sd$ and $\sd$ as the same seed, omitting the symbol $\sigma$.

Recall \cite[Definition on p.~112]{GWbook} that a \emph{(left) Ore domain} is an associative unital ring $R$ with no zero divisors whose nonzero elements form a left Ore set (i.e., it satisfies the left Ore condition: $a R\cap b R\neq\{0\}$ for any nonzero $a,b\in R$). By Ore's theorem \cite[Theorem 6.8]{GWbook}, this is equivalent to saying that $R$ has a skew-field of fraction constructed by the left localization of $R$ by the set of its nonzero elements. If $\kk$ is an integral domain, then $\LP(\sd)$ is a left and right Ore domain. (For example, this condition is satisfied when $\kk=\Z$ in the classical situation and  $\kk=\Z[q^{\pm\Hf}]$ in the quantum situation.) To see the stated implication, denote by $F$ the field of fractions of $\kk$. Then $\LP(\sd)_F$ is a Noetherian domain, and thus a left and right Ore domain. Its skew-field of fractions $\cF(\sd)$ is the left and right localization of $\LP(\sd)$ by its nonzero elements; so $\LP(\sd)$ is an Ore domain with the same skew-field of fractions $\cF(\sd)$.

For the rest of this subsection, we assume that $\kk$ is an integral domain.

\subsubsection*{Mutation of seeds}
Let $\sd$ be a classical seed. For any unfrozen vertex $k$, the \emph{mutation $\mu_k$} at the vertex $k$ produces a new seed $\mu_k \sd:=\sd':=(\tB',(x'_i))$, such that $\tB'=(b'_{ij})$ is given by
\begin{align*}
    b'_{ij}&=-b_{ij}\text{ if }k\in \{i,j\},\\
    b'_{ij}&=b_{ij}+[b_{ik}]_+[b_{kj}]_+ - [-b_{ik}]_+[-b_{kj}]_+  \text{ if }k\notin \{i,j\},
\end{align*}
where $[a]_+:=\max(a,0)$. Next, define the linear map $\psi_k:=\psi_{\sd,\sd'}:\Z^I\rightarrow\Z^I$ by
\begin{align*}
    \psi_k(f_k)&=-f_k+\sum_j[-b_{jk}]_+f_j,\\
    \psi_k(f_i)&=f_i,\ \text{if }i\neq k.
\end{align*}
 When $\sd$ is a quantum seed, we quantize $\sd'$ by introducing the matarix $\Lambda'$ such that
\begin{align*}
    \Lambda'_{ij}=\Lambda(\psi_k(f_i),\psi_k(f_j)).
\end{align*}

 We use the decoration $(\sd)$ to denote data associated with $\sd$. We will often omit this decoration when the context is clear.

For any finite sequence $\uk:=(k_1,\ldots,k_r)$ of unfrozen vertices, we use $\seq:=\seq_{\uk}$ to denote the mutation sequence $\mu_{k_r}\cdots\mu_{k_1}$. Note that $j\in I$ is isolated in $\sd$ if and only if it is isolated in $\seq\sd$.

\subsubsection*{Cluster variables in Ore domains}

Since we assume $\kk$ is an integral domain, $\LP(\sd)$ and $\LP(\sd')$ are Ore domains. We introduce the algebra isomorphism $\mu_k^*:\cF(\sd')\simeq \cF(\sd)$, called the mutation map, such that
\begin{align}\label{eq:mutation_cluster_variable}
    \begin{split}
        \mu_k^*(x'_k)&=x^{-f_k+\sum_j [-b_{jk}]_+f_j}+x^{-f_k+\sum_j [b_{jk}]_+f_j},\\
    \mu_k^*(x'_i)&=x_i \text{ if }i\neq k.
    \end{split}
\end{align}
Note that $\mu_k (\mu_k \sd)=\sd$ and the composition $\cF(\sd)=\cF(\mu_k(\mu_k\sd))\xrightarrow{\mu_k^*}\cF(\sd')\xrightarrow{\mu_k^*}\cF(\sd)$ is the identity. We often identify $\cF(\sd')$ and $\cF(\sd)$ via $\mu_k^*$ and omit the symbol $\mu_k^*$.

Choose any initial seed $\sd$. Let $\Delta^+:=\Delta^+_\sd$ denote the set of seeds obtained from $\sd$ by finite sequences of mutations $\seq$. Then for any $\sdt\in \Delta^+_\sd$ and $\sdt'=\mu_k \sdt$, their cluster variables satisfy the following exchange relations in $\LP(\sd)$, where $\Lambda$ and $b_{ij}$ are associated with $\sdt$:
\begin{align}\label{eq:exchange_relation}
    \begin{split}
    x_k(\sdt) * x_k(\sdt')&=q^{\Hf\Lambda(f_k,\sum_j [-b_{jk}]_+f_j)}x(\sdt)^{\sum_j [-b_{jk}]_+f_j}+q^{\Hf\Lambda(f_k,\sum_j [b_{jk}]_+f_j)} x(\sdt)^{\sum_j [b_{jk}]_+f_j},\\
    x_i(\sdt')&=x_i(\sdt) \text{ if }i\neq k.
    \end{split}
\end{align}

\begin{thm}[\cite{FominZelevinsky07}\cite{Tran09}\cite{gross2018canonical}]
Take $(\kk,q^\Hf)$ to be $(\Z,1)$ in the classical case or $(\Z[q^{\pm\Hf}],q^\hf)$ in the quantum case. Then for any mutation sequence $\seq$ and $i\in I$, the (quantum) cluster variable $x_i(\seq \sd)$ takes the following form in $\LP(\sd)$:
\begin{align}\label{eq:cluster_expansion}
    x_i(\seq \sd)=x^{g_i(\seq \sd)}\cdot \sum_{n\in \N^{I_\ufv}}c_n x^{\tB n},
\end{align}
such that $g_i(\seq \sd)\in \Z^{I_\ufv}\oplus \N^{I_\fv}$, $c_n\in \kk$, and $c_0=1$. Moreover, by evaluating $c_0$ at $q^\Hf=1$, the quantum cluster variable becomes the corresponding classical cluster variable.
\end{thm}
\begin{rem}\label{rem:cluster-exp-data}
The vector $g_i(\seq \sd)$ is called the \emph{$i$-th extended $g$-vector of $\seq \sd$}. By \cite{FominZelevinsky07}, its restriction $\pr_{I_\ufv}g_i(\seq \sd)$, called the \emph{principal $g$-vector}, only depends on $i$, $\seq$, and principal $B$-matrix $B$. 

Similarly, the coefficients $c_n$ only depends on $n$, $i$, $\seq$, $B$, and the quantization scaling factors $\ddiag_k$ for $k\in I_{\ufv}$. Therefore, the $F$-polynomial $\sum_{n} c_n y^n$, viewed as an element in the algebra $\kk[y_k^{\pm 1}]_{k\in I_{\ufv}}$, has the same property. Note that this algebra only depends on $B$ and $\ddiag_k$ up to isomorphism.
\end{rem}

\subsection{(Quantum) cluster algebras over $\kk$}

\subsubsection*{Cluster variables and specialization} In this subsection, we consider the case of an arbitrary commutative unital ring $\kk$ under Assumption \ref{asm:bad_case}. 
As before, we set $q^\Hf=1$ for classical seeds. We denote $\AAA=\Z$ in the classical case and $\AAA=\Z[q^{\pm\Hf}]$ in the quantum case. 

\begin{defn}\label{def:specialization}
The \emph{specialization map} $\alpha$ from $(\AAA,q^\Hf)$ to $(\kk,q^\Hf)$ is the unital ring homomorphism from $\AAA$ to $\kk$ sending $q^\Hf\in \AAA$ to $q^\Hf\in \kk^\times$.

Let $\base$ denote any intermediate ring between $\AAA$ and $\Q(q^\Hf)$. If $\alpha$ extends to a homomorphism from $\base$ to $\kk$, the homormorphism is called the specialization map from $\base$ to $\kk$.
\end{defn}
Let $\alpha$ denote the specialization map $\alpha:\AAA\rightarrow\kk$. It turns $\kk$ into an $\AAA$-module. The tensor product $\LP(\sd)_\AAA\otimes_{\AAA}\kk$ is a unital ring and we have the canonical unital ring homomorphism 
\begin{align}\label{eq:specialization}
    \begin{split}
    \alpha:\LP(\sd)_\AAA\rightarrow &\LP(\sd)_\AAA\otimes_{\AAA}\kk\\
z\mapsto & z\otimes 1.
    \end{split}
\end{align}
Since $\LP(\sd)_\AAA$ is a free $\AAA$-module and $\LP(\sd)$ is a free $\kk$-module, we have the canonical $\kk$-module isomorphism between $\LP(\sd)_\AAA\otimes_\AAA \kk$ and $\LP(\sd)$. This is a $\kk$-algebra isomorphism, and we will identify these two algebras. Finally, $\alpha$ induces the following homomorphism, which by abuse of notation will be denoted by the same symbol:
\begin{align}\label{eq:specialization-LP}
\begin{split}
    \alpha:\LP(\sd)_\AAA&\rightarrow \LP(\sd)\\
   \sum_{m\in \Z^I} c_m x^m &\mapsto \sum_m \alpha(c_m)x^m,\ \forall c_m\in \AAA.
\end{split}
\end{align}

\begin{defn}\label{def:general_cluster_variable}
Take any $i\in I$ and mutation sequence $\seq$. Let $x_i(\seq \sd)_\AAA$ denote the $i$-th (quantum) cluster variable of $\seq \sd$ in $\LP(\sd)_\AAA$. Then the $i$-th (quantum) cluster variable of the seed $\seq \sd$ in $\LP(\sd)$ is defined to be the element $\alpha(x_i(\seq \sd)_\AAA)$.
\end{defn}

It follows from our construction that for $\sdt\in \Delta^+_\sd$ and $\sdt'=\mu_k \sdt$, their (quantum) cluster variables defined over $\kk$ satisfy the corresponding exchange relations \eqref{eq:exchange_relation} in $\LP(\sd)$. Using \eqref{eq:cluster_expansion}, we have
\begin{align}\label{eq:cluster_expansion_general}
    x_i(\seq \sd)=x^{g_i(\seq \sd)}\cdot \sum_{n\in \N^{I_\ufv}}\alpha(c_n) x^{\tB n},
\end{align}
where $c_n$ is the corresponding coefficients of $x_i(\seq \sd)_\AAA$. Note that $\alpha(c_0)=1$.

\begin{lem}\label{lem:adjacent-cluster-variable-regular}
Under Assumption \ref{asm:bad_case}, $x_i(\mu_k \sd)$ is a regular element of $\LP(\sd)$ (i.e., not a zero divisor) for all $i\in I$, $k\in I_\ufv$. 
\end{lem}
\begin{proof}
    It suffices to check that $x_k(\mu_k \sd)$ is a regular element of $\LP(\sd)$. We have $x_k(\mu_k \sd)=x^{-d}*(1 + q^{c}y)$, where $x^d, y$ are Laurent monomials in $x_i$, $i\in I$ and $c\in \Z\left[\frac{1}{2}\right]$. Note that $1 + q^{c}y$ is either $2$ (when $k$ is isolated) or a Laurent polynomial which is inhomgeneous in the degree of $x_j$ for some $j\in I$. In either case, it is a regular element of $\LP(\sd)$. The lemma follows from the fact that $x^{-d}$ is a unit of $\LP(\sd)$.
\end{proof}

\subsubsection*{Cluster algebras}
We define the \emph{(partially compactified) upper cluster algebra $\bUpClAlg(\sd)$} to be the subalgebra of $\LP(\sd)$ consisting of the elements $z$, such that for any mutation sequence $\seq$, $z$ has a Laurent expansion:
\begin{align*}
    x(\seq \sd)^d *z =\sum_{m\in \N^{I}} c_m x(\seq \sd)^m,
\end{align*}
for some $d\in \N^{I_\ufv}$ and $c_m\in\kk$. The \emph{(localized) upper cluster algebra $\upClAlg(\sd)$} is defined similarly but the condition on $d$ is changed to $d\in \N^I$.

The \emph{(partially compactified) ordinary cluster algebra $\bClAlg(\sd)$} is the $\kk$-subalgebra of $\bLP(\sd)$ generated by $x_i(\seq \sd)$. The \emph{(localized) ordinary cluster algebra $\clAlg(\sd)$} is the $\kk$-subalgebra of $\LP(\sd)$ generated by $x_i(\seq \sd)$ and $x_j^{-1}$ where $j\in I_{\fv}$. 

\begin{rem}\label{rem:image_specialization}
    We have $\bClAlg(\sd)=\langle \alpha(\bClAlg(\sd)_\AAA)\rangle_{\kk}\subset \langle \alpha(\bUpClAlg(\sd)_\AAA)\rangle_{\kk}\subset \bUpClAlg(\sd)$ and $\clAlg(\sd)=\langle \alpha(\clAlg(\sd)_\AAA)\rangle_{\kk}\subset \langle \alpha(\upClAlg(\sd)_\AAA)\rangle_{\kk}\subset \upClAlg(\sd)$, where $\langle X\rangle_{\kk}$ denotes the $\kk$-subalgebra generated by $X$. Let $\alg$ denote $\bClAlg,\clAlg,\bUpClAlg$, or $\upClAlg$. Clearly, $\alpha(\II \alg(\sd)_\AAA)=0$, where $\II$ is the kernel of $\alpha:\AAA\rightarrow \kk$. However, we do not know if $\alg(\sd)_\AAA/(\II\alg(\sd)_\AAA)\simeq \alpha(\alg(\sd)_\AAA)$, i.e. that $\alpha(\alg(\sd)_\AAA)$ is a specialization of $\alg(\sd)_\AAA$, in general, see Corollaries \ref{cor:specialize_trop_basis} \ref{cor:specialization-bUpClAlg} and \cite[Lemmas 3.1 and 3.3, Theorem 1.1]{geiss2020quantum} for special cases.
\end{rem}

Note that the frozen variables $x_j$, $j\in I_\fv$, are regular and normal elements (Definition \ref{normal-prime}) of $\bUpClAlg(\sd)$ and $\bClAlg(\sd)$, respectively. Hence, $\upClAlg(\sd)$ (resp. $\clAlg(\sd)$) is the localization of $\bUpClAlg(\sd)$ (resp. $\bClAlg(\sd)$) at the frozen variables.

\subsubsection*{Mutations}

Let $\alg$ denote $\bClAlg$, $\clAlg$, $\bUpClAlg$ or $\upClAlg$. We denote $\alg_\kk$ when we want to emphasize our choice of $\kk$. Then $\alpha:\LP(\sd)_\AAA\rightarrow \LP(\sd)$ restricts to a ring homomorphism $\alpha(\sd):\alg(\sd)_\AAA\rightarrow \alg(\sd)$.

For any $k\in I_\ufv$ and $\sd':=\mu_k \sd$, the mutation map restricts to the $\AAA$-algebra homomorphism $(\mu^*_{\sd',\sd})_\AAA:=\mu^*_k:\alg(\sd')_\AAA\simeq \alg(\sd)_\AAA$ such that $(\mu^*_{\sd',\sd})_\AAA x_i(\sd')_\AAA\in \LP(\sd)_\AAA$ is given by \eqref{eq:mutation_cluster_variable} for all $i\in I$.

\begin{lem}\label{lem:one-step-mutation}
    There is a unique $\kk$-algebra homomorphism $\mu^*_{\sd',\sd}:=\mu^*_k:\alg(\sd')\rightarrow \alg(\sd)$, called the mutation map, such that $\mu^*_{\sd',\sd}x_i(\sd')\in \LP(\sd)$ is given by \eqref{eq:mutation_cluster_variable}. Moreover, we have $\alpha(\sd) (\mu^*_{\sd',\sd})_\AAA= \mu^*_{\sd',\sd}\alpha(\sd')$. 
\end{lem}
\begin{proof}
Since all cluster variables in $\alg(\sd')_\AAA$ are Laurent polynomials in $x_i(\sd)_\AAA\in \LP(\sd')_\AAA$, $i\in I$, any element $z$ of $\alg(\sd')$ is a Laurent polynomial in $x_i(\sd)\in \LP(\sd')$. Indeed, $z$ is uniquely determined by its Laurent expansion in $x_i(\sd)$ since $x_i(\sd)$ are regular in $\alg(\sd')$ by Lemma \ref{lem:adjacent-cluster-variable-regular}.

It follows that any $\kk$-algebra homomorphism $f$ from $\alg(\sd')$ to $\LP(\sd)$ such that the elements $f(x_i(\sd))$ are regular is determined by its values $f(x_i(\sd))$. It is straightforward to check that, by defining $\iota x_i(\sd)$ to be $x_i(\sd)\in \LP(\sd)$, we obtain a $\kk$-algebra homomorphism $\iota:\alg(\sd')\rightarrow\LP(\sd)$. It follows from definition that we have the commutativity $\alpha(\sd)\iota_\AAA=\iota \alpha(\sd')$. The remaining statements follow from the fact that $\iota_\AAA \alg(\sd')_\AAA=\alg(\sd)_\AAA$ and $\alpha(\sd)\alg(\sd)_\AAA \subset \alg(\sd)$, where $\mu^*_{\sd',\sd}$ is the restriction of $\iota$ in the image $ \alg(\sd)$.
\end{proof}

 For any mutation sequence $\seq=\mu_{k_r}\cdots \mu_{k_1}$ and $\sd'=\seq \sd$, denote the mutation map $\seq^*:=(\mu^*_{\sd',\sd}):=\mu^*_{k_1}\cdots \mu^*_{k_r}:\alg(\sd')\rightarrow \alg(\sd)$. We recursively deduce that 
\begin{align}\label{eq:mutation_specialization}
\alpha(\sd) (\mu^*_{\seq\sd,\sd})_\AAA= \mu^*_{\seq\sd,\sd}\alpha(\seq\sd).
\end{align}
 
 By \eqref{eq:mutation_specialization}, $\mu^*_{\seq\sd,\sd}$ sends the initial cluster variable $x_i(\seq \sd)\in\LP(\seq \sd)$ to the cluster variable $x_i(\seq \sd):=\alpha(\sd)(x_i(\seq\sd)_\AAA)\in\LP(\sd)$ in Definition \eqref{def:general_cluster_variable}. Therefore, $\mu^*_{\seq \sd,\sd}\mu^*_{\sd,\seq\sd}=\mu^*_{\mu_{k_1}\cdots \mu_{k_r}\mu_{k_r}\cdots\mu_{k_1}\sd,\sd}$ is the identity on $\alg(\sd)$. In particular,  $\mu^*_{\seq\sd,\sd}:\alg(\seq\sd)\rightarrow\alg(\sd)$ is an isomorphism.
 \begin{lem}\label{lem:regular_in_cluster_alg}
For any $i\in I$ and mutation sequence $\seq$, the cluster variable $x_i(\seq \sd)$ is a regular element $\alg(\sd)$.
 \end{lem}
 \begin{proof}
 Note that $x_i(\seq \sd)$ is regular in $\alg(\seq\sd)$. We deduce the claim by using the isomorphism $\mu^*_{\seq \sd,\sd}:\alg(\seq\sd)\simeq \alg(\sd)$.
 \end{proof}
 \begin{cor}\label{cor:regular_in_LP}
For any $i\in I$ and mutation sequence $\seq$, the cluster variable $x_i(\seq \sd)\in\alg(\sd)$ is regular in $\LP(\sd)$.
 \end{cor}
 \begin{proof}
 Assume that there exists $0\neq z\in\LP(\sd)$ such that $x_i(\seq \sd)*z=0$. We can find some $d\in\N^{I}$ and a finite decomposition $z*x^{d}=\sum_{m \in \N^I} c_m x^m$, for $c_m\in\kk$, $m\in \N^I$. Note that $z*x^{d}\neq 0$ since $x^d$ is a unit in $\LP(\sd)$. We deduce $x_i(\seq \sd)*(\sum_m c_m x^m)=0$, thus $x_i(\seq\sd)$ is a zero divisor of $\alg(\sd)$, which is impossible by Lemma \ref{lem:regular_in_cluster_alg}. 
 
 By the same argument, there is no $0\neq z\in\LP(\sd)$ such that $z*x_i(\seq\sd)=0$.
 \end{proof}

 When we do not need to emphasize the choice of initial seed $\sd$, we will use the notation $\alg$ and $\alpha$, identifying $\alg(\seq\sd)$ and $\alg(\sd)$ via $\seq^*$ and omitting the symbols $(\sd)$ and $\seq^*$.

\subsubsection*{Order of vanishing}
Following \cite{qin2023analogs}, for any $j\in I$, let $\nu_j$ denote the order of vanishing at $x_j=0$ on $\LP(\sd)$, i.e., for any $z\neq 0\in\LP(\sd)$ with a reduced Laurent expansion $z*Q=x_j^d*P$, where $P\in \kk[x_i]_{i\in I}$ is not divisible by $x_j$, $Q$ is a monomial in $\{x_i|i\neq j\}$, we have $\nu_j(z)=d$. Denote $\nu_j(0)=+\infty$. Note that $\nu_j(z*z')\geq \nu_j(z)+\nu_j(z')$ when $\kk$ is not necessarily a domain. 

\begin{lem}[{$\kk$-analog of \cite[Lemma 2.12]{qin2023analogs}}]\label{lem:semivaluation}
Let $k\in I_{\ufv}$, $\sd=\mu_k \sd'$, and $j\neq k$. Denote by $\nu_j$ and $\nu'_j$ the order of vanishing on $\LP(\sd)$ and $\LP(\sd')$ respectively. Then for any $z\in \upClAlg$, we have $\nu_j(z)=\nu'_j(\mu_k^* z)$.
\end{lem}
\begin{proof}
Consider an element $z\in\upClAlg(\sd)$ with a reduced Laurent expansion $z*Q=x_j^d* P$ as described above. We need to show that $\nu'_j(\mu_k^* z)=d$. When $i\neq k$, we will denote $x'_i=x_i$ and thus $\mu_k^*(x_i)=x_i$.

We claim $\nu'_j(\mu_k^* P)=0$. Write $P=\sum_{s\geq 0}x_j^s*P_s:=\sum_{s\geq 0}x_j^s*\left(\sum_{r\geq 0}c_{r,s}*(x_k)^r\right)$, where $c_{r,s}\in\kk[x_i]_{i\neq j,k}$. Then $P_0\neq 0$ since $P$ is not divisible by $x_j$. Since $P_0$ is an inhomogenous sum with respect to the degree of $x_k$ (if it is sum of multiple terms), we have $c_{r,0}\neq 0$ for some $r$. We need to show that the constant term of $\mu_k^* P_0$ is nonzero as a polynomial in $x_j$.

Note that $\mu_k^* c_{r,s}=c_{r,s}$ and $\mu_k^*(x_k)^r=M_r*(x'_k)^{-r}$, where $M_r\in \kk[x_i]_{i\neq k}$ either does not contain $x_j$ (when $b_{jk}=0$), or exactly one of its monomial does not contains $x_j$ (when $b_{jk}\neq 0$). Moreover, this monomial in the latter case, denoted $M'_r$, has a coefficient in $q^{\frac{\Z}{2}}\subset\kk^\times$. In the first case, we have $\mu_k^*(P_0)|_{x_j=0}=\mu_k^*(P_0)=\sum_r c_{r,0}*\mu^*_k(x_k)^r$, which is an inhomogenous sum with respect to the degree of $x'_k$. Note that $\mu^*_k(x_k)^r$ is regular (Lemma \ref{lem:adjacent-cluster-variable-regular}) and some $c_{r,0}\neq 0$. Therefore, $\mu_k^*(P_0)\neq 0$. In the second case, we have $\mu_k^*(P_0)|_{x_j=0}=\sum_{r} c_{r,0}* M'_r*(x'_k)^{-r}$, which is an inhomogenous sum with respect to the degree of $x'_k$. Then $\mu_k^*(P_0)|_{x_j=0}\neq 0$ since $c_{r,0}\neq 0$ for some $r$. The desired claim follows.

Denote $Q=c*(x_k)^r$, where $c$ is a monomial in $x_i$, $i\neq j,k$. 
Note that $\nu'_j(\mu_k^* Q)=0$. When $b_{jk}=0$, $(\mu_k^* Q)|_{x_j=0}=\mu^*_k Q$ is regular in $\LP(\sd')$ (Lemma \ref{lem:adjacent-cluster-variable-regular}). When $b_{jk}\neq 0$, $(\mu_k^* Q)|_{x_j=0}=c*M'_r*(x'_k)^{-r}$ is still regular in $\LP(\sd')$. We deduce that $\nu'_j(\mu^*_k z* \mu^*_k Q)=\nu'_j(\mu^*_k z* (\mu^*_k Q)|_{x_j=0})=\nu'_j(\mu^*_k z)$. Therefore, we have $d=\nu'_j(x_j^d*(\mu_k^*P))=\nu'_j(\mu^*_k z*\mu^*_k Q)=\nu'_j(\mu^*_k z)$.
\end{proof}

By Lemma \ref{lem:semivaluation}, for $j\in I_\fv$, $\nu_j$ defined on $\upClAlg$ is independent of the choice of the initial seed $\sd$. Consequently, we have 
\begin{align}\label{eq:valuation_compactified_upper}
    \bUpClAlg=\{z\in\upClAlg\mid \nu_j(z)\geq 0,\forall j\in I_{\fv}\}.
\end{align}

\subsection{Tropical properties and bases}
We assume $\tB$ of $\sd$ is of full rank from now on unless otherwise specified. This property is preserved under mutations.
\subsubsection*{Degrees and pointedness}
For any seed $\sd$, we introduce a partial order $\prec_\sd$ on $\Z^I$, called the \emph{dominance order}, such that $m'\prec_\sd m$ if $m'=m+\tB n$ for some $0\neq n\in \N^{I_{\ufv}}$, see \cite[Definition 3.1.1]{qin2017triangular} and \cite[Proof of Proposition 4.3]{cerulli2015caldero}.

A formal Laurent series $z=\sum c_m x_m\in \hLP(\sd)$, $c_m\in \kk$, is said to have \emph{degree $g$} if $g$ is the unique $\prec_\sd$-maximal elements of $\{m|c_m\neq 0\}$, equivalently, we have
\begin{align*}
    z=c_g x^g+\sum_{m\prec_\sd g} c_m x^m,\ c_g\neq 0.
\end{align*}
Write $\deg z:=\deg^\sd z:=g$ in this case. It is said to be \emph{$g$-pointed} or \emph{pointed at $g$}, if further $c_g=1$. By \cite{FominZelevinsky07}\cite{gross2018canonical}, all cluster monomials of $\upClAlg(\sd)$ are pointed at distinct degrees.

We say $\sd$ is \emph{injective-reachable} if there exists a seed $\sd[1]\in \Delta^+$ and a permutation $\sigma$ of $I_{\ufv}$ such that $\deg^{\sd}x_{\sigma k}(\sd[1])\in -f_k+\sum_{j\in I_{\fv}} \Z f_j$, $\forall k\in I_{\ufv}$. Equivalently, $\sd$ has a green to red sequence \cite{keller2011cluster}, see \cite[Section 2.3.1]{qin2019bases}. This property is preserved under mutations \cite{qin2017triangular}\cite{muller2015existence}. In this case, define $\sd[d]$ recursively such that $\sd[d+1]=\sd[d][1]$, $\forall d\in \Z$.

\subsubsection*{Formal Laurent expansions and mutations of $y$-variables}

Every $g$-pointed element $z\in\hLP(\sd)$ is a unit and its inverse $z^{-1}\in \hLP(\sd)$ is $(-g)$-pointed. In particular, $g$-pointed elements of 
$\hLP(\sd)$ are regular.

For any mutation sequence $\seq$ and $\sd'=\seq \sd$, we define the $\kk$-algebra homomorphism $\iota:\LP(\sd')\rightarrow\hLP(\sd)$ such that $\iota(x_i(\sd')^{\pm 1})=(\seq^* x_i(\sd'))^{\pm 1}$, where $\seq^*$ is the mutation map considered in \eqref{eq:mutation_specialization}.
\begin{lem}[{\cite[Lemma 3.3.7]{qin2019bases}}]\label{lem:mutation_formal_expansion}
$\iota$ is injective. Moreover, if $z\in \upClAlg(\sd')$, then $\iota(z)=\seq^*(z)$.
\end{lem}

Consider the specialization map $\alpha(\sd):\hLP(\sd)_\AAA\rightarrow \hLP(\sd)$, which is the ring homomorphism extending $\alpha:\AAA\rightarrow \kk$ and sending $x(\sd)_\AAA^m$ to $x(\sd)^m$. It is straightforward to check that $\iota \alpha(\sd')=\alpha(\sd) \iota_\AAA\colon \LP(\sd')_\AAA\to \hLP(\sd)$.

We call $\iota(z)$ the \emph{formal Laurent expansion} of $z$ in $\hLP(\sd)$. In view of Lemma \ref{lem:mutation_formal_expansion}, we could denote $\iota=\seq^*$ by abuse of notation. 

Now consider the case $\seq=\mu_k$ and denote $\iota$ by $\mu_k^*$. Using the specialization map $\alpha(\sd)$ and $\alpha(\sd')$, we deduce that the $y$-variables of $\sd$ and $\sd'$ satisfy the mutation rules in \cite[Eq. (2.6)]{qin2019bases}.

\subsubsection*{Tropical points}
We associate to each seed $\sd$ a lattice $\cM(\sd)=\Z^I$. For any $k\in I_{\ufv}$ and $\sd'=\mu_k \sd$, define the adjacent tropical transformation  $\varphi_{\sd',\sd}:\cM(\sd)\simeq \cM(\sd')$ such that for any $m=(m_i)_{i\in I}\in\cM(\sd)$, its image $m'=\varphi_{\sd',\sd}(m)=(m'_i)_{i\in I}$ is given by 
\begin{align*}
    m'_k=&-m_k,\\
    m'_i=&m_i+[b_{ik}]_+[m_k]_+-[-b_{ik}]_+[-m_k]_+,\ \text{if }i\neq k.
\end{align*}
The tropical transformation $\varphi_{\sd',\sd}$ for any $\sd',\sd\in \Delta^+$ is defined as a composition of adjacent tropical transformations along any mutation sequence $\seq$ such that $\sd'=\seq\sd$. By \cite{gross2013birational},  $\varphi_{\sd',\sd}$ is independent of the choice of $\seq$. 

\begin{defn}[{\cite[Section 2.1]{qin2019bases}}]
We define the set of tropical points to be $M^{\trop}=\sqcup_{\sd\in \Delta^+}\cM(\sd)/\sim$, where $\sim$ is the equivalence relation such that $m\sim \varphi_{\sd',\sd}(m)$, $\forall m\in \cM(\sd)$, $\sd,\sd'\in\Delta^+$. The equivalence class $[m]$ of any $m\in\cM(\sd)$ is called a \emph{tropical point}.
\end{defn}

\begin{defn}
Take any element $z\in \upClAlg$ which is $m$-pointed in $\LP(\sd)$. If it is $\varphi_{\sd',\sd}(m)$-pointed in $\LP(\sd')$ for some $\sd'\in\Delta^+$, it is said to be \emph{compatibly pointed at $\sd,\sd'$}. It is said to be \emph{$[m]$-pointed} if it is compatibly pointed at $\sd,\sd'$ for any $\sd'\in\Delta^+$. A subset $Z$ of $\upClAlg$ is called \emph{$M^{\trop}$-pointed} if it takes the form $\{Z_{[m]}|[m]\in M^{\trop} \}$ such that $Z_{[m]}$ are $[m]$-pointed.
\end{defn}

\subsubsection*{Properties of pointed elements}
We next recall important properties of pointed elements following \cite{qin2019bases}. Although \cite{qin2019bases} deals with the case $\kk=\AAA$, all definitions and results in \cite[Sections 3, 4, and Lemma 5.1.1]{qin2019bases} only depend on consideration of linear combinations of pointed Laurent polynomials with distinct degrees and products of pointed Laurent polynomials; no special properties of $(\kk,q^\Hf)$ are used. Hence, the statements and proofs carry over verbatim to general pairs $(\kk,q^\Hf)$. The single place where \cite[Section 3, 4]{qin2019bases} used $\AAA$ is an explicit basis of any type $A_1$ upper cluster algebra in \cite[Section 4.2]{qin2019bases}. We therefore verify the corresponding statement over general commutative unital rings $\kk$ below.

\begin{lem}\label{lem:basis_A1}
Let $\sd$ be a classical or quantum seed such that $k$ is its only unfrozen vertex, whose exchange matrix $\tB$ is not necessarily of full rank. Then the localized cluster monomials form a $\kk$-basis of $\upClAlg(\sd)$ under Assumption \ref{asm:bad_case}.
\end{lem}
\begin{proof}
Denote $\sd=\mu_k \sd'$. Regard $\LP(\sd)$ and $\LP(\sd')$ as $\Z$-graded $\kk$-algebras by setting $\gr x_i=0$ when $i\neq k$, $\gr x_k=1$, $\gr x'_k=-1$. 
For all $r\in \N$, write $(x'_k)^r*(x_k)^r=M_r$, where $M_r$ is a polynomial in $x_i$, $i\neq k$.

For any $z\in \upClAlg$, we have the Laurent expansion $z=\sum_{r\geq 0} \alpha_r*x_k^r+ \sum_{r>0}\beta_r*x_k^{-r}$ in $\LP(\sd)$, where $\alpha_r,\beta_r$ are Laurent polynomials in $R:=\kk[x_i^{\pm 1}]_{i\neq k}$. Then $z':=\sum_{r>0}\beta_r*x_k^{-r}\in \LP(\sd)$ is contained in $\upClAlg$. We claim that, for all $r\in \Z_{>0}$, there exists $\beta'_r\in R$ such that $\beta_r=\beta'_r*M_r$ in $R$ or, equivalently, $z'=\sum_{r>0}\beta'_r*(x'_k)^r$ in $\upClAlg$.

(i) First, assume that $k$ is not isolated. Then $\tB(\sd')$ is of full rank, $M_r$ are pointed in $\LP(\sd')$ up to some $q$-multiples, and we can define $\mathcal{Q}:=\hLP(\sd')$. The grading on $\LP(\sd')$ extends to a grading on $\mathcal{Q}$. Recall that we have the homomorphism $\mu_k^*:\LP(\sd)\rightarrow \mathcal{Q}$, which restricts to the usual mutation map $\upClAlg(\sd)\simeq\upClAlg(\sd')$ and sends any element to its formal Laurent expansion (Lemma \ref{lem:mutation_formal_expansion}). Note that $\mu_k^*$ respects the grading. We have $\mu^*_k(z')=\sum_{r} \beta_r*\mu^*_k((x_k)^{-r})$ in $\mathcal{Q}$, whose homogeneous component of grading $-r$ is $\beta_r*\mu^*_k((x_k)^{-r})$. Since $\mu^*_k(z')\in \upClAlg(\sd')\subset \LP(\sd')$, these homogeneous components belong to $\LP(\sd')$ as well. We deduce that $\beta_r*\mu^*_k((x_k)^{-r})=\beta'_r*(x'_k)^r$ in $\mathcal{Q}$ for some $\beta'_r\in R$. Note that $M_r$ and $(x'_k)^r$ are invertible in $\mathcal{Q}$. We obtain $\beta_r*(M_r)^{-1}*(x'_k)^{r}=\beta'_r*(x'_k)^r$ in $\mathcal{Q}$, and thus $\beta_r=\beta'_r*(M_r)$ as claimed.

(ii) Next, assume that $k$ is isolated, then $\sd$ must be a classical seed and $M_r=2^r$. Since $2$ is a regular central element of $\LP(\sd')$, we have the localization $\mathcal{Q}:=\LP(\sd')[\Hf]$. We then construct the $R$-algebra homomorphism $\mu^*_k:\LP(\sd)\rightarrow \mathcal{Q}$ such that $x_k$ is sent to the cluster variable $x_k(\sd)=\frac{2}{x_k(\sd')}\in \LP(\sd')$. The grading on $\LP(\sd')$ extends to a grading on $\mathcal{Q}$. We then repeat the arguments in (i) and verify the desired claim.  

By this claim, $B=\{x_k^r,(x'_k)^r\mid r\in \N\}$ is an $R$-spanning set of $\upClAlg$ which consist of regular elements. Moreover, it is $R$-linearly independent in $\LP(\sd)$ by using the $\Z$-grading. In addition, the Laurent monomials in $x_i$, where $i\neq k$, are $\kk$-linearly independent in $\LP(\sd)$. Hence, we deduce that the localized cluster monomials, which span $\upClAlg$ over $\kk$, are $\kk$-linearly independent in $\LP(\sd)$ and thus form a $\kk$-basis of $\upClAlg$.
\end{proof}

\begin{lem}[{\cite[Lemma 3.4.12]{qin2019bases}}]\label{lem:tropical_cluster_mono}
Assume $\sd$ is injective-reachable and $z\in\LP(\sd)$ has degree $\deg z=\deg x(\sd')^m$ for some localized cluster monomial $ x(\sd')^m$. If it is compatibly pointed at $\sd,\sd',\sd'[-1]$, then $z=x(\sd')^m$.
\end{lem}
Lemma \ref{lem:tropical_cluster_mono} plays a key role in the proof of Theorem \ref{thm:tropical-basis}, and it is also of independent interest.

\begin{thm}[{\cite[Theorem 4.3.1]{qin2019bases}}]\label{thm:tropical-basis}
Assume that $\sd$ is injective-reachable with $\sd[1]=\seq \sd$ for some mutation sequence $\seq$. If $Z=\{z_m\mid m\in \cM(\sd)\}$ is a subset of $\upClAlg$ such that $z_m$ are $\varphi_{\sdt,\sd}(m)$-pointed in $\LP(\sdt)$ for any seed $\sdt$ appearing along the mutation sequence $\seq$ starting from $\sd$, then $Z$ is a $\kk$-basis of $\upClAlg$. 
\end{thm}
\begin{proof}
    The result was proved in \cite[Section 4]{qin2019bases} over $\AAA$ based on the fact that the localized cluster monomials form a basis of a type $A_1$ cluster algebra with full rank seeds. This fact is still true over $\kk$ by Lemma \ref{lem:basis_A1}. The rest statements in \cite[Section 4]{qin2019bases} carry over verbatim to $\kk$.
\end{proof}

Following \cite{gross2018canonical}, we say $\sd$ can be \emph{optimized} if for any frozen vertex $j$, there exist a seed $\sd_j\in \Delta^+$ such that $b_{jk}(\sd_j)\geq 0$ for all $k\in I_{\ufv}$. Using the description of $\bUpClAlg$ in \eqref{eq:valuation_compactified_upper}, we deduce the following results.

\begin{prop}[{\cite[Proposition 2.15]{qin2023analogs}}]\label{prop:partial-basis}
Assume that $\sd$ can be optimized. If $Z$ is an $M^{\trop}$-pointed basis of $\upClAlg$, then $Z\cap \bUpClAlg$ is a basis of $\bUpClAlg$.
\end{prop}

\begin{lem}[{\cite[Lemma 2.17]{qin2023analogs}}]\label{lem:optimized-tropical-pt}
Assume that $\sd$ can be optimized. Then, for any $[m]$-pointed $z\in\upClAlg$, we have $z\in \bUpClAlg$ if and only if $(\deg^{\sd_j}z)_j\geq 0$, $\forall j\in I_{\fv}$.
\end{lem}

\subsubsection{Algebras under specialization}
Let us consider the specialization homomorphism $\alpha:\upClAlg_\AAA\rightarrow \upClAlg$, which can be realized as the restriction of $\alpha(\sdt):\LP(\sdt)_\AAA\rightarrow\LP(\sdt)$ for any $\sdt\in\Delta^+$.

\begin{lem}\label{lem:specialization_tropical_point}
If $z\in \upClAlg_\AAA$ is $[m]$-pointed, then $\alpha(z)\in \upClAlg$ is $[m]$-pointed too.
\end{lem}
\begin{proof}
The claim follows from the commutativity between the specializations $\alpha(\sdt):\LP(\sdt)_\AAA\rightarrow\LP(\sdt)$, $\forall\sdt\in\Delta^+$, and mutations in \eqref{eq:mutation_specialization}.
\end{proof}

Let $\II$ denote the kernel of $\alpha:\AAA\rightarrow\kk$. Lemma \ref{lem:specialization_tropical_point} and Theorem \ref{thm:tropical-basis} imply the following.
\begin{cor}\label{cor:specialize_trop_basis}
Assume that $\sd$ is injective-reachable. If $Z$ is an $M^{\trop}$-pointed basis of a  (quantum) upper cluster algebra $\upClAlg_\AAA$, then its specialization $\alpha(Z)$ is an $M^{\trop}$-pointed basis of the (quantum) upper cluster algebra $\upClAlg$. Consequently, we have $\upClAlg_\AAA/(\II \upClAlg_\AAA)\simeq \alpha(\upClAlg_\AAA)$ and $\upClAlg_\AAA \otimes_\AAA \kk$ is canonically isomorphic to $\upClAlg$.
\end{cor}

\begin{cor}\label{cor:specialization-bUpClAlg}
Assume that $\sd$ is injective-reachable and can be optimized. If $\upClAlg_\AAA$ has an $M^{\trop}$-pointed basis $Z$, then $\alpha(Z\cap \bUpClAlg_\AAA)$ is a basis of $\bUpClAlg$.  Consequently, we have $\bUpClAlg_\AAA/(\II \bUpClAlg_\AAA)\simeq \alpha(\bUpClAlg_\AAA)$ and $\bUpClAlg_\AAA \otimes_\AAA \kk$ is canonically isomorphic to $\bUpClAlg$.
\end{cor}
\begin{proof}
By Corollary \ref{cor:specialize_trop_basis}, $\upClAlg_\AAA \otimes_\AAA \kk$ equals $\upClAlg$, and $\alpha(Z)$ is its $M^{\trop}$-pointed basis.

    Since the seed $\sd$ can be optimized, $Z\cap \bUpClAlg_\AAA$ is a basis of $\bUpClAlg_\AAA$ and $\alpha(Z)\cap \bUpClAlg$ is a basis of $\bUpClAlg$ by Proposition \ref{prop:partial-basis}. Moreover, Lemma \ref{lem:optimized-tropical-pt} and Lemma \ref{lem:specialization_tropical_point} imply that $\alpha(Z\cap \bUpClAlg_\AAA)=\alpha(Z)\cap \bUpClAlg$. The desired claims follow.
\end{proof}

\subsection{Berenstein--Zelevinsky quantum seeds}
\label{sec:BZ-seed}
Recall from Section \ref{2.1} that $C=(c_{ij})_{i,j=1}^r$ denotes the Cartan matrix of $G$ (what follows also applies to generalized Cartan matrices).
For any given element $(u,w)\in W\times W$, denote $\ell=\ell(u)+\ell(w)$. Choose a reduced word $\ubi=(\bi_1,\ldots,\bi_{\ell(w)+\ell(u)})$ for $(u,w)$ such that the simple reflections for the first component are enumerated by $[-r,-1]$ and those for the second component by $[1,r]$. Let $(u_{\leq k},w_{\leq k})\in W\times W$ denote the element corresponding to the product of the simple reflections associated with $(\bi_1,\ldots,\bi_k)$. Fix a permutation $(\bi_{-r},\ldots,\bi_{-1})$ of $[1,r]$. 

Define the set of vertices to be $I=[-r,-1]\sqcup [1,\ell]$. For $k\in I$, define
\begin{align*}
    \gamma_k:=\begin{cases}
        \varpi_{\bi_k},&\text{ if }k\in [-r,-1]\\
        u_{\leq k}\varpi_{|\bi_k|},&\text{ if }k\in [1,\ell]
    \end{cases},\  \delta_k:=\begin{cases}
        w^{-1}\varpi_{\bi_k},&\text{ if }k\in [-r,-1]\\
        w^{-1}w_{\leq k}\varpi_{|\bi_k|},&\text{ if }k\in [1,\ell].
    \end{cases}
\end{align*}
Consider the $I \times I$ skew-symmetric matrix $\Lambda$ such that $\Lambda_{kj}=(\gamma_k,\gamma_j)-(\delta_k,\delta_j),\ \forall k>j$.

For any $k\in I$, define $k[1]=\min(\{j\in I|j>k,|\bi_j|=|\bi_k| \}\sqcup \{\infty\})$. The set of unfrozen vertices is defined to be $I_{\ufv}:=\{k\in [1,l]|k[1]\leq \ell\}$. Denote $\varepsilon_k:=\sign(\bi_k)$. Define the matrix $\tB=(b_{jk})_{j\in I,k\in I_{\ufv}}$ with entries
\begin{align*}
    b_{jk}=
    \begin{cases}
        -\varepsilon_k & k=j[1] \\
        \varepsilon_j & j=k[1]\\
        -\varepsilon_k c_{|\bi_j|,|\bi_k|} & \varepsilon_k=\varepsilon_{j[1]},\quad j<k<j[1]<k[1]\\
        \varepsilon_j c_{|\bi_j|,|\bi_k|} & \varepsilon_j=\varepsilon_{k[1]},\quad k<j<k[1]<j[1]\\
        -\varepsilon_k c_{|\bi_j|,|\bi_k|} & \varepsilon_k=-\varepsilon_{k[1]},\quad j<k<k[1]<j[1]\\
        \varepsilon_j c_{|\bi_j|,|\bi_k|} & \varepsilon_j=-\varepsilon_{j[1]},\quad k<j<j[1]<k[1]\\
        0 & \text{otherwise.}       
    \end{cases}.
\end{align*}

The collection $\sd^\BZ:=(\tB,\Lambda,(x_k)_{k\in I})$ is a quantum seed associated with $\ubi$ by \cite{BerensteinZelevinsky05}. We note that the matrix $\tB$ is negative to the one in \cite{BerensteinFominZelevinsky05}, and the classical seed $(-\tB, (x_k)_{k\in I})$ is denoted by $\sd^\BFZ$. 
By construction, there is no isolated unfrozen vertex $k\in I_{\ufv}$, hence Assumption \ref{asm:bad_case} is always satisfied for $\sd^\BZ$ and $\sd^\BFZ$.

Write $\bUpClAlg(\sd):=\bUpClAlg(\sd)_{\KK}$, where $\KK=\Q(q^{\Hf})$. We have the following result.
\begin{thm}[{\cite{QY2025partially}}]\label{thm:up-cluster-k-G}
Let $\sd^\BZ$ be a Berenstein--Zelevinsky quantum seed associated to a reduced word of $(w_0, w_0) \in W \times W$, where $w_0$ is the longest element of $W$. Then there is an algebra isomorphism $\kappa:\bUpClAlg(\sd^\BZ)\simeq R_q[G]$, sending $x_k$ to the generalized quantum minor $\Delta_{\gamma_k,\delta_k}$, for all $k\in I$. 

Moreover, different choices of $\ubi$ produce quantum seeds of $R_q[G]$ related by mutations. In particular, all generalized quantum minors of the irreducible representations $V_q(\varpi_i), i\in [1, r],$ are quantum cluster variables.   
\end{thm}
From now on, we will always view $R_q[G]$ as a cluster algebra via $\kappa$, and will identify $R_q[G]=\bUpClAlg(\sd^\BZ)$.

Let $\AAA$ denote $\Z[q^{\pm\Hf}]$. By \cite{qin2023analogs}, $\upClAlg(\sd^{\BZ})_\AAA$ has the the common triangular basis $\can$ in the sense of \cite{qin2017triangular}. It is an analog of the dual canonical basis and is $M^{\trop}$-pointed. In addition, the seed $\sd^{\BZ}$ can be optimized by \cite[Proposition 6.16]{qin2023analogs}. Hence $\overline{\can}:=\can\cap \bUpClAlg(\sd^{\BZ})_\AAA$ is a basis of $\bUpClAlg(\sd^{\BZ})_\AAA$ by Proposition \ref{prop:partial-basis}. Moreover, $\sd^\BZ$ is known to be injective-reachable; see \cite{shen2021cluster} or \cite[Section 8.1, Lemma 8.4]{qin2023analogs} for details. Therefore, Corollary \ref{cor:specialization-bUpClAlg} implies the following result over an arbitrary commutative unital ring $\kk$ endowed with the specialization map $\alpha\colon \AAA\to \kk$.
\begin{thm}\label{thm:specialize_bUpClAlg}
    The algebra $\bUpClAlg(\sd^{\BZ})_\AAA \otimes_\AAA \kk$ is canonically isomorphic to $\bUpClAlg(\sd^{\BZ})_\kk$. Moreover, $\alpha(\overline{\can})$ is a basis of $\bUpClAlg(\sd^{\BZ})_\kk$.
\end{thm}

\section{Upper cluster algebra structure on the integral form of $R_q[G]$}

In this section, the quantum cluster algebras $\clAlg$, $\bClAlg$, $\upClAlg$, and $\bUpClAlg$ are assumed to be defined over $\KK$ and, for an intermediate ring $\AAA \subseteq \NewRing \subseteq \KK$, we will use $\clAlg_{\NewRing}$, $\bClAlg_{\NewRing}$, $\upClAlg_{\NewRing}$, and $\bUpClAlg_{\NewRing}$ to denote the corresponding quantum cluster algebras defined over $\NewRing$, see Section \ref{sec:cluster-alg}. For any quantum seed $\sd$, let $\LP(\sd)$ denote the quantum torus algebra associated with $\sd$ over $\KK$ and $\LP(\sd)_{\NewRing}$ the one over $\NewRing$. 
Recall from Theorem \ref{thm:up-cluster-k-G} that we have an upper cluster structure on the quantized coordinate ring $R_q[G]=\bUpClAlg(\sd^\BZ)$, where $\sd^\BZ$ denotes the BZ quantum seed associated with any signed reduced word for $(w_0,w_0)\in W\times W$. We strengthen Theorem \ref{thm:up-cluster-k-G} to integral forms as follows. 

\begin{thm}\label{thm:upcluster-G-U}
The equality $R_q[G]=\bUpClAlg(\sd^\BZ)$ is restricted to $R_q[G]_{\AAA}=\bUpClAlg(\sd^\BZ)_{\AAA}$. 
\end{thm}

\begin{cor}\label{cor:upcluster-G-U}
Let $\kk$ be a commutative unital ring, $q^\Hf\in\kk^\times$, and $\alpha:\AAA\rightarrow\kk$ be any specialization map (Definition \ref{def:specialization}). Then  
$R_q[G]_{\kk}\simeq \bUpClAlg(\sd^\BZ)_{\kk}$. In particular, we have 
$R[G]_{\Z}\simeq \bUpClAlg(\sd^\BZ)_{\Z}=\bUpClAlg(\sd^\BFZ)_{\Z}$.
\end{cor}
\begin{proof}[Proof of Corollary \ref{cor:upcluster-G-U}]
    By \eqref{AAA'}, $R_q[G]_{\kk}=R_q[G]_\AAA \otimes_{\AAA} \kk$. On the other hand, Theorem \ref{thm:specialize_bUpClAlg} implies $\bUpClAlg(\sd^\BZ)_{\kk}\simeq \bUpClAlg(\sd^{\BZ})_\AAA \otimes_\AAA \kk$. Hence, the assertion follows from Theorem \ref{thm:upcluster-G-U}. 
\end{proof}

The rest of this section is devoted to the proof of Theorem \ref{thm:upcluster-G-U}. The following lemma holds for the upper cluster algebra $\bUpClAlg$ associated with any quantum seed $\sd$.
\begin{lem}\label{lem:cluster-intersect-integral}
For any quantum seed $\sd$, $\bUpClAlg\cap \LP(\sd)_{\AAA}=\bUpClAlg_{\AAA}$.
\end{lem}
\begin{proof}
    For any $z\in \bUpClAlg$ and any seed $\sd'$ obtained from $\sd$ by iterated mutations, we have $z*x(\sd')^d=\sum_{m\in \N^I} c_m x(\sd')^m$ for some $d\in \N^{I_\ufv}$ and $c_m\in \KK$, $m\in \N^I$. Assume now that $z\in \LP(\sd)_{\AAA}$. It suffices to show that $c_m\in \AAA$ for all $m\in \N^I$.  
    Let us take the Laurent expansions of $z$, $x(\sd')^d$, and $x(\sd')^m$, $m\in \N^I$ in $\LP(\sd)_{\AAA}$. Recall that $x(\sd')^m$ are pointed elements of distinct degrees in $\LP(\sd)_{\AAA}$, whose degrees could be computed by the map in \cite[Definition 3.3.1]{qin2019bases}. Note that the quantum seed $\sd$ is of full rank. 
    Let $m_0$ denote an element of $m\in \N^I$ such that $c_{m_0}\neq 0$ and $\deg^\sd x(\sd')^m\prec_\sd\deg^\sd x(\sd')^{m_0}$ for all $m_0\neq m\in \N^I$ with $c_m\neq 0$. Then $c_{m_0}$ is the Laurent coefficients of $x(\sd)^{m_0}$ in $z*x(\sd')^d$, which implies $c_{m_0}\in {\AAA}$ since $z*x(\sd')^d$ is contained in $\LP(\sd)_{\AAA}$. Next,  consider $z*x(\sd')^d-c_{m_0}x(\sd')^{m_0}=\sum_{m:m\neq m_0} c_m x(\sd')^m\in \LP(\sd)_{\AAA}$. We can deduce similarly that $c_{m_1}\in {\AAA}$ for some $m_1$ appearing on the right. Repeating this process, we deduce that $c_m\in {\AAA}$ for all $m\in \N^I$. Therefore, $z\in \bUpClAlg_{\AAA}$.
\end{proof}

\begin{defn} 
\label{normal-prime}
\hfill
\begin{enumerate}
\item[(i)] An element $\Delta$ of a (noncommutative) ring $R$ is called {\em{normal}} if 
\[
\Delta R = R \Delta.
\] 
\item[(ii)] Let $R$ be a domain, i.e., $0$ is the only zero divisor of $R$. A {\em{prime element}} of $R$ is a nonzero, nonunit, normal element $\Delta \in R$ such that $R/(\Delta)$ is a domain, where $(\Delta):= R \Delta = \Delta R$ is the principal ideal generated by $\Delta$.
\end{enumerate}
\end{defn}
For example, every central element $\Delta$ of $R$ is normal. For a domain $R$, $a\in R$ and a normal element $\Delta \in R$, we write 
\[
\Delta | a \quad \mbox{if} \quad a\in (\Delta). 
\]
The condition in part (ii) of the definition that $R/(\Delta)$ is a domain is equivalent to the classical property
\[
\Delta | ab \quad \Rightarrow \quad \Delta | a \,\,\,\,
\mbox{or} \,\,\,\, \Delta | b, \quad \forall a, b \in R.
\]

Recall that a (noncommutative) ring $R$ is a left (resp.~right) Ore domain if $R$ has a left (resp.~right) quotient ring $Q(R)$, see e.g. \cite[\S 2.1.14]{McConnell-Robson}. For example, every left Noetherian domain is a left Ore domain \cite[Theorem 2.1.15]{McConnell-Robson}. 

We will base our arguments on the following useful result.
\begin{thm}[\cite{QY2025partially}]\label{thm:localization-intersection}
    Assume that $R$ is a left (resp.~right) Ore domain. Let $\Delta_j$, $j\in J$, be a collection of prime elements in $R$, where $J$ is an index set. Let $E$ be a left (resp.~right) denominator set of $R$ (see \cite[\S 2.1.13]{McConnell-Robson}) such that $\Delta_j \nmid x$, $\forall x\in E$. Then we have
    \[
        R[\Delta_j^{-1}\mid j\in J]\cap R[E^{-1}]=R.
    \]
\end{thm}

It is well known that $R_q[G]$ is a Noetherian domain \cite[Lemma 9.1.9(i) and Proposition 9.2.2]{joseph2012quantum}.
Let 
\begin{equation}
    \label{set}
\{ \Delta_j \mid j\in J \} 
\end{equation}
denote the set of irreducible elements of $\AAA$, which are considered as elements of $R_q[G]_\AAA$. Remark that \eqref{set} can be also regarded as the set of the prime elements of $\AAA$ since $\AAA$ is a unique factorization domain.  
\begin{proposition}
\label{prop:prime_irreducible} The following hold:
\begin{enumerate}
\item[(i)] All elements $\Delta_j$ are prime elements of $R_q[G]_{\AAA}$. 
\item[(ii)] $\Delta_j \nmid x$ in $R_q[G]_{\AAA}$ for any generalized quantum minor $x$.
\end{enumerate}
\end{proposition}
\begin{proof} 
(i) All elements of the set \eqref{set} are in the center of $R_q[G]_\AAA$, in particular, they are normal elements. Hence it suffices to show that $R_q[G]_\AAA/(\Delta_j)$ is a domain for all $j\in I$. We can regard $\Delta_j$ as an element of $\AAA$ and set $\kk:= \AAA/(\Delta_j)$. Write $\alpha\colon \AAA\to \kk$ as a natural projection. Then $R_q[G]_\AAA/(\Delta_j)$ is isomorphic to $R_q[G]_{\kk}$ defined in \eqref{AAA'}. Since $\kk$ is an integral domain,  $R_q[G]_\AAA/(\Delta_j)$ is a domain by Lemma \ref{l:domain}. 

(ii) Take any generalized quantum minor $x=\Delta_{u\varpi_i,v\varpi_i}$. Assume that $\Delta_j\mid x$, then there exists $x'\in R[G]_{\AAA}$ such that $\Delta_j x'=x$. Then we have $\langle x,U_q(\g)_\AAA\rangle=\Delta_j\langle x',U_q(\g)_\AAA\rangle\subset \Delta_j\AAA$.

Recall that the extremal vectors satisfy the following relations:
\begin{align*}
    X_j^{+(-\langle w\varpi_i,h_j\rangle)}.v_{w\varpi_i}&=v_{s_jw\varpi_i},\quad\text{if }l(s_jw)=l(w)-1,\\
    X_j^{-(\langle w\varpi_i,h_j\rangle)}.v_{w\varpi_i}&=v_{s_jw\varpi_i},\quad\text{if }l(s_jw)=l(w)+1.
\end{align*}
Hence there exists some $\tilde{X}=X_{j_1}^{-(m_1)}\cdots X_{j_t}^{-(m_t)}X_{j'_1}^{+(m'_1)}\cdots X_{j'_{t'}}^{(m'_{t'})}\in U(\g)_\AAA$ such that $X_{j'_1}^{+(m'_1)}\cdots X_{j'_{t'}}^{(m'_{t'})}v_{v\varpi_i}=v_{\varpi_i}$ and $\tilde{X}v_{v\varpi_i}=v_{u\varpi_i}$. We deduce that $\langle \Delta_{u\varpi_i,v\varpi_i},\tilde{X}\rangle=1\notin \Delta_j\AAA$. This contradiction shows that the assumption $\Delta_j\mid x$ is impossible.
\end{proof}

Fix a reduced word $\bi=(i_1,\dots, i_N)$ of $w_0$. Then 
    \[
    \ubi^{\bullet}\coloneqq (i_N,\dots, i_1, -i_1,\dots, -i_N)
    \]
is a reduced word for $(w_0, w_0)$.  For $k=1,\dots, N$, set $w_0^{\leq k}\coloneqq s_{i_1}\cdots s_{i_k}$. 
Denote by $\sd^\BZ_{\bullet}$ the Berenstein--Zelevinsky quantum seed associated with $\ubi^{\bullet}$ and $(\bi_{-r}^{\bullet},\ldots,\bi_{-1}^{\bullet})=(r,\dots, 1)$. Recall the notation in Section \ref{sec:BZ-seed}. In this case, we have $I=[-r,-1]\sqcup [1, 2N]$, and through the identification via $\kappa$ in Theorem \ref{thm:up-cluster-k-G}, the quantum cluster $(x_k^{\bullet})_{k\in I}$ of $\sd^\BZ_{\bullet}$ is described as follows.
\[
x_k^{\bullet}=
\begin{cases}
    \text{$\Delta_{\varpi_{|k|},w_0\varpi_{|k|}}$ for $k=-r,\dots,-1$}, \\
    \text{$\Delta_{\varpi_{i_{N+1-k}}, w_0^{\leq N-k}\varpi_{i_{N+1-k}}}$ for $k=1,\dots,N$}, \\
    \text{$\Delta_{w_0^{\leq k-N}\varpi_{i_{k-N}}, \varpi_{i_{k-N}}}$ for $k=N+1,\dots,2N$}.
\end{cases}
\]
Set $E_{\bullet}:=\{x_k^{\bullet}\mid k\in I\}$. It is straightforward to verify $E_{\bullet}\supset \{\Delta_{\varpi_{k},\varpi_{k}}\mid k\in [1, r]\}$.
\begin{lem}\label{lem:Laurent}
In the skew-field of fractions $Q(R_q[G])$, we have $R_q[G]_{\AAA}[E_{\bullet}^{-1}]=\LP(\sd^\BZ_{\bullet})_{\AAA}$. 
\end{lem}
\begin{proof}
    Since $x_k^{\bullet}\in R_q[G]_{\AAA}$ for $k\in I$, we have $\LP(\sd^\BZ_{\bullet})_{\AAA}\subset R_q[G]_{\AAA}[E_{\bullet}^{-1}]$. Let us prove the other inclusion. It suffices to show that, for any $\phi\in R_q[G]_{\AAA}$, there exists a monomial $m$ in $E_{\bullet}$ such that $m\phi$ is contained in $\sum_{m\in \N^I}\AAA x(\sd^\BZ_{\bullet})^m$.  

    Recall the notation in the proof of Lemma \ref{l:domain}. 
    For $\mu\in P$, define  $c^-(\mu)\in U^{\leq}_q(\g)^{\ast}$ as 
\[
\langle c^-(\mu), xK_{\alpha}\rangle =\varepsilon(x)q^{(\mu, \alpha)}
\]
for $x\in U^-_q(\g)$ and $\alpha\in Q$. For $\widetilde{b}\in \BBB(\infty)$, define $c^-(\widetilde{b})\in U^{\leq}_q(\g)^{\ast}$ by 
\[
\langle c^-(\widetilde{b}), \widetilde{b}'K_{\alpha}\rangle =\delta_{\widetilde{b}, \widetilde{b}'}
\]
for $\widetilde{b}'\in \BBB(\infty)$ and $\alpha\in Q$. Note that, by the antiisomorphism $\omega^{\ast}\colon R_q[B^+]_{\AAA}\xrightarrow[]{\sim} R_q[B^-]_{\AAA}, \phi\mapsto \phi\circ\omega$, we have $\omega^{\ast}(c(\mu))=c^-(-\mu)$ and $\omega^{\ast}(c(\widetilde{b}))=c^-(\widetilde{b})$ for $\mu\in P$ and $\widetilde{b}\in \BBB(\infty)$. 
In particular, $R_q[U^-]_{\AAA}:=\sum_{\widetilde{b}\in \BBB(\infty)}\AAA c^-(\widetilde{b})$ is an $\AAA$-subalgebra of $R_q[B^-]_{\AAA}$.

The injective $\AAA$-algebra homomorphism $\iota_{\triangle}\colon R_q[G]_{\AAA}\to R_q[B^-]_{\AAA}\otimes_{\AAA}R_q[B^+]_{\AAA}$ in \cite[Lemma 3.7]{Lus-Zform} is given by 
\[
\iota_{\triangle}=(\pi_-\otimes \pi_+)\circ \Delta_{R_q[G]},
\]
where $\Delta_{R_q[G]}$ is a coproduct of $R_q[G]_{\AAA}$. If $\phi\in R_q[G]_{\AAA}$ satisfies 
\begin{equation}
    \phi(xK_{\alpha})=q^{(\mu, \alpha)}\phi(x),\ \forall x\in U_q(\g), \forall \alpha\in Q, \label{eq:weight}
\end{equation}
for some $\mu\in P$, then for $\widetilde{b}, \widetilde{b}'\in \BBB(\infty)$ and $\alpha, \alpha'\in Q$, 
\begin{align*}
    \iota_{\triangle}(\phi)(\widetilde{b}K_{\alpha}\otimes \omega(\widetilde{b}')K_{\alpha'})
    &=\phi(\widetilde{b}K_{\alpha} \omega(\widetilde{b}')K_{\alpha'})\\
    &=q^{(-\wt\widetilde{b}', \alpha)}\phi(\widetilde{b}\omega(\widetilde{b}')K_{\alpha+\alpha'})\\
    &=q^{(\mu-\wt\widetilde{b}', \alpha)+(\mu, \alpha')}\phi(\widetilde{b}\omega(\widetilde{b}'))\in \AAA.
\end{align*}
Therefore, 
\[
\iota_{\triangle}(\phi)=\sum_{\widetilde{b}, \widetilde{b}'\in \BBB(\infty)}\phi(\widetilde{b}\omega(\widetilde{b}')) c^-(\mu-\wt\widetilde{b}')c^-(\widetilde{b})\otimes c(\widetilde{b}')c(\mu).
\]
Note that this is a finite sum. Since, as an $\AAA$-module, $R_q[G]_{\AAA}$ is spanned by the elements satisfying \eqref{eq:weight} for some $\mu\in P$ (see Proposition \ref{p:Lusgen}), $\iota_{\triangle}(R_q[G]_{\AAA})$ is contained in 
\[
\widetilde{R}_0\coloneqq \sum_{\mu\in P, \widetilde{b}, \widetilde{b}'\in \BBB(\infty)}\AAA c^-(\mu-\wt\widetilde{b}')c^-(\widetilde{b})\otimes c(\widetilde{b}')c(\mu).
\]
Hence it suffices to show that, for any $\widetilde{\phi}\in \widetilde{R}_0$, there exists a monomial $m$ in $E_{\bullet}$ such that $\iota_{\triangle}(m)\widetilde{\phi}$ is contained in $\sum_{m\in \N^I}\AAA \iota_{\triangle}(x(\sd^\BZ_{\bullet}))^m$. 

For $i\in [1, r]$ and $w\in W$, define $D_{\varpi_i, w\varpi_i}\in R_q[U^+]_{\AAA}\subset R_q[B^+]_{\AAA}$ by 
\[
\langle D_{\varpi_i, w\varpi_i}, xK_{\alpha}\rangle =\langle \xi_{\varpi_i}, x.v_{w\varpi_i}\rangle. 
\]
Note that $D_{\varpi_i, w\varpi_i}=\pi_+(\Delta_{\varpi_i, w\varpi_i})c(-w\varpi_i)$. These elements are called the unipotent quantum minors, cf.~\cite[Eq. (6.8)]{goodearl2020integral}. Then, 
\begin{align*}
&\iota_{\triangle}(x_k^{\bullet})=\\
&\begin{cases}
    \text{$c^-(\varpi_{|k|})\otimes D_{\varpi_{|k|},w_0\varpi_{|k|}}c(w_0\varpi_{|k|})$ for $k=-r,\dots,-1$}, \\
    \text{$c^-(\varpi_{i_{N+1-k}})\otimes D_{\varpi_{i_{N+1-k}}, w_0^{\leq N-k}\varpi_{i_{N+1-k}}}c(w_0^{\leq N-k}\varpi_{i_{N+1-k}})$ for $k=1,\dots,N$}, \\
    \text{$c^-(\varpi_{i_{k-N}})\varphi^{\ast}(D_{\varpi_{i_{k-N}}, w_0^{\leq k-N}\varpi_{i_{k-N}}})\otimes c(\varpi_{i_{k-N}})$ for $k=N+1,\dots,2N$}.
\end{cases}
\end{align*}
Here $\ast\colon U_q^-(\g)\to U_q^-(\g)$ is the antiautomorphism defined by $X_i^-\mapsto X_i^-$ for $i\in [1, r]$, $\ast\colon R_q[U^+]_{\AAA}\to R_q[U^+]_{\AAA}$ is the antiautomorphism\footnote{By the isomorphism $\iota\colon U_q^-(\g)\xrightarrow[]{\sim}R_q[U^+]_{\KK}$ in \cite[Eq. (6.3)]{goodearl2020integral}, the algebra antiautomorphism $\ast\colon R_q[U^+]_{\AAA}\to R_q[U^+]_{\AAA}$ is also given by $\iota\circ \ast\circ \iota^{-1}|_{R_q[U^+]_{\AAA}}$, cf.~\cite[Lemma 6.16]{jantzen1996lectures} and \cite[Lemma 3.5]{Kimura10}.
} defined by $c(\widetilde{b})\mapsto c(\ast\widetilde{b})$ for $\widetilde{b}\in \BBB(\infty)$, and $\varphi^{\ast}\coloneqq \omega^{\ast}|_{R_q[U^+]_{\AAA}}\circ \ast\colon R_q[U^+]_{\AAA}\xrightarrow[]{\sim} R_q[U^-]_{\AAA}$. Thus our claim follows from the quantum cluster algebra structure on $R_q[U^+]_{\AAA}$ (and on $R_q[U^-]_{\AAA}$ via $\varphi^{\ast}$) verified in \cite[Theorem 7.3]{goodearl2020integral} together with the Laurent phenomenon. Indeed, the initial cluster variables for $R_q[U^+]_{\AAA}$ is given by 
\[
q^{a_k}D(k;\bi)\coloneqq q^{a_k}D_{\varpi_{i_{k}}, w_0^{\leq k}\varpi_{i_k}},\quad k=1,\dots, N
\]
for some $a_k\in \Z\left[\frac{1}{2}\right]$. Note that we have 
\begin{align*}
    \{\iota_{\triangle}(x_k^{\bullet})\mid k\in I\}=
    &\{c^-(\varpi_{i_k})\otimes D(k;\bi)c(w_0^{\leq k}\varpi_{i_k})\mid k\in [1, N]\}\\
    &\cup \{c^-(\varpi_{k})\otimes c(\varpi_{k})\mid k\in [1, r]\}\\ 
    &\cup\{c^-(\varpi_{i_k})\varphi^{\ast}(D(k;\bi))\otimes c(\varpi_{i_k})\mid k\in [1, N]\}.
\end{align*}
For $m=(m_1,\dots, m_N), m'=(m'_1,\dots, m'_N)\in \N^N$, we set 
\[
D(\bi)^m\coloneqq D(1;\bi)^{m_1}\cdots D(N;\bi)^{m_N},
\]
and  
\begin{align*}
&\widetilde{D}(\ubi^{\bullet})^{m\circ m'}\coloneqq \\
&(c^-(\varpi_{i_1})\varphi^{\ast}(D(1;\bi))\otimes c(\varpi_{i_1}))^{m_1}\cdots (c^-(\varpi_{i_N})\varphi^{\ast}(D(N;\bi))\otimes c(\varpi_{i_N}))^{m_N}\\
&\cdot (c^-(\varpi_{i_1})\otimes D(1;\bi)c(w_0^{\leq 1}\varpi_{i_1}))^{m'_1}\cdots (c^-(\varpi_{i_N})\otimes D(N;\bi)c(w_0^{\leq N}\varpi_{i_N}))^{m'_N}. 
\end{align*}
Note that $\widetilde{D}(\ubi^{\bullet})^{m\circ m'}$ is a monomial in $\iota_{\triangle}(E_{\bullet})$. By \eqref{eq:borel}, 
\[
\widetilde{D}(\ubi^{\bullet})^{m\circ m'}\simeq 
c^-(\mu^-_{m\circ m'})\varphi^{\ast}(D(\bi)^{m})\otimes D(\bi)^{m'}c(\mu_{m\circ m'}),
\]
where $\simeq$ stands for the equality up to some powers of $q^{\Hf}$, and $\mu_{m\circ m'}\coloneqq \sum_{k=1}^Nm_k\varpi_{i_k}+\sum_{k=1}^Nm'_kw_0^{\leq k}\varpi_{i_k}$, $\mu^-_{m\circ m'}\coloneqq \sum_{k=1}^Nm_k\varpi_{i_k}+\sum_{k=1}^Nm'_k\varpi_{i_k}$. 

Let $\widetilde{b}, \widetilde{b}'\in \BBB(\infty)$. By the Laurent phenomenon, there exists $n, n'\in \N^N$ such that 
\[
\varphi^{\ast}(D(\bi)^{n})c^-(\widetilde{b})\in \sum_{m\in \N^N}\AAA\varphi^{\ast}(D(\bi)^m),\quad 
D(\bi)^{n'}c(\widetilde{b}')\in \sum_{m'\in \N^N}\AAA D(\bi)^{m'}.
\]

Then, for $\mu\in P$, $\widetilde{D}(\ubi^{\bullet})^{n\circ n'}(c^-(\mu-\wt\widetilde{b}')c^-(\widetilde{b})\otimes c(\widetilde{b}')c(\mu))$ belongs to
\begin{align*}
&\sum_{m, m'\in \N^N, \nu\in P}\AAA c^-(\nu+\mu^-_{m\circ m'}-\mu_{m\circ m'})\varphi^{\ast}(D(\bi)^m)\otimes D(\bi)^{m'} c(\nu)\\
&=\sum_{m, m'\in \N^N, \nu\in P}\AAA (c^-(\nu)\otimes c(\nu))\widetilde{D}(\ubi^{\bullet})^{m\circ m'}.
\end{align*}
Therefore, for $\mu\in P$, there exists $n''_1,\dots, n''_r \in \N$ such that 
\begin{align*}
&\left(\prod_{k=1}^r(c^-(\varpi_k)\otimes c(\varpi_k))^{n''_k}\right)\widetilde{D}(\ubi^{\bullet})^{n\circ n'}(c^-(\mu-\wt\widetilde{b}')c^-(\widetilde{b})\otimes c(\widetilde{b}')c(\mu))\\
&\in 
\sum_{m, m'\in \N^N, m''_1,\dots, m''_r \in \N}\AAA \left(\prod_{k=1}^r(c^-(\varpi_k)\otimes c(\varpi_k))^{m''_k}\right)\widetilde{D}(\ubi^{\bullet})^{m\circ m'},
\end{align*}
which proves our claim. 
\end{proof}
\begin{proof}[Proof of Theorem \ref{thm:upcluster-G-U}]
   Write $\ring:=R_q[G]_{\AAA}$ for brevity. First, by Proposition \ref{prop:prime_irreducible}(i), the irreducible elements of $\AAA$, denoted by $\Delta_j, j\in J$, are prime elements of $\ring$. 
   Furthermore, the localization $\ring[\Delta_j^{-1}\mid j\in J]$ equals $R_q[G]$. Therefore, Theorem \ref{thm:localization-intersection} and Proposition \ref{prop:prime_irreducible}(ii) 
    imply $\ring[\Delta_j^{-1}]_{j\in J}\cap \ring[E_{\bullet}^{-1}]=\ring=R_q[G]_\AAA$.  On the other hand, we have $\ring[\Delta_j^{-1}]_{j\in J}=R_q[G]=\bUpClAlg(\sd^\BZ)$. Thus Lemmas \ref{lem:cluster-intersect-integral} and \ref{lem:Laurent} imply $\ring[\Delta_j^{-1}]_{j\in J}\cap \ring[E_{\bullet}^{-1}]=\bUpClAlg(\sd^\BZ)\cap \LP(\sd^\BZ_{\bullet})_{\AAA}=\bUpClAlg(\sd^\BZ)_{\AAA}$. Thus, we obtain $R_q[G]_{\AAA}=\bUpClAlg(\sd^\BZ)_{\AAA}$.    
\end{proof}
\section{Cluster algebra structure on $R_q[G]$}\label{sec:q-minor-generate}
\subsection{Main results}
Theorem \ref{thm:up-cluster-k-G} confirms a cluster structure on the quantized coordinate ring $R_q[G]=\bUpClAlg(\sd^\BZ)$, where $\sd^\BZ$ denotes the BZ-seed associated with any signed reduced word for $(w_0,w_0)\in W\times W$. Hence, the Laurent phenomenon implies $\bClAlg(\sd^\BZ)\subset \bUpClAlg(\sd^\BZ)=R_q[G]$. In this section, we study the relations between $\bClAlg(\sd^\BZ)$ and $R_q[G]$. The following is our first main result in this direction. 
\begin{thm}\label{thm:cluster-G-A-U}
        If $G$ is not of type $F_4$, then $R_q[G]=\bClAlg(\sd^\BZ)$.
\end{thm}
Note that the classical version (i.e., over $\KK=\C$) of Theorem \ref{thm:cluster-G-A-U} was shown in \cite[Theorem 3.1]{Oya} under the same assumption on $G$.

We have shown the integrality of the equality $\bUpClAlg(\sd^\BZ)=R_q[G]$ in  Theorem \ref{thm:upcluster-G-U}. In the same spirit, we will study the integrality of the equality $R_q[G]=\bClAlg(\sd^\BZ)$ in Theorem \ref{thm:cluster-G-A-U} and obtain the following result. 

\begin{thm}\label{thm:integral-cluster-G}
Assume $G$ is not of type $E_8$ or $F_4$. We choose 
\[
\base=\begin{cases}
    \AAA_{1/2}=\AAA[(q^2+1)^{-1}]&\text{when $G$ is of type $G_2$},\\
    \AAA&\text{otherwise.}
\end{cases}
\]
Then the equality $R_q[G]=\bClAlg(\sd^\BZ)$ in Theorem \ref{thm:cluster-G-A-U} can be restricted to
\[
R_q[G]_{\base}=\bClAlg(\sd^\BZ)_{\base}.
\]
\end{thm}

From now on, let $\kk$ be any commutative unital ring with a chosen $q^\Hf\in\kk^\times$. Then we have the unique specialization map $\alpha:\AAA\rightarrow\kk$ sending $q^\Hf$ to $q^\Hf\in\kk^\times$ (Definition \ref{def:specialization}). 
\begin{thm}\label{thm:integral-cluster-G-k}
Assume $G$ is not of type $F_4$. We choose 
\[
\base=
\begin{cases}
    \KK&\text{when $G$ is of type $E_8$},\\
    \AAA_{1/2}=\AAA[(q^2+1)^{-1}]&\text{when $G$ is of type $G_2$},\\
    \AAA&\text{otherwise.}
\end{cases}
\]
If $\alpha:\AAA\rightarrow\kk$ extends to a specialization map $\alpha:\base\rightarrow\kk$, then
\[
R_q[G]_{\kk}\simeq \bClAlg(\sd^\BZ)_{\kk}.
\]
\end{thm}

Since we already know $\bClAlg(\sd^\BZ)_{\kk}\subset \bUpClAlg(\sd^\BZ)_{\kk}\simeq R_q[G]_{\kk}$ by the Laurent phenomenon and Corollary \ref{cor:upcluster-G-U}, we can prove $\bClAlg(\sd^\BZ)_{\kk}\simeq R_q[G]_{\kk}$ by showing  $\bClAlg(\sd^\BZ)_{\kk}$ contains a generating set of $R_q[G]_{\kk}$ (through the isomorphism). The equality \eqref{DeDe} and Theorem \ref{thm:up-cluster-k-G} imply that all generalized quantum minors belong to $\bClAlg(\sd^\BZ)_{\kk}$. Therefore, Theorems \ref{thm:cluster-G-A-U}, \ref{thm:integral-cluster-G}, and \ref{thm:integral-cluster-G-k} are deduced from the following theorem, which might be of independent interest. 

\begin{theorem}\label{thm:integral-coord-ring}
Assume $G$ is not of type $F_4$. We choose $\base$ as in Theorem \ref{thm:integral-cluster-G-k}. Then the generalized quantum minors generate $R_q[G]_\base$ as an $\base$-algebra. 
\end{theorem}

Finally, in the classical $E_8$ case, we can show that generalized quantum minors generate $R[G]_{\Z}$ as a $\Z$-algebra (Theorem \ref{thm:integral-coord-ring-classical}). Hence, we obtain the following refinement of \cite[Theorem 3.1]{Oya}. 
\begin{theorem}\label{thm:integral-cluster-G-classical}
Assume $G$ is not of type $F_4$. We choose 
\[
\base=\begin{cases}
    \Z\left[\frac{1}{2}\right]&\text{when $G$ is of type $G_2$},\\
    \Z&\text{otherwise.}
\end{cases}
\]
Then we have $R[G]_\base\simeq \bClAlg(\sd^\BFZ)_{\base}=\bUpClAlg(\sd^\BFZ)_{\base}$. Moreover, If $\alpha:\AAA\rightarrow\kk$ extends to a specialization map $\alpha:\base\rightarrow\kk$, then  
\[
R[G]_{\kk}\simeq \bClAlg(\sd^\BFZ)_{\kk}=\bUpClAlg(\sd^\BFZ)_{\kk}.
\]
\end{theorem}
\begin{proof}
When $G$ is not of type $E_8$, the result is the special case of Theorem \ref{thm:integral-cluster-G-k}. In the case of type $E_8$, the results follows from Corollary \ref{cor:upcluster-G-U} and Theorem \ref{thm:integral-coord-ring-classical}.
\end{proof}
The rest of this paper is devoted to the proof of Theorems \ref{thm:integral-coord-ring} and \ref{thm:integral-coord-ring-classical}.
\subsection{Generators of $R_q[G]$}
Define the $U_q(\g)$-module $M_q$ as follows\footnote{Note that we do not exclude the case of  type $F_4$ here.}: 
\begin{align}
    M_q&\coloneqq 
    \begin{cases}
        V_q(\varpi_{r-1})\oplus V_q(\varpi_{r})&\text{if $\g$ is of type $D_r$,}\\
        V_q(\varpi_1)&\text{otherwise.}
    \end{cases}\label{eq:Mq}
\end{align}

Write the integral form of $M_q$ as $M_\AAA$, and set $M\coloneqq M_\AAA|_{q=1}$. If we regard $M$ as a $\g$-module (hence, $G$-module), then the matrix coefficients of $M$ generates the coordinate ring $R[G]_{\C}=\C[G]$. See \cite[Propositions 2.1 and 2.2]{IOS} and the argument in \cite[Proof of Theorem 3.2]{Oya}. Therefore, every finite dimensional irreducible $\g$-module appears as a direct summand of a certain iterated tensor product of $M$. Thus, by \eqref{mult}, every finite dimensional type one irreducible $U_q(\g)$-module can be obtained as a direct summand of a certain iterated tensor product of $M_q$. Hence, $R_q[G]$ is generated by the matrix coefficients of $M_q$ as a $\KK$-algebra. 

Denote by $\BBB_M$ the canonical basis of $M_q$ in the sense of \cite{Lus:intro}. Define a basis $\BBB_M^{\ast}=\{b^{\ast}\mid b\in \BBB_M\}$ of $M_q^{\ast}$ by $\langle b^{\ast}, b'\rangle=\delta_{b, b'}$ for $b, b'\in \BBB_M$. Set 
\[
\Gamma\coloneqq \{c^{M_q}(b^{\ast}, b')\mid b, b'\in \BBB_M\}\subset R_q[G]_\AAA. 
\]
As mentioned above, $\Gamma$ generates $R_q[G]$ as a $\KK$-algebra. The following theorem is a refinement of this statement. 
\begin{theorem}\label{t:integral-gen}
The set $\Gamma$ generates $R_q[G]_\AAA$ as an $\AAA$-algebra. 
\end{theorem}
\begin{proof}
    Denote by $R_{\AAA}$ the $\AAA$-subalgebra of $R_q[G]_\AAA$ generated by $\Gamma$. We shall show that $R_{\AAA}=R_q[G]_\AAA$.   
    Thanks to Proposition \ref{p:Lusgen}, it suffices to show that the set $\widetilde{\Gamma}$ in Proposition \ref{p:Lusgen} belongs to $R_{\AAA}$.  
    
    The module $M_q$ together with $\BBB_M$ forms a based module in the sense of \cite[Chapter 27]{Lus:intro}. For $k\in \Z_{>0}$, consider the $k$-th tensor product $(M_q, \BBB_M)^{\otimes k}=(M_q^{\otimes k}, \BBB_{k})$
    of the based module $(M_q, \BBB_M)$, which is again a based module, \cite[Section 27.3]{Lus:intro}. For $\lambda\in P^+$ and a type one $U_q(\g)$-module $V$, denote by $V[\lambda]$ the sum of the simple $U_q(\g)$-submodules of $V$ which are isomorphic to $V_q(\lambda)$, and set 
    \[
    V[\geq\! \lambda]\coloneqq \bigoplus_{\lambda'\geq  \lambda}V[\lambda'],\qquad V[>\!\lambda]\coloneqq \bigoplus_{\lambda'> \lambda}V[\lambda'].
    \]
    We have the following properties.
    \begin{itemize}
        \item[(i)] $\BBB_{k}[\geq\!\lambda]\coloneqq \BBB_{k}\cap (M_q^{\otimes k})[\geq\!\lambda]$ is a basis of $(M_q^{\otimes k})[\geq\!\lambda]$, and $\BBB_{k}[>\!\lambda]\coloneqq \BBB_{k}\cap (M_q^{\otimes k})[>\! \lambda]$ is a basis of $(M_q^{\otimes k})[>\!\lambda]$ for $\lambda\in P^+$ (\cite[Proposition 27.1.8]{Lus:intro}).
        \item[(ii)]  Let $\pi_{k, \lambda}\colon (M_q^{\otimes k})[\geq\!\lambda]\to (M_q^{\otimes k})[\geq\! \lambda]/(M_q^{\otimes k})[>\!\lambda]$ be a projection. Then $((M_q^{\otimes k})[\geq\!\lambda]/(M_q^{\otimes k})[>\!\lambda], \pi_{k, \lambda}(\BBB_{k}[\geq\!\lambda]\setminus \BBB_{k}[>\!\lambda]))$ is a based module. Moreover, this based module is isomorphic to the direct sum of $\dim (M_q^{\otimes k})[\lambda]_{\lambda}$ copies of $(V_q(\lambda), \BBB(\lambda))$ (\cite[Section 27.1.4, Proposition 27.1.7]{Lus:intro}). 
        \item[(iii)]  Any element of  $\BBB_{k}$ is an $\AAA$-linear combination of elements of $\BBB_M^{\otimes k}\coloneqq \{b'_1\otimes \cdots \otimes b'_k\mid b'_1,\dots, b'_k\in \BBB_M\}$. Conversely, any element of $\BBB_M^{\otimes k}$ is an $\AAA$-linear combination of elements of $\BBB_{k}$ (\cite[Theorem 27.3.2]{Lus:intro}). 
    \end{itemize} 
    
Let $\lambda\in P^+$ be a dominant integral weight such that $(M_q^{\otimes k})[\lambda]\neq 0$. By (ii), we can take a direct summand $V\subset \left((M_q^{\otimes k})[\geq \lambda]/(M_q^{\otimes k})[> \lambda]\right)$ and a subset $\BBB\subset \left(V\cap \pi_{k, \lambda}(\BBB_{k}[\geq\lambda]\setminus \BBB_{k}[>\lambda])\right)$ so that $(V, \BBB)$ is isomorphic to $(V_q(\lambda), \BBB(\lambda))$ as a based module. For $b_1, b_2\in \BBB(\lambda)$, we take the corresponding elements of $\BBB$, which are again written as $b_1, b_2$, respectively. Then, there exist $\widetilde{b}_i\in \BBB_{k}[\geq\lambda]\setminus \BBB_{k}[>\lambda]$ ($i=1,2$) such that 
\[
\pi_{k, \lambda}(\widetilde{b}_i)=b_i.
\]
Let $\BBB_k^{\ast}=\{\widetilde{b}^{\ast}\mid \widetilde{b}\in \BBB_k\}$ be the dual basis of $\BBB_k$ in $(M_q^{\otimes k})^{\ast}$, where $\langle \widetilde{b}^{\ast}, \widetilde{b}'\rangle=\delta_{\widetilde{b}, \widetilde{b}'}$ for $\widetilde{b}, \widetilde{b}'\in \BBB_k$. Then, by (ii),  we have 
\begin{equation}
c^{M_q^{\otimes k}}(\widetilde{b}_1^{\ast}, \widetilde{b}_2)=c^{V_q(\lambda)}(b_1^{\ast}, b_2).\label{eq:tilde}
\end{equation}
By (iii), $\widetilde{b}_1^{\ast}\in \BBB_{k}^{\ast}$ is an $\AAA$-linear combination of the elements of $\{(b'_k)^{\ast}\otimes \cdots \otimes (b'_1)^{\ast}\mid b'_1,\dots, b'_k\in \BBB_{M}\}$, and $\widetilde{b}_2$ is an $\AAA$-linear combination of the elements of $\BBB_M^{\otimes k}$.  Therefore, the left hand side of \eqref{eq:tilde} belongs to $R_{\AAA}$ by \eqref{eq:product}. 

The discussion before Theorem \ref{t:integral-gen} shows that, for all $\lambda\in P^+$, there exists $k\in \Z_{>0}$ such that $(M_q^{\otimes k})[\lambda]\neq 0$. Moreover, the elements of $\widetilde{\Gamma}$ are of the form of the right-hand side of \eqref{eq:tilde}. Therefore, we conclude that $\widetilde{\Gamma}\subset R_{\AAA}$. 
\end{proof}
When $G$ is of type $A_n$, $B_n$, $C_n$, $D_n$, $E_6$, or $E_7$, the module $M_q$ is a direct sum of minuscule representations, namely $\BBB_M$ consists of extremal weight vectors. Therefore, in this case, $\Gamma$ consists of generalized quantum minors. In the case of type $E_8$ and $G_2$, we have the following result. 
\begin{theorem}\label{t:EG-gamma}
When $G$ is of type $E_8$ and $G_2$, the elements of $\Gamma$ are given by polynomials in the generalized quantum minors with coefficients in $\KK$. Moreover, when $G$ is of type $G_2$, these coefficients can be taken from $
\AAA_{1/2}\coloneqq \AAA[[2]_q^{-1}]=\AAA[(q^2+1)^{-1}]$.
\end{theorem}
\begin{proof}
In the case of type $E_8$, $\varpi_1$ equals the highest root $\theta$ of $\g$. Hence,  $M_q=V_q(\varpi_1)=V_q(\theta)$ is the quantum adjoint representation. In this case, the theorem follows from Theorem \ref{t:adj-minor} below. Note that the set $\Upsilon_{\neq 0}$ in Theorem \ref{t:adj-minor} consists of generalized quantum minors when $G$ is simply-laced). 

In the following, we give a proof of the theorem in the case of type $G_2$. Denote by $R'_{\AAA_{1/2}}$ the $\AAA_{1/2}$-subalgebra of $R_q[G]_{\AAA_{1/2}}$ generated by the generalized quantum minors. We need to show that $\Gamma\subset R'_{\AAA_{1/2}}$. 
    In the case of type $G_2$, $M_q=V_q(\varpi_1)$ is a $7$-dimensional quasi-minuscule $U_q(\g)$-module, see \cite[Definition 3.4]{Oya} for the definition of quasi-minuscule $U_q(\g)$-module. Write $b_{(1,0)}=v_{\varpi_1}$, and set 
    \begin{align*}
        b_{(-1,1)}&\coloneqq X_1^-. b_{(1,0)},&
        b_{(2,-1)}&\coloneqq X_2^-. b_{(-1,1)},&
        b_{(0,0)}&\coloneqq X_1^-. b_{(2,-1)},\\
        b_{(-2,1)}&\coloneqq \frac{1}{[2]_{q}}X_1^-. b_{(0,0)},&
        b_{(1,-1)}&\coloneqq X_2^-. b_{(-2,1)},&
        b_{(-1,0)}&\coloneqq X_1^-. b_{(1,-1)}.
    \end{align*}
    These seven vectors form the canonical basis $\BBB_M=\BBB(\varpi_1)$ of $M_q$. Note that $\wt b_{(m, n)}=m\varpi_1+n\varpi_2$. 
    
    We already know that $c^{M_q}(b^{\ast}, b')$ is a generalized quantum minor if $b, b'\in \BBB_M\setminus\{b_{(0,0)}\}$. Therefore, it suffices to show that $c^{M_q}(b^{\ast}, b_{(0,0)})$ and  $c^{M_q}(b_{(0,0)}^{\ast}, b)$ belong to $R'_{\AAA_{1/2}}$ for all $b\in \BBB_M$. 
    
    Direct calculation shows that  
    \[
    \widetilde{b}_{(1,0)}:= b_{(1,0)}\otimes b_{(0,0)}- q[2]_qb_{(-1,1)}\otimes b_{(2,-1)}+q^4[2]_qb_{(2,-1)}\otimes b_{(-1,1)}-q^6b_{(0,0)}\otimes b_{(1,0)}
    \]
    satisfies $X_1^+. \widetilde{b}_{(1,0)}=X_2^+. \widetilde{b}_{(1,0)}=0$ and $\wt \widetilde{b}_{(1,0)}=\varpi_1$ in $M_q^{\otimes 2}$. Hence there exists an injective $U_q(\g)$-module homomorphism $\iota\colon M_q\to M_q^{\otimes 2}$ satisfying $\iota(b_{(1,0)})=\widetilde{b}_{(1,0)}$. We write $\iota(b_{(m,n)})$ as $\widetilde{b}_{(m,n)}$ for $b_{(m,n)}\in \BBB_M$. Then we have 
    \begin{align*}
        \widetilde{b}_{(-1,1)}
        &=[2]_qb_{(1,0)}\otimes b_{(-2,1)}-q^2b_{(-1,1)}\otimes b_{(0,0)}\\
        &\phantom{==}+q^4b_{(0,0)}\otimes b_{(-1,1)}-q^5[2]_qb_{(-2,1)}\otimes b_{(1,0)},\\
        \widetilde{b}_{(2,-1)}&=[2]_qb_{(1,0)}\otimes b_{(1,-1)}
        -q^2b_{(2,-1)}\otimes b_{(0,0)}\\
        &\phantom{==}+q^4b_{(0,0)}\otimes b_{(2,-1)}-q^5[2]_qb_{(1,-1)}\otimes b_{(1,0)},\\
        \widetilde{b}_{(0,0)}
        &=[2]_qb_{(1,0)}\otimes b_{(-1,0)}
        +q^{-1}[2]_qb_{(-1,1)}\otimes b_{(1,-1)}\\
        &\phantom{==}-q^2[2]_qb_{(2,-1)}\otimes b_{(-2,1)}
        +q^3(q-q^{-1})b_{(0,0)}\otimes b_{(0,0)}\\
        &\phantom{==}+q^2[2]_qb_{(-2,1)}\otimes b_{(2,-1)}
        -q^5[2]_qb_{(1,-1)}\otimes b_{(-1,1)}
        -q^4[2]_qb_{(-1,0)}\otimes b_{(1,0)},
        \\
        \widetilde{b}_{(-2,1)}
        &=[2]_qb_{(-1,1)}\otimes b_{(-1,0)}
        -q^2b_{(0,0)}\otimes b_{(-2,1)}\\
        &\phantom{==}+q^4b_{(-2,1)}\otimes b_{(0,0)}-q^5[2]_qb_{(-1,0)}\otimes b_{(-1,1)},\\
        \widetilde{b}_{(1,-1)}
        &=[2]_qb_{(2,-1)}\otimes b_{(-1,0)}
        -q^2b_{(0,0)}\otimes b_{(1,-1)}\\
     &\phantom{==}+q^4b_{(1,-1)}\otimes b_{(0,0)}
        -q^5[2]_qb_{(-1,0)}\otimes b_{(2,-1)},\\
        \widetilde{b}_{(-1,0)}
        &=b_{(0,0)}\otimes b_{(-1,0)}
        -q[2]_qb_{(-2,1)}\otimes b_{(1,-1)}\\
        &\phantom{==}+q^4[2]_qb_{(1,-1)}\otimes b_{(-2,1)}
        -q^6b_{(-1,0)}\otimes b_{(0,0)}.        
    \end{align*}
    First, we show $c^{M_q}(b_{(1,0)}^{\ast}, b_{(0,0)})\in R'_{\AAA_{1/2}}$. We have 
    \[
    c^{M_q}(b_{(1,0)}^{\ast}, b_{(0,0)})=-\frac{1}{q[2]_q}c^{M_q^{\otimes 2}}(b_{(2,-1)}^{\ast}\otimes b_{(1,-1)}^{\ast}, \widetilde{b}_{(0,0)}).
    \]
    Indeed, if we write $x. b_{(0,0)}=\sum_{b\in \BBB_M} a_b(x) b$ with $a_b(x)\in \KK$ for $x\in U_q(\g)$, then $x. \widetilde{b}_{(0,0)}=\sum_{b\in \BBB_M} a_b(x) \widetilde{b}$. 
 In addition, $b_{(-1,1)}\otimes b_{(2,-1)}$ does not appear in the expansion of $\widetilde{b}_{(m,n)}$ for $(m,n)\neq (1,0)$. Hence,
\begin{align*}
\langle c^{M_q}(b_{(1,0)}^{\ast}, b_{(0,0)}), x\rangle&=a_{b_{(1,0)}}(x)=\langle -\frac{1}{q[2]_q}c^{M_q^{\otimes 2}}(b_{(2,-1)}^{\ast}\otimes b_{(-1,1)}^{\ast}, \widetilde{b}_{(0,0)}), x\rangle.
\end{align*}
Here recall Remark \ref{r:dual}. Therefore, 
\begin{align*}
        c^{M_q}(b_{(1,0)}^{\ast}, b_{(0,0)})
        &=-\frac{1}{q[2]_q}c^{M_q^{\otimes 2}}(b_{(2,-1)}^{\ast}\otimes b_{(-1,1)}^{\ast}, \widetilde{b}_{(0,0)})\\
        &=-q^{-1}c^{M_q^{\otimes 2}}(b_{(2,-1)}^{\ast}\otimes b_{(-1,1)}^{\ast}, b_{(1,0)}\otimes b_{(-1,0)})\\
        &\phantom{==}-q^{-2}c^{M_q^{\otimes 2}}(b_{(2,-1)}^{\ast}\otimes b_{(-1,1)}^{\ast}, 
        b_{(-1,1)}\otimes b_{(1,-1)})\\
        &\phantom{==}+qc^{M_q^{\otimes 2}}(b_{(2,-1)}^{\ast}\otimes b_{(-1,1)}^{\ast}, 
        b_{(2,-1)}\otimes b_{(-2,1)})\\
        &\phantom{==}-q^2(q-q^{-1})c^{M_q^{\otimes 2}}(b_{(2,-1)}^{\ast}\otimes b_{(-1,1)}^{\ast}, 
        \frac{1}{[2]_q}b_{(0,0)}\otimes b_{(0,0)})\\
        &\phantom{==}-qc^{M_q^{\otimes 2}}(b_{(2,-1)}^{\ast}\otimes b_{(-1,1)}^{\ast}, b_{(-2,1)}\otimes b_{(2,-1)})\\
        &\phantom{==}+q^4c^{M_q^{\otimes 2}}(b_{(2,-1)}^{\ast}\otimes b_{(-1,1)}^{\ast}, b_{(1,-1)}\otimes b_{(-1,1)})\\
        &\phantom{==}+q^3c^{M_q^{\otimes 2}}(b_{(2,-1)}^{\ast}\otimes b_{(-1,1)}^{\ast}, b_{(-1,0)}\otimes b_{(1,0)}).
\end{align*}
Hence, by \eqref{eq:product}, it suffices to show that 
\[
c^{M_q^{\otimes 2}}(b_{(2,-1)}^{\ast}\otimes b_{(-1,1)}^{\ast}, 
        \frac{1}{[2]_q}b_{(0,0)}\otimes b_{(0,0)})\in R'_{\AAA_{1/2}}. 
\]
A direct calculation shows that 
\begin{align*}
\widetilde{\mathbf{1}} 
        &=q^{-6}b_{(1,0)}\otimes b_{(-1,0)}
        -q^{-5}b_{(-1,1)}\otimes b_{(1,-1)}\\
        &\phantom{==}+q^{-2}b_{(2,-1)}\otimes b_{(-2,1)}
        -\frac{1}{[2]_q}b_{(0,0)}\otimes b_{(0,0)}\\
        &\phantom{==}+b_{(-2,1)}\otimes b_{(2,-1)}
        -q^3b_{(1,-1)}\otimes b_{(-1,1)}
        +q^{4}b_{(-1,0)}\otimes b_{(1,0)}
\end{align*}
satisfies $X_1^+.\widetilde{\mathbf{1}}=X_2^+.\widetilde{\mathbf{1}}=0$ and $\wt \widetilde{\mathbf{1}}=0$ in $M_q^{\otimes 2}$. Therefore, $\widetilde{\mathbf{1}}$ spans a trivial $U_q(\g)$-submodule of $M_q^{\otimes 2}$. Thus,
we have 
\[
c^{M_q^{\otimes 2}}(b_{(2,-1)}^{\ast}\otimes b_{(-1,1)}^{\ast}, 
        \widetilde{\mathbf{1}})=0.
\]
Hence,
\begin{align*}
    &c^{M_q^{\otimes 2}}(b_{(2,-1)}^{\ast}\otimes b_{(-1,1)}^{\ast}, 
        \frac{1}{[2]_q}b_{(0,0)}\otimes b_{(0,0)})\\
        &=
        c^{M_q^{\otimes 2}}(b_{(2,-1)}^{\ast}\otimes b_{(-1,1)}^{\ast}, q^{-6}b_{(1,0)}\otimes b_{(-1,0)}
        -q^{-5}b_{(-1,1)}\otimes b_{(1,-1)}\\
        &\phantom{==}+q^{-2}b_{(2,-1)}\otimes b_{(-2,1)}+b_{(-2,1)}\otimes b_{(2,-1)}
        -q^3b_{(1,-1)}\otimes b_{(-1,1)}
        +q^{4}b_{(-1,0)}\otimes b_{(1,0)}),
\end{align*}
and the right-hand side belongs to $R'_{\AAA_{1/2}}$ by \eqref{eq:product}. This shows $c^{M_q}(b_{(1,0)}^{\ast}, b_{(0,0)})\in R'_{\AAA_{1/2}}$. One can show $c_{b^{\ast}, b_{(0,0)}}\in R'_{\AAA_{1/2}}$ for $b\in \BBB_M\setminus \{b_{(0,0)}\}$ in the same manner, where $b_{(2,-1)}^{\ast}\otimes b_{(-1,1)}^{\ast}$ is replaced by any chosen term $b_{(m_2,n_2)}^{\ast}\otimes b_{(m_1,n_1)}^{\ast}$ such that $(m_1,n_1),(m_2,n_2)\neq (0,0)$ and $b_{(m_1,n_1)}\otimes b_{(m_2,n_2)}$ appears in the expansion of $\widetilde{b}$. 

Next, we prove $c^{M_q}(b_{(0,0)}^{\ast}, b_{(0,0)})\in R'_{\AAA_{1/2}}$. 
By the same argument as above, 
    \[
    c^{M_q}(b_{(0,0)}^{\ast}, b_{(0,0)})=\frac{1}{[2]_q}c^{M_q^{\otimes 2}}(b_{(-1,0)}^{\ast}\otimes b_{(1,0)}^{\ast}, \widetilde{b}_{(0,0)}),
    \]
    and we need to show that 
    \[
    c^{M_q^{\otimes 2}}(b_{(-1,0)}^{\ast}\otimes b_{(1,0)}^{\ast}, 
        \frac{1}{[2]_q}b_{(0,0)}\otimes b_{(0,0)})\in R'_{\AAA_{1/2}}. 
    \]
    We have 
    \[
    c^{M_q^{\otimes 2}}(b_{(-1,0)}^{\ast}\otimes b_{(1,0)}^{\ast}, 
        \widetilde{\mathbf{1}})=q^{-6}.
    \]
    Hence, 
    \begin{align*}
    &c^{M_q^{\otimes 2}}(b_{(-1,0)}^{\ast}\otimes b_{(1,0)}^{\ast}, 
        \frac{1}{[2]_q}b_{(0,0)}\otimes b_{(0,0)})\\
        &=
        c^{M_q^{\otimes 2}}(b_{(-1,0)}^{\ast}\otimes b_{(1,0)}^{\ast}, q^{-6}b_{(1,0)}\otimes b_{(-1,0)}
        -q^{-5}b_{(-1,1)}\otimes b_{(1,-1)}\\
        &\phantom{==}+q^{-2}b_{(2,-1)}\otimes b_{(-2,1)}+b_{(-2,1)}\otimes b_{(2,-1)}
        -q^3b_{(1,-1)}\otimes b_{(-1,1)}
        +q^{4}b_{(-1,0)}\otimes b_{(1,0)})\\
        &\phantom{==}-q^{-6},
\end{align*}
and the right-hand side belongs to $R'_{\AAA_{1/2}}$. This shows $c^{M_q}(b_{(0,0)}^{\ast}, b_{(0,0)})\in R'_{\AAA_{1/2}}$. 

Finally, we prove $c^{M_q}(b_{(0,0)}^{\ast}, b_{(1,0)})\in R'_{\AAA_{1/2}}$. 
By the same argument as above, 
    \begin{align*}
        &c^{M_q}(b_{(0,0)}^{\ast}, b_{(1,0)})\\
        &=\frac{1}{[2]_q}c^{M_q^{\otimes 2}}(b_{(-1,0)}^{\ast}\otimes b_{(1,0)}^{\ast}, \widetilde{b}_{(1,0)})\\
        &=\frac{1}{[2]_q}c^{M_q^{\otimes 2}}(b_{(-1,0)}^{\ast}\otimes b_{(1,0)}^{\ast},  b_{(1,0)}\otimes b_{(0,0)}- q[2]_qb_{(-1,1)}\otimes b_{(2,-1)}\\
        &\phantom{==}+q^4[2]_qb_{(2,-1)}\otimes b_{(-1,1)}-q^6b_{(0,0)}\otimes b_{(1,0)}),
    \end{align*}
    and we have already shown that the right-hand side belongs to $R'_{\AAA_{1/2}}$. 
    One can show $c_{b_{(0,0)}^{\ast}, b}\in R'_{\AAA_{1/2}}$ for $b\in \BBB_M\setminus \{b_{(0,0)}\}$ in the same manner. This completes the proof of the theorem in the $G_2$ case. 
\end{proof}
\begin{rem}
If we consider the parallel argument in the classical case, $b_{(0,0)}\otimes b_{(0,0)}$ does not occur in the expansion of $\widetilde{b}_{(0,0)}$. Thus, we do not need the argument involving the trivial $U_q(\g)$-submodule of $M_q^{\otimes 2}$. Hence, the proof is much simpler and it is the argument adopted in the paper \cite{Oya}. Such an increase in terms is a major difficulty in the quantum case. Indeed, also in the case of type $E_8$, the argument becomes intricate in the quantum setting due to this problem. See Remark \ref{r:A=O-classical} and Theorem \ref{t:adj-minor-classical}.
\end{rem}
\begin{proof}[{Proof of Theorem \ref{thm:integral-coord-ring}}]
    The assertion immediately follows from Theorems \ref{t:integral-gen} and \ref{t:EG-gamma} together with the remark just before Theorem \ref{t:EG-gamma}.
\end{proof}
\subsection{Matrix coefficients of the quantum adjoint representation}
Throughout this subsection, we assume that the rank $r$ of $\g$ is greater than $1$.
Recall from Section \ref{2.2} that $\theta$ denotes the highest root of $\g$. In this subsection, we prove the following. 
\begin{thm}\label{t:adj-minor} Assume that the rank $r$ of $\g$ is greater than $1$. The elements of 
    \[
    \Upsilon_0\coloneqq \{c^{V_q(\theta)}(b^{\ast}, b')\mid b, b'\in \BBB(\theta),\ \wt b=0\ \text{or}\ \wt b'=0\}
    \]
    can be written as polynomials in 
    \[
    \Upsilon_{\neq 0}\coloneqq \{c^{V_q(\theta)}(b^{\ast}, b')\mid b, b'\in \BBB(\theta),\ \wt b\neq 0, \wt b'\neq 0\}.
    \]
    with coefficients in $\KK$. 
\end{thm}
As explained in the beginning of the proof of Theorem \ref{t:EG-gamma}, this theorem implies Theorem \ref{t:EG-gamma} in the case of type $E_8$. 

The structure of the based module $(V_q(\theta), \BBB(\theta))$ is explicitly described as follows. 

\begin{theorem}[{\cite[Theorem 0.6]{Lusztig-qadjoint}}]\label{t:qadjoint-str} The quantum adjoint representation $V_q(\theta)$ has a unique $\KK$-basis $\BBB=\{X_{\alpha}\mid \alpha\in \Phi\}\sqcup \{t_i\mid i\in [1,r]\}$ satisfying the following properties: 
\begin{itemize}
    \item[(i)] $X_{\theta}=v_{\theta}$.
    \item[(ii)] $X_{\alpha}\in V_q(\theta)_{\alpha}$ for $\alpha\in \Phi$, and $t_i\in V_q(\theta)_{0}$ for $i\in [1,r]$.
    \item[(iii)] For $i\in [1, r]$, the actions of $X_i^+$ and $X_i^-$ are given as follows: 
    \begin{itemize}
        \item[] $X_i^+. X_{\alpha}=[q_{i, \alpha}+1]_iX_{\alpha+\alpha_i}$ if $\alpha\in \Phi$ and $p_{i, \alpha}>0$, 
        \item[] $X_i^+. X_{-\alpha_i}=t_i$, 
        \item[] $X_i^+. X_{\alpha}=0$ if $\alpha\in \Phi$, $p_{i, \alpha}=0$, and $\alpha\neq -\alpha_i$, 
        \item[] $X_i^+. t_j=[|c_{ji}|]_jX_{\alpha_i}$ if $j\in [1, r]$,         
        \item[] $X_i^-. X_{\alpha}=[p_{i, \alpha}+1]_iX_{\alpha-\alpha_i}$ if $\alpha\in \Phi$ and $q_{i, \alpha}>0$, 
        \item[] $X_i^-. X_{\alpha_i}=t_i$, 
        \item[] $X_i^-. X_{\alpha}=0$ if $\alpha\in \Phi$, $q_{i, \alpha}=0$, and $\alpha\neq \alpha_i$, 
        \item[] $X_i^-. t_j=[|c_{ji}|]_jX_{-\alpha_i}$ if $j\in [1, r]$,         
    \end{itemize}
     where 
     \begin{align*}
      p_{i, \alpha}&:=\max\big(\{0\}\cup \{p\in \Z_{>0}\mid \alpha+k\alpha_i\in \Phi\text{ for }k=1,\dots, p\}\big),\\
      q_{i, \alpha}&:=\max\big(\{0\}\cup \{q\in \Z_{>0}\mid \alpha-k\alpha_i\in \Phi\text{ for }k=1,\dots, q\}\big). 
     \end{align*}
\end{itemize}
     Moreover, $\BBB=\BBB(\theta)$. 
\end{theorem}
Throughout this subsection, we will use the notation in Theorem \ref{t:qadjoint-str} for the elements of $\BBB(\theta)$, and its dual basis will be written as $\BBB(\theta)^{\ast}=\{X_{\alpha}^{\ast}\mid \alpha\in \Phi\}\sqcup \{t_i^{\ast}\mid i\in [1,r]\}$ as before. 
By the formula in Theorem \ref{t:qadjoint-str}, we have
\begin{align}
    X_i^+. t_j^{\ast}&=-\delta_{ij}X_{-\alpha_i}^{\ast},&
    X_i^+. X_{\alpha_j}^{\ast}&=-\delta_{ij}q_j^{-2}\sum_{k\in [1,r]}[|c_{kj}|]_kt_k^{\ast}\label{eq:action-ex}
\end{align}
for $i, j\in [1,r]$. 

Before proving Theorem \ref{t:adj-minor}, we prepare two technical lemmas (Lemmas \ref{l:symm} and \ref{l:psibeta}). 

Since the Lie bracket gives a $\g$-module homomorphism $V(\theta)\otimes V(\theta)\to V(\theta)$, we have $[V_q(\theta)\otimes V_q(\theta) : V_q(\theta)]\geq 1$ by \eqref{mult}. We fix a $U_q(\g)$-module homomorphism $m\colon V_q(\theta)\otimes V_q(\theta) \to V_q(\theta)$. By taking dual, we have $\iota\coloneqq m^{\ast}\colon V_q(\theta)^{\ast}\to V_q(\theta)^{\ast}\otimes V_q(\theta)^{\ast}$. For $i\in [1,r]$, we write 
\begin{multline*}
\iota(t_i^{\ast})=\sum_{k,l\in [1,r]}a_{kl}^{(i)}t_k^{\ast}\otimes t_l^{\ast}+\sum_{k\in [1,r]}b_k^{(i)}X_{\alpha_k}^{\ast}\otimes X_{-\alpha_k}^{\ast}
\\
+\sum_{k\in [1,r]}c_k^{(i)}X_{-\alpha_k}^{\ast}\otimes X_{\alpha_k}^{\ast}+(\mbox{other terms}),
\end{multline*}
where $(\mbox{other terms}) \in \bigoplus_{\beta\in \Phi\setminus (\Pi\cup -\Pi)}(V_q(\theta)^{\ast})_{-\beta}\otimes (V_q(\theta)^{\ast})_{\beta}$. Set
\[
A_{\iota}^{(i)}\coloneqq (a_{kl}^{(i)})_{k, l\in [1,r]}\in \Mat_{r\times r}(\KK). 
\]
\begin{lem}\label{l:symm}
The matrix $A_{\iota}^{(i)}$ is symmetric for all $i\in [1,r]$.  
\end{lem}
\begin{proof}
    Fix $i\in [1,r]$. For $j\in [1, r]\setminus \{i\}$, \eqref{eq:action-ex} implies 
    \begin{align*}
        0&=X_j^+\cdot \iota(t_i^{\ast})\\
        &=-\sum_{l\in [1,r]}a_{jl}^{(i)}X_{-\alpha_j}^{\ast}\otimes t_l^{\ast}
        -
        \sum_{k\in [1,r]}a_{kj}^{(i)}t_k^{\ast}\otimes X_{-\alpha_j}^{\ast}\\
        &\phantom{==}
        -q_j^{-2}\sum_{k\in [1,r]}b_j^{(i)}[|c_{kj}|]_kt_k^{\ast}
        \otimes X_{-\alpha_j}^{\ast}
        -\sum_{l\in [1,r]}
        c_j^{(i)}[|c_{lj}|]_lX_{-\alpha_j}^{\ast}\otimes t_l^{\ast}+(\mbox{other terms}).
    \end{align*}
    Here $(\mbox{other terms}) \in \bigoplus_{\substack{\beta_1, \beta_2\in \Phi;\\ \beta_1+\beta_2=\alpha_j}}(V_q(\theta)^{\ast})_{\beta_1}\otimes (V_q(\theta)^{\ast})_{\beta_2}$. Therefore, 
    \[
    \begin{cases}
    -a_{kj}^{(i)}-q_j^{-2}b_j^{(i)}[|c_{kj}|]_k=0&\text{for $j\in [1, r]\setminus \{i\}$ and $k\in [1,r]$},\\
        -a_{jl}^{(i)}-c_j^{(i)}[|c_{lj}|]_l=0&\text{for $j\in [1, r]\setminus \{i\}$ and $l\in [1,r]$}.
    \end{cases}
    \]
    For $j\in [1, r]\setminus \{i\}$, we have $
    q_j^{-2}b_j^{(i)}[2]_j=-a_{jj}^{(i)}=c_j^{(i)}[2]_j$, hence 
    \begin{align*}
        q_j^{-2}b_j^{(i)}=c_j^{(i)}.
    \end{align*}
    This implies 
    \begin{align*}
        a_{jk}^{(i)}=-c_j^{(i)}[|c_{kj}|]_k=-q_j^{-2}b_j^{(i)}[|c_{kj}|]_k=a_{kj}^{(i)}\tag{$(\mathrm{eq})_{jk}$} 
    \end{align*}
    for $j\in [1, r]\setminus \{i\}$ and $k\in [1,r]$. 
    
    We need to show that $a_{jk}^{(i)}=a_{kj}^{(i)}$ for all $j, k\in [1, r]$ with $j\neq k$. When $j\neq i$, $(\mathrm{eq})_{jk}$ implies $a_{jk}^{(i)}=a_{kj}^{(i)}$, and when $j=i$, we have $k\neq j=i$ and $(\mathrm{eq})_{kj}$ implies $a_{jk}^{(i)}=a_{kj}^{(i)}$. 
\end{proof}
\begin{rem}\label{r:A=O-classical}
    In the classical case, we can take $m$ to be (a scalar multiple) of the Lie bracket and consider $A_{m^{\ast}}^{(i)}$ in the same way. In this case, $A_{m^{\ast}}^{(i)}$ is the zero matrix since $[\h,\h]=0$, cf.~\cite[Eq. (3.2)]{Oya}.
\end{rem}
For $\beta\in \Phi$, set 
\[
\Psi(\beta)\coloneqq \{(\beta_1, \beta_2)\in (\wt V_q(\theta))^{\times 2}\mid \beta=\beta_1+\beta_2\}. 
\]
\begin{lem}\label{l:psibeta}
    Let $\beta\in \Phi\setminus \{\theta, -\theta\}$. Then there exists $(\beta_1^\circ, \beta_2^\circ)\in \Psi(\beta)$ satisfying the following.  
    \begin{itemize}
        \item[(i)] $\beta_1^\circ \not \leq \beta_1$ for all $(\beta_1, \beta_2)\in \Psi(\beta)\setminus \{(\beta_1^\circ, \beta_2^\circ)\}$. 
        \item[(ii)] $\beta_1^\circ\neq 0$ and $\beta_2^\circ\neq 0$. 
    \end{itemize}
\end{lem}
\begin{proof}
    Consider the partial ordering on $\Psi(\beta)$ defined as follows.  
    \[
    (\beta_1, \beta_2)\geq (\beta'_1, \beta'_2)\quad \Leftrightarrow\quad \beta_1\geq \beta'_1\ \text{and}\ \beta_2\leq \beta'_2.
    \]
    First, assume that $\beta\in \Phi_+\setminus \{\theta\}$. Since $\beta\neq \theta$, there exists $i\in [1, r]$ such that $\beta+\alpha_i\in \Phi_+$. Then $(\beta+\alpha_i, -\alpha_i)\in \Psi(\beta)$. Hence, if we take a maximal element $(\beta_1^\circ, \beta_2^\circ)\in \Psi(\beta)$ satisfying $(\beta_1^\circ, \beta_2^\circ)\geq (\beta+\alpha_i, -\alpha_i)$, then $(\beta_1^\circ, \beta_2^\circ)$ satisfies (i) and (ii). 
    
    Next, assume that $\beta\in -\Phi_+\setminus \{-\theta\}$. Since $\beta\neq -\theta$, there exists $i\in [1, r]$ such that $\beta-\alpha_i\in -\Phi_+$. Then $(\alpha_i, \beta-\alpha_i)\in \Psi(\beta)$. Hence, if we take a maximal element $(\beta_1^\circ, \beta_2^\circ)\in \Psi(\beta)$ satisfying $(\beta_1^\circ, \beta_2^\circ)\geq (\alpha_i, \beta-\alpha_i)$, then $(\beta_1^\circ, \beta_2^\circ)$ satisfies (i) and (ii). 
\end{proof}
\begin{proof}[Proof of Theorem \ref{t:adj-minor}]
    Let $R''$ be the $\KK$-subalgebra of $R_q[G]$ generated by $\Upsilon_{\neq 0}$. It suffices to show that $\Upsilon_0\subset R''$. Consider the $U_q(\g)$-module isomorphism
    \[
    \mathcal{R}\colon V_q(\theta)\otimes V_q(\theta)\xrightarrow[]{\sim} V_q(\theta)\otimes V_q(\theta),
    \]
    given by the action of the universal $R$-matrix of $U_q(\g)$, see \cite[Theorem 7.3]{jantzen1996lectures}. By definition, 
    \[
    \mathcal{R}=\Theta\circ \widetilde{f}\circ P,
    \]
    where 
    \begin{itemize}\setlength{\itemsep}{3pt}
        \item $P\colon V_q(\theta)\otimes V_q(\theta)\to V_q(\theta)\otimes V_q(\theta),\ v_1\otimes v_2\mapsto v_2\otimes v_1$,
        \item $\widetilde{f} \colon V_q(\theta)\otimes V_q(\theta)\to V_q(\theta)\otimes V_q(\theta),\ v_1\otimes v_2\mapsto q^{-(\wt v_1, \wt v_2)}v_1\otimes v_2$,
        \item $\Theta \colon V_q(\theta)\otimes V_q(\theta)\to V_q(\theta)\otimes V_q(\theta),$
        \[
        v_1\otimes v_2\mapsto \sum_{\alpha\in Q^+}\Theta_{\alpha}. v_1\otimes v_2
        \]
        for some element $\Theta_{\alpha}\in U_q^-(\g)_{-\alpha}\otimes U_q^+(\g)_{\alpha}$ satisfying $\Theta_0=1\otimes 1$. See \cite[Chapter 7]{jantzen1996lectures} for the explicit construction of $\Theta_{\alpha}$.
    \end{itemize}     
    Let $\beta\in \Phi\setminus \{ \pm \theta\}$. We first show  $c^{V_q(\theta)}(t_i^{\ast}, X_{\beta})\in R''$ for $i\in [1, r]$. 
    Fix an element $(\beta_1^\circ, \beta_2^\circ)\in \Psi(\beta)$ with the properties in Lemma \ref{l:psibeta}. Then, 
    \[
    \mathcal{R}(X_{\beta_1^\circ}\otimes X_{\beta_2^\circ})=q^{-(\beta_1^\circ, \beta_2^\circ)}X_{\beta_2^\circ}\otimes X_{\beta_1^\circ}.
    \]
    Since $[V_q(\theta)\otimes V_q(\theta) : V_q(\theta)]=[V(\theta)\otimes V(\theta) : V(\theta)]$, we can take a $U_q(\g)_\AAA$-module homomorphism 
    \[
    m_{\AAA}\colon V_q(\theta)_{\AAA}\otimes_{\AAA} V_q(\theta)_{\AAA}\to V_q(\theta)_{\AAA}
    \]
    such that its specialization $m_{\AAA}|_{q=1}\colon V(\theta)\otimes V(\theta) \to V(\theta)$ at $q=1$ coincides with a non-zero scalar multiple of the Lie bracket, cf.~\cite[Proposition 31.2.6]{Lus:intro}. Then, by \cite[Proposition 8.4 (d)]{Humphreys72}, 
    \begin{align}
        m_{\AAA}(X_{\beta_1^\circ}\otimes X_{\beta_2^\circ})=c_1X_{\beta},\quad m_{\AAA}(q^{-(\beta_1^\circ, \beta_2^\circ)}X_{\beta_2^\circ}\otimes X_{\beta_1^\circ})=c_2X_{\beta}\label{eq:q-bracket}
    \end{align}
for some $c_1, c_2\in \AAA$ such that $c_1|_{q=1}=-c_2|_{q=1}\neq 0$. Thus,
\[
m_{\AAA}(X_{\beta_1^\circ}\otimes X_{\beta_2^\circ}-q^{-(\beta_1^\circ, \beta_2^\circ)}X_{\beta_2^\circ}\otimes X_{\beta_1^\circ})=(c_1-c_2)X_{\beta}\neq 0.
\]
Set $c\coloneqq c_1-c_2\in \AAA$. Then,
\begin{align*}
    &c^{V_q(\theta)}(t_i^{\ast}, X_{\beta})\\
    &=c^{-1}c^{V_q(\theta)}\left(t_i^{\ast}, m_{\AAA}(X_{\beta_1^\circ}\otimes X_{\beta_2^\circ}-q^{-(\beta_1^\circ, \beta_2^\circ)}X_{\beta_2^\circ}\otimes X_{\beta_1^\circ})\right)\\
    &=c^{-1}c^{V_q(\theta)^{\otimes 2}}\left(m_{\AAA}^{\ast}(t_i^{\ast}), X_{\beta_1^\circ}\otimes X_{\beta_2^\circ}-q^{-(\beta_1^\circ, \beta_2^\circ)}X_{\beta_2^\circ}\otimes X_{\beta_1^\circ}\right).
\end{align*}
Write
\begin{align}
m_{\AAA}^{\ast}(t_i^{\ast})=\sum_{k,l\in [1,r]}a_{kl}^{(i)}t_k^{\ast}\otimes t_l^{\ast}+(\mbox{other term})\in V_q(\theta)^{\ast}\otimes V_q(\theta)^{\ast}.\label{eq:expansion}
\end{align}
Here $(\mbox{other terms}) \in \bigoplus_{\beta\in \Phi}(V_q(\theta)^{\ast})_{-\beta}\otimes (V_q(\theta)^{\ast})_{\beta}$. Then, by \eqref{eq:product}, $c^{V_q(\theta)}(t_i^{\ast}, X_{\beta})\in R''$ can be shown by proving that 
\begin{align}
c^{V_q(\theta)^{\otimes 2}}\left(\sum_{k,l\in [1,r]}a_{kl}^{(i)}t_k^{\ast}\otimes t_l^{\ast}, X_{\beta_1^\circ}\otimes X_{\beta_2^\circ}-q^{-(\beta_1^\circ, \beta_2^\circ)}X_{\beta_2^\circ}\otimes X_{\beta_1^\circ}\right)\in R''.\label{eq:key}    
\end{align}
We will show \eqref{eq:key}. For $x\in U_q(\g)$, we have
\begin{align*}
    &\left\langle \sum_{k,l\in [1,r]}a_{kl}^{(i)}t_k^{\ast}\otimes t_l^{\ast}, \Delta(x). (q^{-(\beta_1^\circ, \beta_2^\circ)}X_{\beta_2^\circ}\otimes X_{\beta_1^\circ})\right\rangle\\
    &=\left\langle \sum_{k,l\in [1,r]}a_{kl}^{(i)}t_k^{\ast}\otimes t_l^{\ast}, \Delta(x). \mathcal{R}(X_{\beta_1^\circ}\otimes X_{\beta_2^\circ})\right\rangle\\
    &=\left\langle \sum_{k,l\in [1,r]}a_{kl}^{(i)}t_k^{\ast}\otimes t_l^{\ast}, \mathcal{R}(\Delta(x). X_{\beta_1^\circ}\otimes X_{\beta_2^\circ})\right\rangle\\
    &=\sum_{\alpha\in Q^+}\left\langle \sum_{k,l\in [1,r]}a_{kl}^{(i)}(S^{-1}\otimes S^{-1})(\Theta_{\alpha}). (t_k^{\ast}\otimes t_l^{\ast}), \widetilde{f}(P(\Delta(x). X_{\beta_1^\circ}\otimes X_{\beta_2^\circ}))\right\rangle\\
    &=\sum_{\alpha\in Q^+}q^{(\alpha, \alpha)}\left\langle \sum_{k,l\in [1,r]}a_{kl}^{(i)}(S^{-1}\otimes S^{-1})(\Theta_{\alpha}). (t_k^{\ast}\otimes t_l^{\ast}), P(\Delta(x). X_{\beta_1^\circ}\otimes X_{\beta_2^\circ})\right\rangle\\
    &=\left\langle \sum_{k,l\in [1,r]}a_{kl}^{(i)}t_l^{\ast}\otimes t_k^{\ast}, \Delta(x). X_{\beta_1^\circ}\otimes X_{\beta_2^\circ}\right\rangle\\
    &\phantom{=}+\sum_{0\neq \alpha\in Q^+}q^{(\alpha, \alpha)}\left\langle \sum_{k,l\in [1,r]}a_{kl}^{(i)}P^{\ast}((S^{-1}\otimes S^{-1})(\Theta_{\alpha}). (t_k^{\ast}\otimes t_l^{\ast})), \Delta(x). X_{\beta_1^\circ}\otimes X_{\beta_2^\circ}\right\rangle.
\end{align*}
Therefore, we have 
\begin{align*}
&c^{V_q(\theta)^{\otimes 2}}\left(\sum_{k,l\in [1,r]}a_{kl}^{(i)}t_k^{\ast}\otimes t_l^{\ast}, q^{-(\beta_1^\circ, \beta_2^\circ)}X_{\beta_2^\circ}\otimes X_{\beta_1^\circ}\right)\\
&=c^{V_q(\theta)^{\otimes 2}}\left(\sum_{k,l\in [1,r]}a_{kl}^{(i)}t_l^{\ast}\otimes t_k^{\ast},  X_{\beta_1^\circ}\otimes X_{\beta_2^\circ}\right)+\\
&\phantom{=}\sum_{0\neq \alpha\in Q^+}q^{(\alpha, \alpha)}c^{V_q(\theta)^{\otimes 2}}\left(\sum_{k,l\in [1,r]}a_{kl}^{(i)}P^{\ast}((S^{-1}\otimes S^{-1})(\Theta_{\alpha}). (t_k^{\ast}\otimes t_l^{\ast})),  X_{\beta_1^\circ}\otimes X_{\beta_2^\circ}\right)\\
&=c^{V_q(\theta)^{\otimes 2}}\left(\sum_{k,l\in [1,r]}a_{kl}^{(i)}t_k^{\ast}\otimes t_l^{\ast},  X_{\beta_1^\circ}\otimes X_{\beta_2^\circ}\right)+\\
&\phantom{=}\sum_{0\neq \alpha\in Q^+}q^{(\alpha, \alpha)}c^{V_q(\theta)^{\otimes 2}}\left(\sum_{k,l\in [1,r]}a_{kl}^{(i)}P^{\ast}((S^{-1}\otimes S^{-1})(\Theta_{\alpha}). (t_k^{\ast}\otimes t_l^{\ast})),  X_{\beta_1^\circ}\otimes X_{\beta_2^\circ}\right).
\end{align*}
Here the second equality follows from Lemma \ref{l:symm}. Thus, 
\begin{align*}
    &c^{V_q(\theta)^{\otimes 2}}\left(\sum_{k,l\in [1,r]}a_{kl}^{(i)}t_k^{\ast}\otimes t_l^{\ast}, X_{\beta_1^\circ}\otimes X_{\beta_2^\circ}-q^{-(\beta_1^\circ, \beta_2^\circ)}X_{\beta_2^\circ}\otimes X_{\beta_1^\circ}\right)\\
    &=-\sum_{0\neq \alpha\in Q^+}q^{(\alpha, \alpha)}c^{V_q(\theta)^{\otimes 2}}\left(\sum_{k,l\in [1,r]}a_{kl}^{(i)}P^{\ast}((S^{-1}\otimes S^{-1})(\Theta_{\alpha})\cdot (t_k^{\ast}\otimes t_l^{\ast})),  X_{\beta_1^\circ}\otimes X_{\beta_2^\circ}\right),
\end{align*}
and the right hand side belongs to $R''$, which proves \eqref{eq:key}.

Next, we show that $c^{V_q(\theta)}(t_i^{\ast}, X_{\theta})\in R''$ for $i\in [1,r]$. Since $\theta \in\Phi_+\setminus \Pi$, there exists $j\in [1,r]$ such that $\theta-\alpha_j\in \Phi_+$. Note that the rank $r$ of $\g$ is greater than $1$. Then,
\[
m_{\AAA}(X_{\alpha_j}\otimes X_{\theta-\alpha_j})=c_{\theta}X_{\theta}
\]
for some $0\neq c_{\theta}\in \AAA$. Thus,
\begin{align*}
    c^{V_q(\theta)}(t_i^{\ast}, X_{\theta})
    &=c_{\theta}^{-1}c^{V_q(\theta)}\left(t_i^{\ast}, m_{\AAA}(X_{\alpha_j}\otimes X_{\theta-\alpha_j})\right)\\
    &=c_{\theta}^{-1}c^{V_q(\theta)^{\otimes 2}}\left(m_{\AAA}^{\ast}(t_i^{\ast}), X_{\alpha_j}\otimes X_{\theta-\alpha_j}\right).
\end{align*}
We have already shown that the right-hand side belongs to $R''$. Hence $c^{V_q(\theta)}(t_i^{\ast}, X_{\theta})\in R''$. The statement that $c^{V_q(\theta)}(t_i^{\ast}, X_{-\theta})\in R''$ for $i\in [1,r]$ can be shown in the parallel manner. 

Finally, we prove that $c^{V_q(\theta)}(b^{\ast}, t_i)\in R''$ for all $b\in \BBB(\theta)$ and $i\in [1,r]$. Since $m_\AAA |_{q=1}(\bigoplus_{\beta\in \Phi}V(\theta)_{\beta}\otimes V(\theta)_{-\beta})=V(\theta)_0$, we have $m_\KK(\bigoplus_{\beta\in \Phi}V_q(\theta)_{\beta}\otimes V_q(\theta)_{-\beta})=V_q(\theta)_0$. Here $m_\KK$ is the $U_q(\g)$-module homomorphism $V_q(\theta)\otimes V_q(\theta)\to V_q(\theta)$ induced from $m_{\AAA}$. Hence, for $i\in [1, r]$, there exists $\widetilde{t}_i=\sum_{\beta\in \Phi}a_{\beta}^{(i)}X_{\beta}\otimes X_{-\beta}\in \bigoplus_{\beta\in \Phi}V_q(\theta)_{\beta}\otimes V_q(\theta)_{-\beta}$ such that $m_\KK(\widetilde{t}_i)=t_i$. Then, for $b\in \BBB(\theta)$, 
\begin{align*}
    c^{V_q(\theta)}(b^{\ast}, t_i)
    &=c^{V_q(\theta)}\left(b^{\ast}, m_{\KK}(\widetilde{t}_i)\right)
    =\sum_{\beta\in \Phi}a_{\beta}^{(i)} c^{V_q(\theta)}\left(m_{\KK}^{\ast}(b^{\ast}), X_{\beta}\otimes X_{-\beta}\right).
\end{align*}
We have already shown that the right-hand side belongs to $R''$, which completes the proof of the theorem. 
\end{proof}
\subsection{Matrix coefficients of the adjoint representation}
We finally prove Theorem \ref{thm:integral-coord-ring-classical} which implies Theorem \ref{thm:integral-cluster-G-classical}. 
 By abuse of notation we will denote by $\BBB(\theta)=\{X_{\alpha}\mid \alpha\in \Phi\}\sqcup \{t_i\mid i\in [1,r]\}$ (resp.~$\BBB(\theta)^{\ast}=\{X_{\alpha}^{\ast}\mid \alpha\in \Phi\}\sqcup \{t_i^{\ast}\mid i\in [1,r]\}$) the $\C$-basis of the $\g$-module $V(\theta)$ (resp.~$V(\theta)^{\ast}$) induced from the $\AAA$-basis $\BBB(\theta)$ of $V_q(\theta)_{\AAA}$ (resp.~$V_q(\theta)\spcheck_\AAA$).  
\begin{thm}\label{t:adj-minor-classical}
    Assume that $\g$ is of simply-laced type of rank greater than $1$. As before, $\theta$ denotes the highest root of $\g$. Then the elements of 
    \[
    \Upsilon_0^{q=1}\coloneqq \{c^{V(\theta)}(b^{\ast}, b')\mid b, b'\in \BBB(\theta),\ \wt b=0\ \text{or}\ \wt b'=0\}
    \]
    can be written as polynomials in 
    \[
    \Upsilon_{\neq 0}^{q=1}\coloneqq \{c^{V(\theta)}(b^{\ast}, b')\mid b, b'\in \BBB(\theta),\ \wt b\neq 0, \wt b'\neq 0\}.
    \]
    with coefficients in $\Z$. 
\end{thm}
\begin{proof}
    Let $R''_{\Z}$ be the $\Z$-subalgebra of $R[G]_{\Z}$ generated by $\Upsilon_{\neq 0}^{q=1}$. It suffices to show that $\Upsilon_0^{q=1}\subset R''_{\Z}$. Since $V(\theta)$ is isomorphic to the adjoint representation of $\g$, we have an isomorphism $\varphi\colon V(\theta)\xrightarrow[]{\sim} \g$ of $\g$-modules with the adjoint action on the second. Hence, the set $\varphi(\BBB(\theta))$ is a $\C$-basis of $\g$. Write 
\begin{align*}
\g_{\Z}&\coloneqq \sum_{b\in \BBB(\theta)}\Z \varphi(b),&\g_{\Z,\neq 0}&\coloneqq \sum_{\beta\in \Phi}\Z \varphi(X_{\beta}),\\
\g\spcheck_{\Z}&\coloneqq\sum_{b\in \BBB(\theta)}\Z (\varphi^{-1})^{\ast}(b^{\ast})&(\g\spcheck_{\Z})_{\neq 0}&\coloneqq\sum_{\beta\in \Phi}\Z (\varphi^{-1})^{\ast}(X_{\beta}^{\ast}).&
\end{align*}
By \cite[Theorem 5.7]{Geck2017}, $\g_{\Z}$ is closed under the Lie bracket of $\g$. Moreover, by \cite[Lemma 5.1 and Theorem 5.7]{Geck2017}, we have the following: 
\begin{itemize}
    \item[(i)] For all $\beta\in \Phi$, there exist $i\in [1,r]$ and $\epsilon, \epsilon'\in \{1,-1\}$ such that $\beta-\epsilon\alpha_{i}\in \Phi$ and $    [\varphi(\epsilon'X_{\epsilon\alpha_{i}}), \varphi(X_{\beta-\epsilon\alpha_{i}}))]=\varphi(X_{\beta})$. Here is where we use our assumptions on $\g$.
    \item[(ii)] For all $i\in [1, r]$, there exists $\epsilon\in \{1,-1\}$ such that $[\varphi(\epsilon X_{\alpha_i}), \varphi(X_{-\alpha_i})]=\varphi(t_i)$. 
\end{itemize}
Write the Lie bracket as 
\[
m\colon \g\otimes \g\to \g,\ x\otimes y\to [x, y].
\]
Note that $m$ is a $\g$-module homomorphism. We will show that $c^{V(\theta)}(t_i^{\ast}, b)\in R''_{\Z}$ for all $i\in [1, r]$ and $b\in \BBB(\theta)$. By (i) and (ii), there exist $x_1, x_2\in\g_{\Z,\neq 0}$ such that $m(x_1\otimes x_2)=\varphi(b)$. Thus,   
\begin{align*}
    c^{V(\theta)}(t_i^{\ast}, b)&=c^{\g}((\varphi^{-1})^{\ast}(t_i^{\ast}), \varphi(b))\\
    &=c^{\g}((\varphi^{-1})^{\ast}(t_i^{\ast}), m(x_1\otimes x_2))\\
    &=c^{\g^{\otimes 2}}(m^{\ast}((\varphi^{-1})^{\ast}(t_i^{\ast})), x_1\otimes x_2).
\end{align*}
Since $m(\h\otimes \h)=0$, we have $m^{\ast}((\varphi^{-1})^{\ast}(t_i^{\ast}))\in (\g\spcheck_{\Z})_{\neq 0}\otimes (\g\spcheck_{\Z})_{\neq 0}$, cf.~Remark \ref{r:A=O-classical} and \cite[Eq. (3.2)]{Oya}. Therefore, $c^{V(\theta)}(t_i^{\ast}, b)=c^{\g^{\otimes 2}}(m^{\ast}((\varphi^{-1})^{\ast}(t_i^{\ast})), x_1\otimes x_2)\in R''_{\Z}$ by \eqref{eq:product}. 

It remains to show that $c^{V(\theta)}(b^{\ast}, t_i)\in R''_{\Z}$ for all $b\in \BBB(\theta)$ and $i\in [1, r]$. By (ii), there exists $\epsilon\in \{1,-1\}$ such that $m(\varphi(\epsilon X_{\alpha_i})\otimes  \varphi(X_{-\alpha_i}))=\varphi(t_i)$. Hence, 
\begin{align*}
    c^{V(\theta)}(b^{\ast}, t_i)&=c^{\g}((\varphi^{-1})^{\ast}(b^{\ast}), \varphi(t_i))\\
    &=c^{\g}((\varphi^{-1})^{\ast}(b^{\ast}), m(\varphi(\epsilon X_{\alpha_i})\otimes  \varphi(X_{-\alpha_i})))\\
    &=c^{\g^{\otimes 2}}(m^{\ast}((\varphi^{-1})^{\ast}(b^{\ast})), \varphi(\epsilon X_{\alpha_i})\otimes  \varphi(X_{-\alpha_i})).
\end{align*}
Since $m^{\ast}((\varphi^{-1})^{\ast}(b^{\ast}))\in \g\spcheck_{\Z}\otimes \g\spcheck_{\Z}$, the previous argument shows that the rightmost expression lies in $R''_{\Z}$. This completes the proof of the theorem.
\end{proof}
This theorem implies the following refinement of \cite[Theorem 3.2]{Oya}. 
\begin{theorem}\label{thm:integral-coord-ring-classical}
When $G$ is not of type $F_4$ and $G_2$, the generalized minors generate $R[G]_{\Z}$ as a $\Z$-algebra. When $G$ is of type $G_2$, the generalized minors generate $R[G]_{\Z\left[\frac{1}{2}\right]}$ as a $\Z\left[\frac{1}{2}\right]$-algebra.
\end{theorem}
\begin{proof}
    When $\g$ is not of type $E_8$, the result follows from Theorem \ref{thm:integral-coord-ring} by the specialization at $q=1$. In the case of type $E_8$, the result follows from Theorems \ref{t:integral-gen} and \ref{t:adj-minor-classical}. 
\end{proof}

\appendix

\section{Comparison of definitions of quantized coordinate rings}\label{app:QCA}
In this appendix, we show the equivalence between Lusztig's definition of the quantized coordinate ring in \cite{Lus:intro,Lus-Zform} and the definition used in this paper. Indeed, the quantized coordinate ring in \cite{Lus:intro,Lus-Zform} is defined as a subspace of the dual of the modified quantized enveloping algebra, whereas our $R_q[G]$ is defined as a subspace of the dual of the usual quantized enveloping algebra $U_q(\g)$. This equivalence is well known to experts, but we recall it for the readers' convenience, since it is an important point in this paper. 

For $\mu, \mu'\in P$, set
\[
_{\mu}\dot{U}_{\mu'}\coloneqq U_q(\g)/\left(\sum\nolimits_{i\in [1, r]}(K_i-q^{(\mu, \alpha_i)})U_q(\g)+\sum\nolimits_{i\in [1,r]}U_q(\g)(K_i-q^{(\mu', \alpha_i)})\right).
\]

Let $\pi_{\mu, \mu'}\colon U_q(\g)\to {_{\mu}\dot{U}_{\mu'}}$ be the canonical projection. Note that $\pi_{\mu, \mu'}(U_q(\g)_{\gamma})=0$ unless $\mu-\mu'=\gamma$. Write $1_{\mu}\coloneqq \pi_{\mu, \mu}(1)$. Set
\[
\dot{U}=\bigoplus_{\mu, \mu'\in P} {_{\mu}\dot{U}_{\mu'}}.
\]
We can define a $\KK$-algebra structure on $\dot{U}$ by 
\[
\pi_{\mu_1, \mu_1'}(x)\pi_{\mu_2, \mu_2'}(y)=\delta_{\mu_1', \mu_2}\pi_{\mu_1, \mu_2'}(xy)
\]
for $\mu_1, \mu_1', \mu_2, \mu_2'\in P$, and $x\in U_q(\g)_{\mu_1-\mu_1'}, y\in U_q(\g)_{\mu_2-\mu_2'}$. We have 
\begin{equation}
1_{\mu}1_{\mu'}=\delta_{\mu, \mu'}1_{\mu},\quad {_{\mu}\dot{U}_{\mu'}}=1_{\mu}\dot{U}1_{\mu'}\label{eq:idem}
\end{equation}
for $\mu, \mu'\in P$. The $\KK$-algebra $\dot{U}$ is called the \emph{modified quantized enveloping algebra} \cite[Chapter 23]{Lus:intro}. 
The modified quantized enveloping algebra $\dot{U}$ has a $U_q(\g)$-bimodule structure given by 
\[
x\pi_{\mu, \mu'}(y)x'=\pi_{\mu+\wt x, \mu'-\wt x'}(xyx')
\]
for $\mu, \mu'\in P$ and homogeneous elements $x, x', y\in U_q(\g)$.

For $\mu, \mu'\in P$, define a $\KK$-linear map 
\[
\Delta_{\mu, \mu'}\colon  {_{\mu}\dot{U}_{\mu'}}\to \prod\nolimits_{\substack{\mu=\mu_1+\mu_2\\ \mu'=\mu_1'+\mu_2'}}{_{\mu_1}\dot{U}_{\mu_1'}}\otimes {_{\mu_2}\dot{U}_{\mu_2'}}
\]
by 
\[
\Delta_{\mu, \mu'}(\pi_{\mu, \mu'}(x))=\left( (\pi_{\mu_1, \mu_1'}\otimes\pi_{\mu_2, \mu_2'})(\Delta(x))\right)_{\mu_1, \mu_1',\mu_2, \mu_2'}
\]
for $x\in U_q(\g)$. The $\KK$-linear map 
\[
\Delta\coloneqq \bigoplus_{\mu, \mu'\in P}\Delta_{\mu, \mu'}\colon \dot{U}\to \prod_{\mu_1, \mu_1', \mu_2, \mu_2'\in P}{_{\mu_1}\dot{U}_{\mu_1'}}\otimes {_{\mu_2}\dot{U}_{\mu_2'}}
\]
is called the comultiplication of $\dot{U}$. 

A left $\dot{U}$-module $V$ is said to be \emph{unital} if it satisfies the following conditions:
\begin{itemize}
\item For any $v\in V$,  $1_{\mu}.v=0$ except for finitely many $\mu\in P$.
\item For any $v\in V$, $\sum_{\mu\in P}1_{\mu}.v=v$.
\end{itemize}
The tensor product $V\otimes W$ of unital $\dot{U}$-modules $V, W$ is again a unital $\dot{U}$-module via $\Delta$. Note that the unital $\dot{U}$-module $V$ has a decomposition $V=\bigoplus_{\mu\in P}1_{\mu}.V$ by \eqref{eq:idem}. For a unital $\dot{U}$-module $V$, the matrix coefficient associated to a pair of vectors $v \in V$ and $\xi \in V^*$ is denoted by  
\begin{equation*} 
\dot{c}^V(\xi, v) \in \dot{U}^*, \quad 
\mbox{where} \quad \dot{c}^V(\xi, v)(x) := \langle \xi, x. v\rangle,  \quad \forall x \in \dot{U}. 
\end{equation*}

A unital $\dot{U}$-module $V$ can be regarded as a type one $U_q(\g)$-module by 
\[
x.v\coloneqq \sum_{\mu \in P}x1_{\mu}.v
\]
for $x\in U_q(\g)$ and $v\in V$. Then, $V=\bigoplus_{\mu\in P} V_{\mu}$ as a $U_q(\g)$-module, and $V_{\mu}=1_{\mu}.V$. Conversely, a type one $U_q(\g)$-module $V$ has a structure of unital $\dot{U}$-module such that, for $\mu, \mu'\in P, x\in U_q(\g)_{\mu-\mu'}$ and $v\in V_{\nu}$,  
\[
\pi_{\mu, \mu'}(x).v\coloneqq \delta_{\mu', \nu}x.v. 
\]
Then these correspondences are mutually inverse. Hence we will identify the unital $\dot{U}$-modules with the type one $U_q(\g)$-modules. 

Let $\dot{\BBB}$ be the canonical basis of $\dot{U}$ (see \cite[Chapter 25]{Lus:intro} for the definition). The following are important properties of $\dot{\BBB}$: 
\begin{itemize}
    \item[(B1)] $\dot{\BBB}$ forms an $\AAA$-basis of $\dot{U}_{\AAA}$, where $\dot{U}_{\AAA}$ is the $\AAA$-subalgebra of $\dot{U}$ generated by $\{X_i^{+(k)} 1_{\mu}, X_i^{-(k)} 1_{\mu}\mid i\in [1,r], k\in \Z_{\geq 0}, \mu\in P\}$. Moreover, for $\mu, \mu'\in P$, $\dot{\BBB}\cap 1_{\mu}\dot{U}_{\AAA}1_{\mu'}$ is an $\AAA$-basis of $1_{\mu}\dot{U}_{\AAA}1_{\mu'}$ (\cite[Theorem 25.2.1]{Lus:intro}). 
    \item[(B2)] For $\mu\in P^+$ and $\bm{a}=(a_i)_{i\in [1, r]}, \bm{a}'=(a'_i)_{i\in [1, r]}\in \Z_{\geq 0}^{[1, r]}$, set
\[
\mathcal{P}(\mu; \bm{a}, \bm{a}')\coloneqq \sum_{i\in [1, r]}\sum_{k>a_i}\dot{U}X_i^{+(k)}1_{\mu}+\sum_{i\in [1, r]}\sum_{l>a'_i}\dot{U}X_i^{-(l)}1_{\mu}.
\]  
Then $\dot{\BBB}\cap \mathcal{P}(\mu; \bm{a}, \bm{a}')$ forms a $\KK$-basis of $\mathcal{P}(\mu; \bm{a}, \bm{a}')$ (\cite[Theorem 25.2.5]{Lus:intro}).
\end{itemize}

For $\dot{b}\in \dot{\BBB}$, define $\dot{b}^{\ast}\in \dot{U}^{\ast}$ by 
\[
\langle \dot{b}^{\ast}, \dot{b}'\rangle=\delta_{\dot{b}, \dot{b}'}\text{ for }\dot{b}'\in \dot{\BBB}. 
\]
In \cite[Section 29.5]{Lus:intro},\cite[Section 3.1]{Lus-Zform}, the \emph{quantum coordinate algebra} $\mathbf{O}$ and its integral form $\mathbf{O}_{\AAA}$ are defined as 
\[
\mathbf{O}\coloneqq \sum_{\dot{b}\in \dot{\BBB}}\KK \dot{b}^{\ast},\qquad 
\mathbf{O}_{\AAA}\coloneqq \sum_{\dot{b}\in \dot{\BBB}}\AAA \dot{b}^{\ast}
\]
in $\dot{U}^{\ast}$. They are the Hopf algebras over $\KK$ and $\AAA$, respectively, by 
\[
\dot{b}_1^{\ast}\dot{b}_2^{\ast}=\sum_{\dot{b}\in \dot{\BBB}}\widehat{m}^{\dot{b}_1, \dot{b}_2}_{\dot{b}}\dot{b}^{\ast},\quad \Delta(\dot{b}^{\ast})=\sum_{\dot{b}_1, \dot{b}_2\in \dot{\BBB}}m_{\dot{b}_1, \dot{b}_2}^{\dot{b}}\dot{b}_1^{\ast}\otimes \dot{b}_2^{\ast},
\]
where $\widehat{m}^{\dot{b}_1, \dot{b}_2}_{\dot{b}}$ and $m_{\dot{b}_1, \dot{b}_2}^{\dot{b}}$ are defined by
\begin{align*}
    \Delta(\dot{b})=\sum_{\dot{b}_1, \dot{b}_2\in \dot{\BBB}}\widehat{m}^{\dot{b}_1, \dot{b}_2}_{\dot{b}}\dot{b}_1\otimes \dot{b}_2,\quad \dot{b}_1\dot{b}_2=\sum_{\dot{b}\in \dot{\BBB}}m_{\dot{b}_1, \dot{b}_2}^{\dot{b}}\dot{b}.
\end{align*}
Note that $\Delta(\dot{b})\in \prod_{\mu_1, \mu_1', \mu_2, \mu_2'\in P}{_{\mu_1}\dot{U}_{\mu_1'}}\otimes {_{\mu_2}\dot{U}_{\mu_2'}}$. The main statement in this appendix is the following. 
\begin{prop}\label{prop:comparison}
    Define a $\KK$-linear map $\Upsilon\colon R_q[G]\to \mathbf{O}$ by 
    \[
    c^{V}(\xi, v)\mapsto \dot{c}^{V}(\xi, v)
    \]
    for a finite dimensional type one $U_q(\g)$-module $V$ and $v\in V, \xi\in V^{\ast}$. Then $\Upsilon$ is an isomorphism of Hopf algebras and it restricts to an isomorphism $R_q[G]_{\AAA}\xrightarrow{\sim} \mathbf{O}_{\AAA}$.  
\end{prop}
\begin{proof}
We first show that $\dot{c}^{V}(\xi, v)\in \mathbf{O}$ for any finite dimensional type one $U_q(\g)$-module $V$ and $v\in V, \xi\in V^{\ast}$. It suffices to show that $\langle \xi, \dot{b}.v\rangle=0$ except for finitely many $\dot{b}\in \dot{\BBB}$ when $v\in V_{\mu}$ for some $\mu\in P$. 

By (B1), $\dot{\BBB}=\sqcup_{\mu'\in P}(\dot{\BBB}\cap\dot{U}1_{\mu'})$. 
If $\mu'\neq \mu$, $\dot{b}.v=0$ for $\dot{b}\in \dot{\BBB}\cap\dot{U}1_{\mu'}$. Moreover, since $V$ is finite dimensional, there exist $\bm{a}=(a_i)_{i\in [1, r]}, \bm{a}'=(a'_i)_{i\in [1, r]}\in \Z_{\geq 0}^{[1, r]}$ such that 
\begin{itemize}
    \item $X_i^{+(k)}.v=0=X_i^{-(l)}.v$ for $k>a_i, l>a'_i$, and 
    \item $\langle \alpha_i\spcheck, \mu\rangle=-a_i+a'_i$ for all $i\in [1, r]$. 
\end{itemize}
Then, $\mathcal{P}(\mu; \bm{a}, \bm{a}').v=0$. Hence, by (B2), it suffices to show that $(\dot{\BBB}\cap\dot{U}1_{\mu}) \setminus (\dot{\BBB}\cap \mathcal{P}(\mu; \bm{a}, \bm{a}'))$ is a finite set. Write $\lambda=\sum_{i\in [1, r]}a_i\varpi_i, \lambda'=\sum_{i\in [1, r]}a'_i\varpi_i\in P^+$. Note that $\mu=-\lambda+\lambda'$. Then, by \cite[Proposition 23.3.6]{Lus:intro}, the $\KK$-linear map 
\[
\dot{U}1_{\mu}\to V_q(-w_0\lambda)\otimes V_q(\lambda'),\ x\mapsto x.(v_{-\lambda}\otimes v_{\lambda'})
\]
is surjective and its kernel is equal to $\mathcal{P}(\mu; \bm{a}, \bm{a}')$. Therefore, by (B1) and (B2), $(\dot{\BBB}\cap\dot{U}1_{\mu}) \setminus (\dot{\BBB}\cap \mathcal{P}(\mu; \bm{a}, \bm{a}'))$ induces a $\KK$-basis of the finite dimensional vector space $V_q(-w_0\lambda)\otimes V_q(\lambda')$. Therefore, it is a finite set.  

The injectivity of $\Upsilon$ is clear from the correspondence between the unital $\dot{U}$-modules and the type one $U_q(\g)$-modules. We shall prove that $\Upsilon$ is surjective. We need to show that the image of $\Upsilon$ contains $\dot{b}^{\ast}$ for all $\dot{b}\in \dot{\BBB}$. Let $\dot{b}\in \dot{\BBB}\cap\dot{U}1_{\mu}$. Then, by \cite[Theorem 25.2.1]{Lus:intro}, there exist $\lambda, \lambda'\in P_+$ with $-\lambda+\lambda'=\mu$ such that $\dot{b}_{-\lambda, \lambda'}\coloneqq \dot{b}. (v_{-\lambda}\otimes v_{\lambda'})\neq 0$. Moreover,  
    \[
    \dot{\BBB}_{-\lambda, \lambda'}\coloneqq \{\dot{b}'.(v_{-\lambda}\otimes v_{\lambda'})\mid \dot{b}'\in \dot{\BBB}\text{ such that }\dot{b}'.(v_{-\lambda}\otimes v_{\lambda'})\neq 0\}
    \]
    forms a $\KK$-basis of $V_q(-w_0\lambda)\otimes V_q(\lambda')$. Take its dual basis $\dot{\BBB}_{-\lambda, \lambda'}^{\ast}$ of $(V_q(-w_0\lambda)\otimes V_q(\lambda'))^{\ast}$, and denote by 
   $\dot{b}_{-\lambda, \lambda'}^{\ast}$ the unique element of $\dot{\BBB}_{-\lambda, \lambda'}^{\ast}$ satisfying $\langle\dot{b}_{-\lambda, \lambda'}^{\ast}, \dot{b}_{-\lambda, \lambda'}\rangle=1$. Then, 
   \[
   \dot{c}^{V_q(-w_0\lambda)\otimes V_q(\lambda')}(\dot{b}_{-\lambda, \lambda'}^{\ast}, v_{-\lambda}\otimes v_{\lambda'})=\dot{b}^{\ast}.
   \]
   Therefore, $\dot{b}^{\ast}$ belongs to the image of $\Upsilon$. 

   Next, we show that $\Upsilon$ is a homomorphism of Hopf algebras. Let $\dot{b}_1\in \dot{\BBB}\cap {_{\mu_1}\dot{U}_{\mu'_1}}, \dot{b}_2\in \dot{\BBB}\cap {_{\mu_2}\dot{U}_{\mu'_2}}$. Then, by the above argument, there exist finite dimensional type one $U_q(\g)$-modules $V_i$ and $v_i\in V_{\mu'_i}, \xi_i\in (V^{\ast})_{-\mu_i}$  such that 
   \[
   c^{V_i}(\xi_i, v_i)=\Upsilon^{-1}(\dot{b}_i^{\ast})
   \]
   for $i=1, 2$. For $\dot{b}'\in \dot{\BBB}\cap {_{\nu}\dot{U}_{\nu'}}$, we take $\widetilde{\dot{b}'}\in U_q(\g)_{\nu-\nu'}$ such that $\pi_{\nu, \nu'}(\widetilde{\dot{b}'})=\dot{b}'$. Then, by definition, for $\dot{b}\in \dot{\BBB}\cap {_{\mu}\dot{U}_{\mu'}}$, 
\begin{align*}
    &\langle\Upsilon(c^{V_1}(\xi_1, v_1)c^{V_2}(\xi_2, v_2)), \dot{b}\rangle\\
    &
    =\langle \xi_2\otimes \xi_1, \dot{b}.(v_1\otimes v_2)\rangle\\
    &
    =\delta_{\mu', \mu'_1+\mu'_2}\langle \xi_2\otimes \xi_1, \widetilde{\dot{b}}.(v_1\otimes v_2)\rangle\\
    &
    =\delta_{\mu', \mu'_1+\mu'_2}\langle \xi_2\otimes \xi_1, (\pi_{\mu_1, \mu_1'}\otimes\pi_{\mu_2, \mu_2'})(\Delta(\widetilde{\dot{b}})).(v_1\otimes v_2)\rangle\\
    &
    =\sum_{\substack{\dot{b}'_1 \in \dot{\BBB}\cap {_{\mu_1}\dot{U}_{\mu'_1}}\\ \dot{b}'_2\in \dot{\BBB}\cap {_{\mu_2}\dot{U}_{\mu'_2}}}}\widehat{m}^{\dot{b}'_1, \dot{b}'_2}_{\dot{b}}
    \langle \xi_2\otimes \xi_1, (\dot{b}'_1. v_1)\otimes (\dot{b}'_2.v_2)\rangle=\widehat{m}^{\dot{b}_1, \dot{b}_2}_{\dot{b}}.\\
\end{align*}
   Therefore, 
   \begin{align*}
       \Upsilon(c^{V_1}(\xi_1, v_1)c^{V_2}(\xi_2, v_2))=\sum_{\dot{b}\in \dot{\BBB}}\widehat{m}^{\dot{b}_1, \dot{b}_2}_{\dot{b}}\dot{b}^{\ast}=\dot{b}_1^{\ast}\dot{b}_2^{\ast}=\Upsilon(c^{V_1}(\xi_1, v_1))\Upsilon(c^{V_2}(\xi_2, v_2)).
   \end{align*}
   Hence $\Upsilon$ is an algebra homomorphism. 
   
   Let $\dot{b}\in \dot{\BBB}\cap {_{\mu}\dot{U}_{\mu'}}$, and take a finite dimensional type one $U_q(\g)$-module $V$ and $v\in V_{\mu'}, \xi\in (V^{\ast})_{-\mu}$  such that 
   \begin{align}
   c^{V}(\xi, v)=\Upsilon^{-1}(\dot{b}^{\ast}).\label{eq:Upsilon_inv}       
   \end{align}
   Then, for $\dot{b}_1\in \dot{\BBB}\cap {_{\mu_1}\dot{U}_{\mu'_1}}$ and $\dot{b}_2\in \dot{\BBB}\cap {_{\mu_2}\dot{U}_{\mu'_2}}$, 
   \begin{align*}
   \langle (\Upsilon\otimes \Upsilon)(\Delta(c^{V}(\xi, v))), \dot{b}_1\otimes \dot{b}_2\rangle
   &=\delta_{\mu, \mu_1}\delta_{\mu'_1, \mu_2}\delta_{\mu'_2, \mu'}\langle \xi, \widetilde{\dot{b}}_1\widetilde{\dot{b}}_2.v\rangle\\
   &=\delta_{\mu, \mu_1}\delta_{\mu'_1, \mu_2}\langle \xi, \pi_{\mu_1, \mu'_2}(\widetilde{\dot{b}}_1\widetilde{\dot{b}}_2).v\rangle\\
   &=\langle \xi, \dot{b}_1\dot{b}_2.v\rangle\\
   &=\sum_{\dot{b}\in \dot{\BBB}}m_{\dot{b}_1, \dot{b}_2}^{\dot{b}'}\langle \xi, \dot{b}'.v\rangle
   =m_{\dot{b}_1, \dot{b}_2}^{\dot{b}}.
   \end{align*}
      Therefore, 
   \begin{align*}
       (\Upsilon\otimes \Upsilon)(\Delta(c^{V}(\xi, v)))=\sum_{\dot{b}_1, \dot{b}_2\in \dot{\BBB}}m_{\dot{b}_1, \dot{b}_2}^{\dot{b}}\dot{b}_1^{\ast}\otimes \dot{b}_2^{\ast}=\Delta (\dot{b}^{\ast})=\Delta(\Upsilon(c^{V}(\xi, v))).
   \end{align*}
   Hence $\Upsilon$ is a coalgebra homomorphism.

Finally, we prove that $\Upsilon(R_q[G]_{\AAA})=\mathbf{O}_{\AAA}$. Let $\dot{b}\in \dot{\BBB}\cap {_{\mu}\dot{U}_{\mu'}}$. By (B1), $\widetilde{\dot{b}}$ can be taken from $U_q(\g)_\AAA$. By Lemma \ref{lem:weight_decomp} below, $R_q[G]_{\AAA}$ is spanned  over $\AAA$ by elements of the form $c^V(\xi, v)$ for some finite dimensional type one $U_q(\g)$-module $V$ and its weight vectors $v\in V, \xi\in V^{\ast}$. Then, 
\[
\langle \Upsilon(c^V(\xi, v)), \dot{b}\rangle =\delta_{\mu', \wt v}\langle \xi, \widetilde{\dot{b}}.v\rangle=\delta_{\mu', \wt v}\langle c^V(\xi, v), \widetilde{\dot{b}}\rangle\in \AAA.
\]
Therefore, $\Upsilon(R_q[G]_{\AAA})\subset \mathbf{O}_{\AAA}$. 

Conversely, let $\dot{b}\in \dot{\BBB}\cap {_{\mu}\dot{U}_{\mu'}}$, and take $V, v, \xi$ as in \eqref{eq:Upsilon_inv}. For $x\in U_q(\g)_\AAA$, we can write 
\[
\pi_{\mu, \mu'}(x)=1_{\mu}x1_{\mu'}=\sum_{\dot{b}'\in \dot{\BBB}}c_{\dot{b}'}\dot{b}'\text{ for }c_{\dot{b}'}\in \AAA 
\]
by (B1) and the relation $K_i1_{\nu}=q^{(\nu,\alpha_i)}$ for $i\in [1, r], \nu\in P$. Hence 
\[
\langle c^V(\xi, v), x\rangle= \langle \dot{c}^V(\xi, v), \pi_{\mu, \mu'}(x)\rangle=\sum_{\dot{b}'\in \dot{\BBB}}c_{\dot{b}'}\langle \dot{b}^{\ast}, \dot{b}'\rangle=c_{\dot{b}}\in \AAA.
\]
Hence $c^V(\xi, v)\in R_q[G]_{\AAA}$. Therefore, $\Upsilon(R_q[G]_{\AAA})\supset \mathbf{O}_{\AAA}$, which completes the proof of the proposition. 
\end{proof}
In the proof of Proposition \ref{prop:comparison}, we used the following general lemma. 
\begin{lem}\label{lem:weight_decomp}
    Let $V$ be a type one $U_q(\g)$-module $V$ and $V_{\AAA}\subset V$ be its $U_q(\g)_\AAA$-submodule. Then $V_{\AAA}=\bigoplus_{\mu\in P}(V_{\AAA}\cap V_{\mu})$.  
\end{lem}
\begin{proof}
    Let $v\in V_{\AAA}$ and write it as $v=\sum_{\mu\in P}v_\mu$ with $v_\mu\in V_{\mu}$. It suffices to show that $v_\mu\in  V_{\AAA}$ for all $\mu \in P$. Set  
    \[
    \{\mu\in P\mid v_{\mu}\neq 0\}=\{\mu_1,\dots, \mu_l\}.
    \]
    By the relabeling, we may assume that, for $k=2,\dots, l$, there exists $i_k\in [1, r]$ such that 
    \[
    \langle \alpha_i\spcheck, \mu_1\rangle=\langle \alpha_i\spcheck, \mu_k\rangle\text{ for } i=1,\dots, i_k-1,\text{ and }
    \langle \alpha_{i_k}\spcheck, \mu_1\rangle>\langle \alpha_{i_k}\spcheck, \mu_k\rangle.
    \]
    It is known that $U_q(\g)_\AAA$ contains the elements 
    \[
    \left[\begin{array}{c}
				K_i; a\\
				n
			\end{array}\right]\coloneqq \prod_{s=0}^{n-1}\frac{K_iq_i^{a-s}-K_i^{-1}q_i^{-a+s}}{q_i^{s+1}-q_i^{-s-1}}
    \]
    for $i\in [1, r]$, $a\in \Z$ and $n\in \Z_{\geq 0}$ (see, e.g., \cite[Section 11.1]{jantzen1996lectures}). For $n_1\coloneqq \min\{\langle \alpha_1\spcheck, \mu_k\rangle\mid k=1,\dots, \ell\}$, we have
    \[
    \left[\begin{array}{c}
				K_1; -n_1\\
				\langle \alpha_1\spcheck, \mu_1\rangle-n_1
			\end{array}\right].v_{\mu_k}
            =\begin{cases}
            v_{\mu_k}&\text{if }\langle \alpha_1\spcheck, \mu_k\rangle=\langle \alpha_1\spcheck, \mu_1\rangle\\
			    0&\text{if }\langle \alpha_1\spcheck, \mu_k\rangle\neq \langle \alpha_1\spcheck, \mu_1\rangle.
			\end{cases}
    \]
Therefore, 
\[
    \left[\begin{array}{c}
				K_1; -n_1\\
				\langle \alpha_1\spcheck, \mu_1\rangle-n_1
			\end{array}\right].v=\sum_{k; \langle \alpha_1\spcheck, \mu_k\rangle=\langle \alpha_1\spcheck, \mu_1\rangle}v_{\mu_k}\in V_{\AAA}.
\]
By repeating a similar argument for $i=2,\dots, r$, we can obtain an element $\widetilde{K}\in U_q(\g)_\AAA$ such that 
\[
\widetilde{K}.v=v_{\mu_1}. 
\]
Thus, $v_{\mu_1}\in V_{\AAA}$, and hence, $\sum_{k=2,\dots, l}v_{\mu_k}\in V_{\AAA}$. By applying this argument iteratively to $\sum_{k=2,\dots, l}v_{\mu_k}\in V_{\AAA}$, we conclude that $v_{\mu_k}\in V_{\AAA}$ for $k=1,\dots, l$. 
\end{proof}
\bibliographystyle{amsalphaURL2} 
\bibliography{referenceEprint}        
\index{Bibliography@\emph{Bibliography}}

\end{document}